\documentclass[11pt]{article}
\usepackage{amsfonts}
\usepackage{amsmath,amssymb,amsthm,amsfonts,enumerate}
\usepackage{graphicx}
\usepackage{float}
\usepackage[scriptsize,nooneline]{subfigure}
\usepackage{indentfirst}
\usepackage{cite}
\usepackage{color}
\usepackage{mathrsfs}
\usepackage{epstopdf}
\usepackage{setspace}
\usepackage[labelfont=footnotesize,font=footnotesize]{caption}
\usepackage[colorlinks, citecolor=red]{hyperref}
\allowdisplaybreaks[4]
\newtheorem{theorem}{Theorem}[section]
\newtheorem{lemma}{Lemma}[section]
\newtheorem{definition}{Definition}[section]
\newtheorem{proposition}{Proposition}[section]
\newtheorem{remark}{Remark}[section]
\newtheorem{corollary}{Corollary}[section]
\def\geq{\geqslant}\def\leq{\leqslant}
\begin{document}
\title{\bf Hardy Spaces Associated with Some Anisotropic Mixed-Norm Herz Spaces and Their Applications}
\author{Yichun Zhao, Jiang Zhou \thanks{Corresponding author. The research was supported by Natural Science Foundation of China (Grant No. 12061069).}  \\[.5cm]}
\date{}
\maketitle
{\bf Abstract:}{ In this paper, we introduce anisotropic mixed-norm Herz spaces $\dot K_{\vec{q}, \vec{a}}^{\alpha, p}(\mathbb R^n)$ and $K_{\vec{q}, \vec{a}}^{\alpha, p}(\mathbb R^n)$ and investigate some basic properties of those spaces. Furthermore, establishing the Rubio de Francia extrapolation theory, which resolves the boundedness problems of Calder\'on-Zygmund operators and fractional integral operator and their commutators, on the space $\dot K_{\vec{q}, \vec{a}}^{\alpha, p}(\mathbb R^n)$ and the space $K_{\vec{q}, \vec{a}}^{\alpha, p}(\mathbb R^n)$. Especially, the Littlewood-Paley characterizations of anisotropic mixed-norm Herz spaces also are gained. As the generalization of anisotropic mixed-norm Herz spaces, we introduce anisotropic mixed-norm Herz-Hardy spaces $H\dot K_{\vec{q}, \vec{a}}^{\alpha, p}(\mathbb R^n)$ and $HK_{\vec{q}, \vec{a}}^{\alpha, p}(\mathbb R^n)$, on which atomic decomposition and molecular decomposition are obtained. Moreover, we gain the boundedness of classical Calder\'on-Zygmund operators.
}\par 
{\bf Key Words: }Anisotropic, mixed-norm, Herz space, Herz-Hardy spaces, Hardy-Littlewood maximal operator.

{\bf Mathematics Subject Classification(2010):}  42B35; 42B25; 42B20.

\baselineskip 15pt

\section{Introduction}
The study of Herz spaces originated from the work of Beurling \cite{beurling1964construction}, which analyze some convolution algebras in 1964. Herz \cite{herz1968lipschitz} later systematically studied Herz spaces $K_q \left(\mathbb R^n\right)$ to understand the Fourier series and Fourier transform. After that, Bernstein and Sawyer \cite{baernstein1985embedding} generalized Herz spaces $A_q(\mathbb R^n)$ and explored some applications associated with the classic Hardy spaces. Indeed, another equivalence norm of Beurling algebra was constructed by Feichtinger \cite{feichtinger1984elementary} in 1987 to study the Wiener third Tauberian theorem on $\mathbb R^n$. Since then, Herz space has been crucial in harmonic analysis and partial differential equations, involving the properties of Herz-type spaces, the regularity of solutions to elliptic equations, and the global stability for fractional Navier-Stokes equations\cite{lu2008herz,scapellato2019regularity,chen2018global}. Herz-Hardy spaces as the appropriate substitute for Herz space, in 1989, Chen and Lau \cite{chen1989some} introduced the Hardy space associated with the Beurling algebras on the real line. Later, Grac\'ia-Cuerva and Herrero \cite{garcia1994theory} established various maximal functions, atomic and Littlewood-Paley function characterizations of these Herz-Hardy spaces. At the same time, Lu and Yang \cite{lu1995local} studied a series of problems on Herz-type Hardy spaces, they obtained atomic decomposition, molecular decomposition, and Littlewood-Paley characterization on Herz-Hardy spaces with more wide index range. These results have attracted researchers' attention to the topic of Herz-type spaces, we can see \cite{lu2008herz,zhao2018anisotropic,drihem2013embeddings,2008New,zhao2022characterizations}.
\par 
 Mixed-norm Lebesgue spaces $L^{\vec{p}}\left(\mathbb R^n\right)$, as natural extensions of classical Lebesgue spaces $L^{{p}}\left(\mathbb R^n\right)$, have attracted widespread attention. The theory of mixed-norm function spaces can be traced back to the work of Benedek and Panzone \cite{benedek1961space}, in which one proved that the space $L^{\vec{p}}(\mathbb{R}^n)$ also possesses some basic properties similar to the space $L^{{p}}\left(\mathbb R^n\right)$, such as completeness, H\"{o}lder's inequality, Minkowski's inequality, and so on. These properties provide the possibility to solve a series of subsequence problems. On the other hand, since the study of partial differential equations (for example the heat equation and the wave equation) always involves both space and time variables, mixed-norm spaces possess better structures than classical spaces in the time-space estimates for PDEs. For these reasons, many researchers renewed their interest in mixed-norm Lebesgue spaces and extended them to other mixed-norm function spaces.  For instance, the real variable characterizations and the atom characterization of mixed-norm Hardy spaces were established in \cite{cleanthous2017anisotropic,huang2019atomic}. In 2021, mixed-norm Herz spaces are defined by Wei \cite{wei2021characterization,wei2020Herz}, who also gained boundedness of some classical operators via Rubio de Francia extrapolation and obtained a characterization of mixed-norm $C\dot MO^{\vec{q}} (\mathbb R^n)$ via commutators of Hardy operators. In 2022, Zhao et al.\cite{zhao2022characterizations} introduced the mixed-norm Herz-Hardy spaces and established a series of real variable characterizations.
\par 
Additionally, the outstanding work on parabolic Hardy spaces done by Calder${\rm \acute o}$n and Torchinsky \cite{calderon1977atomic} led us to consider extending the underlying spaces from Euclidean spaces to some more general spaces. In particular, anisotropic Hardy spaces $H_A^p (\mathbb R^n)$ with an expansive matrix $A$ on $\mathbb R^n$, which can be seen as a generalization of parabolic Hardy spaces, introduced by Bownik \cite{bownik2003anisotropic}. Later on, more progress on anisotropic function spaces can refer the reader to \cite{bownik2005atomic,hochmuth2002wavelet,zhao2018anisotropic}. In these results, Zhao and Zhou \cite{zhao2018anisotropic} created anisotropic Herz-type Hardy spaces with variable exponent, given atomic and molecular decomposition characterizations of these spaces. Recently, Cleanthous et al. \cite{cleanthous2017anisotropic} established the anisotropic mixed-norm Hardy space $H_{\vec{a}}^{\vec{p}}(\mathbb R^n)$ with $\vec{a} \in [1, \infty)^n$ and $\vec{p}\in (0,\infty)^n$, in which, authors also established various maximal operators characterization. In 2019, Huang et al. \cite{huang2019atomic} obtain Littlewood-Paley and atomic characterizations of anisotropic mixed-norm Hardy spaces $H_{\vec{a}}^{\vec{p}}(\mathbb R^n)$. 
\par 
Recall that mixed-norm Herz spaces were first introduced by Wei \cite{wei2021characterization}, who investigated the boundedness of classical operators (Calder\'on-Zygmund, fractional integral operator, parametric Marcinkiewicz integrals and so on) via Rubio de Francia extrapolation theory. This extrapolation theory usually depends on the weight conditions and boundedness of Hardy-Littlewood maximal operators. As a generalization of mixed-norm Herz space on classical Euclidean norm spaces, we will introduce anisotropic mixed-norm Herz spaces, moreover, will establish and character anisotropic mixed-norm Herz-Hardy spaces. This article is organized as follows.
\par 
In section 2, we first introduce anisotropic mixed-norm Herz spaces and prove some of their basic properties. Then, the decomposition of anisotropic mixed-norm Herz space is gained. Moreover, we show the boundedness of Hardy-Littlewood maximal operators, which give conditions for Rubio de Francia extrapolation on anisotropic mixed-Herz spaces. At the end of this section, we establish the extrapolation theory on anisotropic mixed-norm Herz spaces, and, the boundedness of various classical operators can be obtained via the extrapolation on anisotropic mixed-norm Herz spaces.
\par 
In section 3, we define anisotropic mixed-norm Herz-Hardy spaces and establish the atomic and molecular decomposition of anisotropic Herz-Hardy spaces. As an application of atomic decomposition, we prove the boundedness of classical Calder\'on-Zygmund operators.
\par 
Throughout this paper, we use the following notations. The letter $\vec{q}$ will denote $n$-tuples of the numbers in $(0,\infty]$ $(n \geq 1)$, $\vec{q}=\left(q_1, q_2, \dots, q_n\right)$. By definition, the inequality $0<\vec{q}<\infty$ means that $0<q_i<\infty$ for all $i$. For $\vec{q}=\left(q_1, q_2, \dots, q_n\right)$, write\\
$$\frac{1}{\vec{q}}=\left(\frac{1}{q_1}, \frac{1}{q_2},\dots, \frac{1}{q_n}\right),\quad \vec{q}^{\prime}=\left(q_1^{\prime},q_2^{\prime},\dots, q_n^{\prime}\right),$$
where $q_i^{\prime}=q_i/(q_i-1)$ is conjugate exponent of $q_i$ . $|B|$ denotes the measure of the ball $B$, $\chi_{_E}$ is the characteristic function of a set $E$. $A\sim B$ means that $A\lesssim B$ and $B\lesssim A$, $\lfloor a \rfloor$ denotes take the integer number for $a$. 
  
\section{Anisotropic Mixed-norm Herz Spaces}
\subsection{Definitions and Fundamental Properties}
The goals of this section are to recall some definitions of anisotropic homogeneous quasi-norms, state some of their key conclusions that will be applied to this article, define anisotropic mixed-norm Herz spaces, and investigate some of their fundamental characteristics.
\par 
We begin with recalling the definition of anisotropic homogeneous quasi-norms in \cite{fabes1966singular} as follows. 
\begin{definition}{\label{anisotropic df}}
	Let $b:=\left(b_{1}, \dots, b_{n}\right)$, $x:=\left(x_{1}, \dots, x_{n}\right) \in \mathbb{R}^{n}$ and $t \in[0, \infty)$, $t^{b} x:=\left(t^{b_{1}} x_{1}, \dots, t^{b_{n}} x_{n}\right)$, fixed a vector $\vec{a}:=\left(a_{1}, \dots, a_{n}\right) \in[1, \infty)^{n}$. The anisotropic homogeneous quasi-norm $|\cdot| \vec{a}$, associated with $\vec{a}$, is a non-negative measurable function on $\mathbb{R}^{n}$ defined by setting $|\vec{0}_{n}|_{\vec{a}}:=0$ and, for any $x \in \mathbb{R}^{n} \backslash\{\vec{0}_{n}\},|x|_{\vec{a}}:=t_{0}$, where $t_{0}$ is the unique positive number such that $\left|t_{0}^{-\vec{a}} x\right|=1$, namely,
$$
\frac{x_{1}^{2}}{t_{0}^{2 a_{1}}}+\cdots+\frac{x_{n}^{2}}{t_{0}^{2 a_{n}}}=1 .
$$
\end{definition}
We also need the following notion of the anisotropic bracket and the homogeneous dimension from \cite{stein1978problems}, which plays an important role in the study of anisotropic function spaces.
\begin{definition}
	Let $\vec{a}:=\left(a_{1}, \ldots, a_{n}\right) \in[1, \infty)^{n}$. The anisotropic bracket, associated with $\vec{a}$, is defined by setting, for any $x \in \mathbb{R}^{n}$,
$$
\langle x\rangle_{\vec{a}}:=|(1, x)|_{(1, \vec{a})} .
$$
\end{definition}
In most situations, we consider issues usually on Euclidean spaces. The next lemma compares the anisotropic homogeneous quasi-norm with the Euclidean norm, and it is readily demonstrated by using Definition \ref{anisotropic df}, here, the details are omitted.
\begin{lemma}
	Let $\vec{a} \in[1, \infty)^{n}$ and $x \in \mathbb{R}^{n}$. Then\\
\rm{(i)} ~$|x|_{\vec{a}}>1$ if and only if $|x|>1$;\\
\rm{(ii)} $|x|_{\vec{a}}<1$ if and only if $|x|<1$.
\end{lemma}
Now let us recall some basic properties of $|\cdot|_{\vec{a}}$, see \cite{cleanthous2017anisotropic} for more details. For any $\vec{a}:=\left(a_{1}, \ldots, a_{n}\right) \in[1, \infty)^{n}$, homogeneous dimension $v$ and $a_+$, $a_{-}$ are defined as follow, 
\begin{equation}
	v:=|\vec{a}|=a_1+a_2+\dots+a_n, \quad a_{-}:=\min \left\{a_{1},\dots, a_{n}\right\}, \quad a_{+}:=\max \left\{a_{1}, \dots, a_{n}\right\}.\label{v}
\end{equation}
\begin{lemma}
Let $\vec{a}:=\left(a_{1}, \dots, a_{n}\right) \in[1, \infty)^{n}, t \in[0, \infty)$. Then, for any $x, y \in \mathbb{R}^{n}$,
\begin{enumerate}\label{norm proper}
\item[\rm{(i)}] $\left|t^{\vec{a}} x\right|_{\vec{a}}=t|x|_{\vec{a}}$;	
\item[\rm{(ii)}] $|x+y|_{\vec{a}} \leq|x|_{\vec{a}}+|y|_{\vec{a}}$;
\item[\rm{(iii)}] $\max \{\left|x_{1}\right|^{1 / a_{1}}, \dots,\left|x_{n}\right|^{1 / a_{n}}\} \leq|x|_{\vec{a}} \leq \sum_{i=1}^{n}\left|x_{i}\right|^{1 / a_{i}}$;
\item[\rm{(iv)}]  when $|x| \geq 1,|x|^{1 / a_{+}} \leq|x|_{\vec{a}} \leq|x|^{1 / a_{-}}$;
\item[\rm{(v)}] when $|x|<1,|x|^{1 / a_{-}} \leq|x|_{\vec{a}} \leq|x|^{1 / a_{+}}$;
\item[\rm{(vi)}] $\left(\frac{1}{2}\right)^{a_{-}}\left(1+|x|_{\vec{a}}\right)^{a_{-}} \leq 1+|x| \leq 2\left(1+|x|_{\vec{a}}\right)^{a_{+}}$;
\item[\rm{(vii)}] for any measurable function $f$ on $\mathbb{R}^{n}$,
$$
\int_{\mathbb{R}^{n}} f(x) d x=\int_{0}^{\infty} \int_{S^{n-1}} f\left(\rho^{\vec{a}} \xi\right) \rho^{v-1} d \sigma(\rho) d \rho,
$$
where $S^{n-1}$ denotes the $n-1$ dimension unit sphere of $\mathbb{R}^{n}$ and $\sigma(\rho)$ the spherical measure.
\end{enumerate}
\end{lemma}

\begin{remark}
\rm{
 By Lemma \ref{norm proper} (i), we easily know that the anisotropic quasi-homogeneous norm $|\cdot|_{\vec{a}}$ is a norm if and only if $\vec{a}=(a_1,a_2, \dots, a_n)=(1,1, \dots, 1)$, and in this case, the homogeneous quasi-norm $|\cdot|_{\vec{a}}$ becomes the Euclidean norm $|\cdot|$.}
\end{remark}
For any $\vec{a} \in[1, \infty)^{n}, r \in(0, \infty)$ and $x \in \mathbb{R}^{n}$. The anisotropic ball with radius $r$ and center $x$ is defined as $B _{\vec{a}}(x, r)=\left\{y \in \mathbb{R}^{n}:|y-x|_{\vec{a}}<r\right\}$. Then $B_{\vec{a}}(x, r)=x+r^{\vec{a}} B_{\vec{a}}(\vec{0}_{n}, 1)$ and $\left|B_{\vec{a}}(x, r)\right|=v_{n} r^{v}$, where $v_{n}:=|B_{\vec{a}}(\vec{0}_{n}, 1)|$. Moreover, from Lemma \ref{norm proper} (ii), we deduce that $B_{0}:=B_{\vec{a}}(\vec{0}_{n}, 1)=B(\vec{0}_{n}, 1)$, where $B(\vec{0}_{n}, 1)$ denotes the unit ball of $\mathbb{R}^{n}$, namely, $B(\vec{0}_{n}, 1):=\left\{y \in \mathbb{R}^{n}:|y|<1\right\}$. We will always consider $\mathfrak{B}$ to be the set of all anisotropic balls in the following.$$
\mathfrak{B}:=\left\{B_{\vec{a}}(x, r): x \in \mathbb{R}^{n}, r \in(0, \infty)\right\} .
$$
For any $B \in \mathfrak{B}$ centered at $x \in \mathbb{R}^{n}$ with radius $r \in(0, \infty)$ and $\delta \in(0, \infty)$, let
$$
B^{(\delta)}:=B_{\vec{a}}^{(\delta)}(x, r):=B_{\vec{a}}(x, \delta r) .
$$
In addition, for any $x \in \mathbb{R}^{n}$ and $r \in(0, \infty)$, the anisotropic cube $Q_{\vec{a}}(x, r)$ is defined by setting $Q_{\vec{a}}(x, r):=x+r^{\vec{a}}(-1,1)^{n}$, whose Lebesgue measure $\left|Q_{\vec{a}}(x, r)\right|$ equals $2^{n} r^{v}$. Denote by $\mathfrak{Q}$ the set of all anisotropic cubes, namely,
$$
\mathbb{Q}:=\left\{Q_{\vec{a}}(x, r): x \in \mathbb{R}^{n}, r \in(0, \infty)\right\} .
$$
Now let us recall the definition of mixed-norm Lebesgue spaces from \cite{benedek1961space}.
\begin{definition}(Mixed Lebesgue spaces)(\cite{benedek1961space})
	Let $\vec{p}=\left(p_1, p_2,\dots, p_n\right)\in (0,\infty]^n.$ Then the mixed-norm Lebesgue space $L^{\vec{p}}(\mathbb{R}^n)$ is defined to be the set of all measurable functions $f$ such that
	$$\left\|f\right \|_{L^{\vec{p}}{(\mathbb{R}^n)}}:= \left(\int_{\mathbb{R}}\dots\left(\int_{\mathbb{R}}\left(\int_{\mathbb{R}}\mid f(x_1, x_2,\dots, x_n)\mid^{p_1}dx_1\right)^{\frac{p_2}{p_1}}dx_2\right)^{\frac{p_3}{p_2}}\dots dx_n\right)^{\frac{1}{p_n}}< \infty ,$$
	 If $p_j=\infty,$ then we have to make appropriate modifications.
\end{definition}
\begin{remark} \label{mixedpro}
	\rm{\begin{enumerate}
		\item [(i)] When $\vec{p}=(p, \dots, p)\in (0,\infty]^n$, the mixed-norm Lebesgue space $L^{\vec{p}}(\mathbb{R}^n)$ is classical Lebesgue space $L^p(\mathbb{R}^n)$.
		\item [(ii)] $(L^{\vec{p}}(\mathbb{R}^n), \|\cdot\|_{L^{\vec{p}}(\mathbb{R}^n)})$ is a quasi-Banach space for any $\vec{p}\in (0,\infty]^n$, and $(L^{\vec{p}}(\mathbb{R}^n), \|\cdot\|_{L^{\vec{p}}(\mathbb{R}^n)})$ is a Banach space for any $\vec{p}\in (1,\infty]^n$,see [\cite{benedek1961space}, p.304, Theorem 1 ].
		\item [(iii)] (H\"older's inequality on mixed-norm Lebesgue spaces) \quad  Let $0<\vec{p}\leq \infty$, $f\in L^{\vec{p}}(\mathbb R^n), g\in L^{\vec{p}^\prime}(\mathbb R^n)$, and for any $p_i~(i\in \{1,2, \dots, n\})$,${1}/{p_i}+{1}/{p_i^{\prime}}=1$, then
	\begin{equation*}
		\int_{\mathbb{R}^n}\left|f(x)g(x)\right|dx \leq \|f\|_{L^{\vec{p}}(\mathbb{R}^n)}\|g\|_{L^{\vec{p}^{\prime}}(\mathbb{R}^n)},
	\end{equation*}
where $\vec{p}^{\prime}$ denotes the conjugate index of $\vec{p}$.
     \item [(iv)]  Let $ \vec{p} \in(0, \infty]^{n}$. Then, for any  $s \in(0, \infty)$ and  $f \in L^{\vec{p}}\left(\mathbb{R}^{n}\right)$\\ 
$$\left\||f|^{s}\right\|_{L^{\vec{p}\left(\mathbb{R}^{n}\right)}}=\|f\|_{L^{s \vec{p}\left(\mathbb{R}^{n}\right)}}^{s} .$$
		\end{enumerate}	}
\end{remark}
Let us recall some necessary notations before defining the anisotropic mixed-norm Herz spaces. Let $B_{k,\vec{a}}:=B^{(2^k)}_{\vec{a}}(\vec{0}_{n},1)=B_{\vec{a}}(\vec{0}_{n},2^k)=\{y\in \mathbb{R}^n:~|y|_{\vec{a}}\leq 2^k\}$, and $A_{k,\vec{a}}=B_{k,\vec{a}}\backslash B_{k-1, \vec{a}}$ for $k\in \mathbb{Z}$. Denote $\chi_{_k}=\chi_{_{A_{k,\vec{a}}}}$ for any $k\in \mathbb{Z}$, and $\widetilde{\chi}_{_k}=\chi_{_k}$ for any $k\in \mathbb{N}$, $\widetilde{\chi}_{_0}=\chi_{_{B_0}}$. 

\begin{definition}
Let $\vec{a}:=\left(a_{1}, \dots, a_{n}\right) \in[1, \infty)^{n}$, $\alpha \in \mathbb{R}$, $0<p\leq \infty$, $0<\vec{q}\leq \infty$. The homogeneous anisotropic mixed-norm Herz spaces $\dot{K}_{\vec{q},\vec{a}}^{\alpha, p}(\mathbb{R}^n)$ associated with $\vec{a}$ is defined by
$$\dot{K}_{\vec{q},\vec{a}}^{\alpha, p}(\mathbb{R}^n):=\left\{f\in L_{{\rm loc}}^{\vec{q}}(\mathbb{R}^n \backslash\{\boldsymbol{0}\}): \|f\|_{\dot{K}_{\vec{q},\vec{a}}^{\alpha, p}(\mathbb{R}^n)}=\left(\sum_{k\in \mathbb{Z}} \left|B_{k,\vec{a}}\right|^{\alpha p} \|f\chi_k\|_{L^{\vec{q}}}^p\right)^{{1}/{p}}< \infty\right\}.$$
\end{definition}

\begin{definition}
Let $\vec{a}:=\left(a_{1}, \dots, a_{n}\right) \in[1, \infty)^{n}$, $\alpha \in \mathbb{R}$, $0<p\leq \infty$, $0<\vec{q}\leq \infty$. The non-homogeneous anisotropic mixed-norm Herz spaces $\dot{K}_{\vec{q},\vec{a}}^{\alpha, p}(\mathbb{R}^n)$ associated with $\vec{a}$ is defined by
$$K_{\vec{q},\vec{a}}^{\alpha, p}(\mathbb{R}^n):=\left\{f\in L_{{\rm loc}}^{\vec{q}}: \|f\|_{K_{\vec{q},\vec{a}}^{\alpha, p}(\mathbb{R}^n)}=\left(\sum_{k=0}^{\infty}\left|B_{k,\vec{a}}\right|^{\alpha p}\|f\widetilde{\chi}_k\|_{L^{\vec{q}}}^p\right)^{{1}/{p}}< \infty\right\}.$$
\end{definition}

\begin{remark}
\rm{
\begin{enumerate}
	\item [(1)] We assume that $\vec{a}=(1,1, \dots, 1)$, then anisotropic mixed-norm Herz spaces $\dot{K}_{\vec{q},\vec{a}}^{\alpha, p}(\mathbb{R}^n)(\text{or}~{K}_{\vec{q},\vec{a}}^{\alpha, p}(\mathbb{R}^n))$ associated with $\vec{a}$ are the typical isotropic mixed-norm Herz spaces $\dot{K}_{\vec{q}}^{\alpha, p}(\mathbb{R}^n)(\text{or}~K_{\vec{q}}^{\alpha, p}(\mathbb{R}^n))$, which are defined by Wei in \cite{wei2021characterization}.
\end{enumerate} }
\end{remark}

\begin{proposition} \label{quasiBanach}
The homogeneous anisotropic mixed-norm Herz space $\dot{K}_{\vec{q}, \vec{a}}^{\alpha, p}(\mathbb{R}^n)$ and the non-homogeneous anisotropic mixed-norm Herz space $K_{\vec{q},\vec{a}}^{\alpha, p}(\mathbb{R}^n)$ are quasi-Banach spaces. 
\end{proposition}
\begin{proof}
 The proof of the non-homogeneous mixed-norm Herz spaces $\dot{K}_{\vec{q}, \vec{a}}^{\alpha, p}(\mathbb{R}^n)$ is similar to the proof of the homogeneous spaces $K_{\vec{q}, \vec{a}}^{\alpha, p}(\mathbb{R}^n)$.\\
  We first assert that  $\|\cdot\|_{\dot{K}_{\vec{q}, \vec{a}}^{\alpha, p}(\mathbb{R}^n)}$ is quasi-norm.
 \par 
 The positivity and the homogeneity are clear. Therefore, just the quasi-triangle inequality needs to be verified.
  \begin{align*}
  \|f+g\|_{\dot{K}_{\vec{q},\vec{a}}^{\alpha, p}(\mathbb{R}^n)} 
  &=\left(\sum_{k\in \mathbb{Z}}\left|B_{k, \vec{a}}\right|^{\alpha p}\|f\chi_k+g\chi_k\|_{L^{\vec{q}}(\mathbb{R}^n)}^p\right)^{{1}/{p}}	\\
 &\leq\left(\sum_{k\in \mathbb{Z}}\left|B_{k, \vec{a}}\right|^{\alpha p}\left(\max(1,2^{\sum\limits_
 {i=1}^n {(1-q_i)}/{qi}})\left(\|f\chi_k\|_{L^{\vec{q}}(\mathbb{R}^n)}+\|g\chi_k\|_{L^{\vec{q}}(\mathbb{R}^n)}\right)\right)^p\right)^{{1}/{p}}	\\ 
 &=\max(1,2^{\sum\limits_
 {i=1}^n {(1-q_i)}/{qi}})\left(\sum_{k\in \mathbb{Z}}\left|B_{k, \vec{a}}\right|^{\alpha p}\left(\left(\|f\chi_k\|_{L^{\vec{q}}(\mathbb{R}^n)}+\|g\chi_k\|_{L^{\vec{q}}(\mathbb{R}^n)}\right)\right)^p\right)^{{1}/{p}}	\\ 
 &\leq \max(1,2^{\sum\limits_{i=1}^n {(1-q_i)}/{qi}})\max(1,2^{{(1-p)}/{p}})\left(\|f\|_{\dot{K}_{\vec{q},\vec{a}}^{\alpha, p}(\mathbb{R}^n)}+\|g\|_{\dot{K}_{\vec{q},\vec{a}}^{\alpha, p}(\mathbb{R}^n)}  \right).
   \end{align*}
   Obviously, when $1\leq \vec{q},p \leq \infty$, $\|\cdot\|_{\dot{K}_{\vec{q},\vec{a}}^{\alpha, p}(\mathbb{R}^n)}$ is a norm.  
\par 
The completeness of spaces $\dot{K}_{\vec{q},\vec{a}}^{\alpha, p}(\mathbb{R}^n)$ is then demonstrated. Choosing a Cauchy sequence $\{f_j\}_{j=1}^{\infty} \in \dot{K}_{\vec{q},\vec{a}}^{\alpha, p}(\mathbb{R}^n)$ satisfies that
   $$\|f_{j+1}-f_{j}\|_{\dot{K}_{\vec{q},\vec{a}}^{\alpha, p}(\mathbb{R}^n)} \leq 2^{-j},$$
   Let $f(x)=f_1(x)+\sum\limits_{j=1}^{\infty}\left(f_{j+1}(x)-f_{j}(x)\right)=\lim\limits_{j\rightarrow \infty}f_j(x)$, then $f\in \dot{K}_{\vec{q},\vec{a}}^{\alpha, p}(\mathbb{R}^n)$, it is deduced by following\\
    \begin{align*}
    \|f\|_{\dot{K}_{\vec{q},\vec{a}}^{\alpha, p}(\mathbb{R}^n)}
    &=	\left\|f_1+\sum_{j=1}^{\infty}\left(f_{j+1}-f_{j}\right)\right\|_{\dot{K}_{\vec{q},\vec{a}}^{\alpha, p}(\mathbb{R}^n)}\\
    &\leq \|f_1\|_{\dot{K}_{\vec{q},\vec{a}}^{\alpha, p}(\mathbb{R}^n)}+\left\|\sum_{j=1}^{\infty}\left(f_{j+1}-f_{j}\right)\right\|_{\dot{K}_{\vec{q},\vec{a}}^{\alpha, p}(\mathbb{R}^n)}\\
    &\leq C\left(\|f_1\|_{\dot{K}_{\vec{q},\vec{a}}^{\alpha, p}(\mathbb{R}^n)}+\sum_{j=1}^{\infty}\|f_{j+1}-f_{j}\|_{\dot{K}_{\vec{q},\vec{a}}^{\alpha, p}(\mathbb{R}^n)}\right)\\
    &\leq C\left(\|f_1\|_{\dot{K}_{\vec{q},\vec{a}}^{\alpha, p}(\mathbb{R}^n)}+\sum_{j=1}^{\infty}2^{-j}\right) \leq C.
    \end{align*}
   Furthermore, $\{f_j\}_{j=1}^{\infty}$ converge to $f\in \dot{K}_{\vec{q},\vec{a}}^{\alpha, p}((\mathbb{R}^n)$,
   \begin{align*}
   	\|f-f_J\|_{\dot{K}_{\vec{q},\vec{a}}^{\alpha, p}(\mathbb{R}^n)}
   	&=\left\|f_1+\sum_{j=1}^{\infty}\left(f_{j+1}-f_{j}\right)-f_J\right\|_{\dot{K}_{\vec{q},\vec{a}}^{\alpha, p}(\mathbb{R}^n)}\\
   	&\leq \left\|\sum_{j=1}^{\infty}\left(f_{j+1}-f_{j}\right)-f_J-\sum_{j=1}^{J-1}\left(f_{j+1}-f_{j}\right)-f_J\right\|_{\dot{K}_{\vec{q},\vec{a}}^{\alpha, p}(\mathbb{R}^n)}\\
   	&\leq \left\|\sum_{j=J+1}^{\infty}\left(f_{j+1}-f_{j}\right)\right\|_{\dot{K}_{\vec{q},\vec{a}}^{\alpha, p}(\mathbb{R}^n)}\\
   	&\leq 2^{1-J}.
   \end{align*}
   Then,
   $$\lim\limits_{J\rightarrow \infty}\|f-f_J\|_{\dot{K}_{\vec{q},\vec{a}}^{\alpha, p}}=0.$$
   So, $\dot{K}_{\vec{q},\vec{a}}^{\alpha, p}(\mathbb{R}^n)$ and $K_{\vec{q},\vec{a}}^{\alpha, p}(\mathbb{R}^n)$ are quasi-Banach spaces.
\end{proof}

\begin{proposition}\label{include}
Let $\vec{a}\in (1, \infty]^n$, $\alpha \in \mathbb{R}$, $0<p\leq \infty$, $0<\vec{q}\leq \infty$, where $\vec{q}=(q_1,q_2,\dots,q_n)$. The following inclusions are valid.
\begin{enumerate}
	\item [\rm{(i)}] If $p_1\leq p_2$, then $\dot{K}_{\vec{q},\vec{a}}^{\alpha, p_1}(\mathbb{R}^n)\subset \dot{K}_{\vec{q},\vec{a}}^{\alpha, p_2}(\mathbb{R}^n)$ and $K_{\vec{q},\vec{a}}^{\alpha, p_1}(\mathbb{R}^n)\subset K_{\vec{q},\vec{a}}^{\alpha, p_2}(\mathbb{R}^n)$.
	\item [\rm{(ii)}] If $\alpha_2\leq \alpha_1$, then $K_{\vec{q},\vec{a}}^{\alpha_1, p}(\mathbb{R}^n)\subset K_{\vec{q},\vec{a}}^{\alpha_2, p}(\mathbb{R}^n)$.
	\item [\rm{(iii)}] If $\vec{q}_{_1}\leq \vec{q}_{_2}$, then $\dot{K}_{\vec{q}_{_2},\vec{a}}^{\alpha+\sum\limits_{i=1}^n({1}/{q_{1i}}-{1}/{q_{2i}}), p}(\mathbb{R}^n)\subset \dot{K}_{\vec{q}_{_1},\vec{a}}^{\alpha, p}(\mathbb{R}^n)$ and $K_{\vec{q}_{_2},\vec{a}}^{\alpha+\sum\limits_{i=1}^n({1}/{q_{1i}}-{1}/{q_{2i}}), p}(\mathbb{R}^n)\subset K_{\vec{q}_{_1},\vec{a}}^{\alpha, p}(\mathbb{R}^n)$.
\end{enumerate}
\end{proposition}

\begin{proof}
	(i)\quad First recall an inequality
\begin{equation}
	\left(\sum\limits_{k=1}^{\infty}|a_k|\right)^r\leq \sum\limits_{k=1}^{\infty}|a_k|^r\quad \quad \quad \quad (0<r\leq 1).\label{inequlity}
\end{equation}
Through this inequality, one can directly get the following result
\begin{align*}
\|f\|_{\dot{K}_{\vec{q},\vec{a}}^{\alpha, p_2}(\mathbb{R}^n)}&=\left(\sum\limits_{k\in \mathbb{Z}}\left(\left|B_{k,\vec{a}}\right|^{\alpha}\|f\chi_k\|_{L^{\vec{q}}(\mathbb{R}^n)}\right)^{p_2}\right)^{{1}/{p_2}}\\
&\leq \left(\sum\limits_{k\in \mathbb{Z}}\left(\left|B_{k,\vec{a}}\right|^{\alpha}\|f\chi_k\|_{L^{\vec{q}}(\mathbb{R}^n)}\right)^{p_1}\right)^{{1}/{p_1}} \leq \|f\|_{\dot{K}_{\vec{q},\vec{a}}^{\alpha, p_1}(\mathbb{R}^n)}.
\end{align*}
(ii) From the definition of anisotropic mixed-norm Herz spaces, we can easily get,
$$\|f\|_{K_{\vec{q},\vec{a}}^{\alpha_2, p}(\mathbb{R}^n)}\leq \left(\sum\limits_{k=0}^{\infty}\left|B_{k, \vec{a}}\right|^{\alpha_1 p}\|f\widetilde{\chi}_k\|_{L^{\vec{q}}(\mathbb{R}^n)}^p\right)^{{1}/{p}}	
\leq C \|f\|_{K_{\vec{q},\vec{a}}^{\alpha_{1}, p}(\mathbb{R}^n)}.$$
(iii)
\begin{align*}
\|f\|_{\dot{K}_{\vec{q}_1,\vec{a}}^{\alpha, p}(\mathbb{R}^n)}&= \left(\sum_{k\in \mathbb{Z}}\left|B_{k,\vec{a}}\right|^{\alpha p}\left(\int_{\mathbb{R}}\dots \left(\int_{\mathbb{R}}\left(\int_{\mathbb{R}}|f\chi_k(x)|^{q_{11}}dx_1\right)^{{q_{12}}/{q_{11}}}dx_2\right)^{{q_{13}}/{q_{12}}}\dots dx_n\right)^{{p}/{q_{1n}}}\right)^{{1}/{p}}	\\
&\leq  \left(\sum_{k\in \mathbb{Z}}\left|B_{k,\vec{a}}\right|^{\alpha p}\left(\int_{2^{k-1}\leq|x_n|_{(a_n)}<2^k}\dots \left(\int_{2^{k-1}\leq|x_1|_{(a_1)}<2^k}|f(x)|^{q_{11}}dx_1\right)^{{q_{12}}/{q_{11}}}\dots dx_n\right)^{{p}/{q_{1n}}}\right)^{\frac{1}{p}}	\\
%&\leq \Bigg(\sum_{k\in \mathbb{Z}}b^{k\alpha p}\bigg(\int_{b^{k-1}\leq|x_n|<b^k}\dots \bigg(\int_{b^{k-1}\leq|x_1|<b^k}\bigg(\int_{b^{k-1}\leq|x_1|<b^k}|f(x)|^{q_{21}}dx_1\bigg)^{\frac{q_{12}}{q_{21}}}\\
%&  \quad \quad  \times b^{k(\frac{1}{q_{11}}-\frac{1}{q_{21}})q_{12}}dx_2\bigg)^{\frac{q_{13}}{q_{12}}}\dots dx_n\bigg)^{\frac{p}{q_{1n}}}\Bigg)^{\frac{1}{p}(A; \mathbb{R}^n)}	\\
%&\quad \quad \quad \quad \quad \quad \quad \quad \quad \quad \quad \quad\dots \mbox{ n \quad  H\"older\quad inequality }\\
&\leq \left(\sum_{k\in \mathbb{Z}}\left|B_{k,\vec{a}}\right|^{(\alpha+\sum\limits_{i=1}^{n}({1}/{q_{1i}}-{1}/{q_{2i}}) p}\left(\int_{\mathbb{R}}\dots \left(\int_{\mathbb{R}}|f\chi_k(x)|^{q_{21}}dx_1\right)^{{q_{22}}/{q_{21}}}\dots dx_n\right)^{{p}/{q_{2n}}}\right)^{{1}/{p}}	\\
&\leq \|f\|_{\dot{K}_{\vec{q}_2,\vec{a}}^{\alpha+\sum\limits_{i=1}^n({1}/{q_{1i}}-{1}/{q_{2i}}), p}(\mathbb{R}^n)}.
\end{align*}
Using the similar method, we can get the result (1) and (3) on spaces ${K}_{\vec{q},\vec{a}}^{\alpha, p}(\mathbb{R}^n)$.%non-homogeneous mixed Herz spaces.
\end{proof}

Concerning boundedness issues on anisotropic mixed Herz spaces, the H\"older's inequality and dual property as following are crucial. 
\begin{proposition}(H\"older's inequality on anisotropic mixed-norm Herz spaces)\label{Holder}
	Let $0<p_i\leq \infty$, $1\leq \vec{q}\leq \infty$, $\alpha, \alpha_i\in \mathbb{R}~(i=1,2)$, and $\alpha=\alpha_1+\alpha_2$, ${1}/{p}={1}/{p_1}+{1}/{p_2}$, ${1}/{\vec{q}}={1}/{\vec{q}_{_1}}+{1}/{\vec{q}_{_2}}$, then
	\begin{equation*}
		\|fg\|_{\dot{K}_{\vec{q},\vec{a}}^{\alpha, p}(\mathbb{R}^n)}\leq \|f\|_{\dot{K}_{\vec{q}_{_1},\vec{a}}^{\alpha_1, p_1}(\mathbb{R}^n)}\|g\|_{\dot{K}_{\vec{q}_{_2},\vec{a}}^{\alpha_2, p_2}(\mathbb{R}^n)}.
	\end{equation*}
\end{proposition}

\begin{proposition}(dual spaces of anisotropic mixed-norm Herz spaces)\label{dual} Let $\alpha \in \mathbb{R}, 0<p<\infty, 1 \leq \vec{q}<\infty$, and $1/p + 1/ p^{\prime}=1$, where $p^{\prime}=\infty$ if $0<p \leq 1$. Then
$$
\left(\dot{K}_{\vec{q}'\vec{a}}^{\alpha, p}\left(\mathbb{R}^{n}\right)\right)^{*}=\dot{K}_{\vec{q}^{\prime},\vec{a}}^{-\alpha, p^{\prime}}\left(\mathbb{R}^{n}\right),~\left(\left(\dot{K}_{\vec{q},\vec{a}}^{-\alpha, \infty}\right)_{0}\left(\mathbb{R}^{n}\right)\right)^{*}=\dot{K}_{\vec{q}^{\prime},\vec{a}}^{\alpha, 1}\left(\mathbb{R}^{n}\right)
$$
and
$$
\left(K_{\vec{q},\vec{a}}^{\alpha, p}\left(\mathbb{R}^{n}\right)\right)^{*}=K_{\vec{q}^{\prime},\vec{a}}^{-\alpha, p^{\prime}}\left(\mathbb{R}^{n}\right),~\left(\left(K_{\vec{q},\vec{a}}^{-\alpha, \infty}\right)_{0}\left(\mathbb{R}^{n}\right)\right)^{*}=K_{\vec{q}^{\prime},\vec{a}}^{\alpha, 1}\left(\mathbb{R}^{n}\right),
$$
where $\left(\dot{K}_{\vec{q},\vec{a}}^{-\alpha, \infty}\right)_{0}\left(\mathbb{R}^{n}\right)$ denote $f \in \dot{K}_{\vec{q},\vec{a}}^{\alpha, \infty}\left(\mathbb{R}^{n}\right)$ and $2^{k \alpha}\left\|f \chi_{k}\right\|_{L^{\vec{q}}\left(\mathbb{R}^{n}\right)} \rightarrow 0$ as $|k| \rightarrow \infty$, similar, $\left(K_{\vec{q},\vec{a}}^{-\alpha, \infty}\right)_{0}\left(\mathbb{R}^{n}\right)$ denote $f \in K_{\vec{q},\vec{a}}^{\alpha, \infty}\left(\mathbb{R}^{n}\right)$ and $2^{k \alpha}\left\|f \chi_{k}\right\|_{L^{\vec{q}}\left(\mathbb{R}^{n}\right)} \rightarrow 0$ as $k \rightarrow \infty$.
\end{proposition}

\begin{remark}
\rm{
We may readily draw the following conclusion in view of Proposition \rm{\ref{Holder}},  Proposition \rm{\ref{dual}} and closed-graph theorem, it is a norm representation of anisotropic mixed-norm Herz spaces.}\\
Let $\alpha \in \mathbb{R}$, $1 \leq p$, $\vec{q}<\infty$, and $1 / p+1 / p^{\prime}=1$, $1 / \vec{q}+1 / \vec{q}^{\prime}=1$. Then $f \in \dot{K}_{\vec{q},\vec{a}}^{\alpha, p}\left(\mathbb{R}^{n}\right)$ if and only if
$$
\left|\int_{\mathbb{R}^{n}} f(x) g(x) d x\right|<\infty,
$$
for every $g \in \dot{K}_{\vec{q}^{\prime},\vec{a}}^{-\alpha, p^{\prime}}\left(\mathbb{R}^{n}\right)$, and in this case we have
$$
\|f\|_{\dot{K}_{\vec{q},\vec{a}}^{\alpha, p}\left(\mathbb{R}^{n}\right)}=\sup \left\{\left|\int_{\mathbb{R}^{n}} f(x) g(x) d x\right|:\|g\|_{\dot{K}_{\vec{q}^{\prime},\vec{a}}^{-\alpha, p^{\prime}}\left(\mathbb{R}^{n}\right)} \leq 1\right\}.
$$
There exists a similar result for the non-homogeneous space $K_{\vec{q},\vec{a}}^{\alpha, p}\left(\mathbb{R}^{n}\right)$.
\end{remark}
\begin{proof}
Proposition \rm{\ref{Holder}} is merely an application of H\"older's inequality on mixed Lebesgue spaces. In order to prove Proposition \rm{\ref{dual}}, Hernandez and Yang give us some ideas. In fact, the methods are similar to the Theorem 3.1 of \cite{wei2021characterization},  The details are omitted here.
\end{proof}

To prove the main theorems in this paper, we need the following Lemma to be true.
\begin{lemma}
	Let $\vec{a}=(a_1, a_2, \dots, a_n)\in [1, \infty)^n$, $\vec{p}=(p_1, p_2, \dots, p_n)\in (0, \infty)^n$, $r\in ((0, \infty)$, $x\in \mathbb{R}^n$ and $B_{\vec{a}}(x, r) \in \mathfrak{B}$. Then, 
	$$\|\chi_{B_{\vec{a}(x,r)}}\|_{L^{\vec{p}}\left(\mathbb{R}^{n}\right)} \leq r^{\sum\limits_{i=1}^n a_i/p_i}.$$ 
\end{lemma}
\begin{proof} According to the definition of mixed-norm Lebesgue spaces and simple calculation, we get,
\begin{align*}
\|\chi_{B_{\vec{a}(x,r)}}\|_{L^{\vec{p}}\left(\mathbb{R}^{n}\right)}
		&=\left(\int_{\mathbb{R}}\dots \left(\int_{\mathbb{R}}\left(\int_{\mathbb{R}}|\chi_{B_{\vec{a}(x,r)}}(x_1,x_2, \dots, x_n)|^{q_{1}}dx_1\right)^{{q_{2}}/{q_{1}}}dx_2\right)^{{q_{3}}/{q_{2}}}\dots dx_n\right)^{{1}/{q_{n}}}	\\
		&\leq \left(\int_{|x_n|<r^{a_n}}\dots \left(\int_{|x_2|<r^{a
		_2}}\left(\int_{|x_1|<r^{a_1}} 1~~ dx_1\right)^{{q_{2}}/{q_{1}}}dx_2\right)^{{q_{3}}/{q_{2}}}\dots dx_n\right)^{{1}/{q_{n}}}	\\
		&\leq r^{\sum\limits_{i=1}^n a_i/p_i}.	
\end{align*}
This finishes the proof.
\end{proof}
\subsection{Block Decomposition for Anisotropic Mixed-Norm Herz Spaces}
In this section, we want to obtain the block decomposition of anisotropic mixed-norm Herz spaces, because the block decompositions of various function spaces are useful for estimating the boundedness of operators. We first provide a necessary definition before starting the demonstration of the block decompositions theorem. The block conditions listed below are more appropriate for the mixed-norm than the traditional conditions.\begin{definition}
	Let $0<\alpha<\infty$, $1\leq \vec{q}<\infty$,
	\begin{enumerate}
		\item [\rm{(1)}] A function $a(x)$ is called to be a central $(\alpha, \vec{q})$-block if
		\begin{itemize}
			\item [\rm{(i)}] ${\rm supp}~a \subset B_{k,\vec{a}}\in \mathfrak{B} \quad k\in \mathbb{Z}, $
			\item [\rm{(ii)}] $\|a\|_{L^{\vec{q}}}\leq C |B_{k,\vec{a}}|^{-\alpha}.$
		\end{itemize}
		\item [\rm{(2)}] A function $a(x)$ is called to be a central $(\alpha, \vec{q})$-block of restrict type if it satisfies ${\rm(ii)}$
		\begin{itemize}
			\item [\rm{(iii)}] ${\rm supp} a \subset B_k \in \mathfrak{B} \quad \mbox {for some}\quad r>1.$
		\end{itemize}
	\end{enumerate}
\end{definition}

\begin{theorem}\label{4.1}
	Let $0<\alpha<\infty$, $0<p<\infty$ and $1\leq \vec{q}<\infty$. Then,
	\begin{enumerate}
		\item [\rm(1)] $f\in \dot{K}_{\vec{q},\vec{a}}^{\alpha, p}(\mathbb{R}^n)$ if and only if
	$f$ can be represented by
	$$f(x)=\sum_{k\in \mathbb{Z}}\lambda_k b_k(x),$$
	where each $b_k$ is a central $(\alpha, \vec{q})$-block with support contained in $B_k$ and $\sum_k|\lambda_k|^p<\infty.$
	\item [\rm(2)] $f\in {K}_{\vec{q},\vec{a}}^{\alpha, p}(\mathbb{R}^n)$ if and only if $f$ can be represented by
	$$f(x)=\sum_{k\in \mathbb{Z}}\lambda_k b_k(x),$$
	where each $b_k$ is a central $(\alpha, \vec{q})$-block of restrict with support contained in $B_k$ and $\sum_k|\lambda_k|^p<\infty.$
	\end{enumerate}
\end{theorem}

\begin{proof} The non-homogeneous case can use the same way, we only prove the homogeneous case.\\
	Sufficiency, assume that $f\in \dot{K}_{\vec{q},\vec{a}}^{\alpha, p}(\mathbb{R}^n)$, then
	\begin{align*}
		f(x)&=\sum_{k\in \mathbb{Z}}f(x)\chi_k(x)\\
		&=\sum_{k\in \mathbb{Z}}|B_{k,\vec{a}}|^{\alpha}\|f\chi_k\|_{L^{\vec{q}}(\mathbb{R}^n)}\frac{f(x)\chi_k(x)}{|B_{k,\vec{a}}|^{\alpha}\|f\chi_k\|_{L^{\vec{q}}(\mathbb{R}^n)}}\\
		&=\sum_{k\in \mathbb{Z}}\lambda_k(x)b_k(x).
	\end{align*}
	It is easy to check that
	$$\|b_k\|_{L^{\vec{q}}(\mathbb{R}^n)}=\left\|\frac{f(x)\chi_k(x)}{|B_{k,\vec{a}}|^{\alpha}\|f\chi_k\|_{L^{\vec{q}}(\mathbb{R}^n)}}\right\|_{L^{\vec{q}}(\mathbb{R}^n)}=|B_{k,\vec{a}}|^{-\alpha},$$
	and
	\begin{align*}
		\sum_{k\in \mathbb{Z}}|\lambda_k|^p=\sum_{k\in \mathbb{Z}}|B_{k,\vec{a}}|^{\alpha p}\|f\chi_k\|_{L^{\vec{q}}(\mathbb{R}^n)}^p=\|f\|_{\dot{K}_{\vec{q},\vec{a}}^{\alpha, p}(\mathbb{R}^n)}^p<\infty.
	\end{align*}
	Necessity, we consider the following two cases. When $0<p\leq 1,$
	\begin{align*}
		\|f\|_{\dot{K}_{\vec{q},\vec{a}}^{\alpha, p}(\mathbb{R}^n)}^p&=\sum_{k\in \mathbb{Z}}|B_{k,\vec{a}}|^{\alpha p}\|f\chi_k\|_{L^{\vec{q}}(\mathbb{R}^n)}^p\\
		&=\sum_{k\in \mathbb{Z}}|B_{k,\vec{a}}|^{\alpha p}\left\|\sum_{j\in \mathbb{Z}}\lambda_jb_j\chi_k\right\|_{L^{\vec{q}}(\mathbb{R}^n)}^p\\
		&\leq \sum_{k\in \mathbb{Z}}|B_{k,\vec{a}}|^{\alpha p}\left\|\sum_{j\geq k}\lambda_jb_j\right\|_{L^{\vec{q}}(\mathbb{R}^n)}^p\\
		&\leq \sum_{k\in \mathbb{Z}}|B_{k,\vec{a}}|^{\alpha p}\left(\sum_{j\geq k}|\lambda_j|\left\|b_j\right\|_{L^{\vec{q}}(\mathbb{R}^n)}\right)^p\\
		&\leq \sum_{k\in \mathbb{Z}}|B_{k,\vec{a}}|^{\alpha p}\sum_{j\geq k}|\lambda_j|^p\left\|b_j\right\|_{L^{\vec{q}}(\mathbb{R}^n)}^p\\
		&\leq C \sum_{j\in \mathbb{Z}}|\lambda_k|^p\sum_{k\leq j}\frac{|B_{k,\vec{a}}|^{\alpha p}}{|B_{j,\vec{a}}|^{\alpha p}}\\
		&\leq C \sum_{j\in \mathbb{Z}}|\lambda_j|^p<\infty.
	\end{align*}
	When $1<p< \infty,$
	\begin{align*}
		\|f\|_{\dot{K}_{\vec{q},\vec{a}}^{\alpha, p}(\mathbb{R}^n)}^p&\leq \sum_{k\in \mathbb{Z}}|B_{k,\vec{a}}|^{\alpha p}\left(\sum_{j\geq k}|\lambda_j|\left\|b_j\right\|_{L^{\vec{q}}(\mathbb{R}^n)}^{1/2}	\left\|b_j\right\|_{L^{\vec{q}}(\mathbb{R}^n)}^{1/2}\right)^p\\
		&\leq \sum_{k\in \mathbb{Z}}|B_{k,\vec{a}}|^{\alpha p}\left(\sum_{j\geq k}|\lambda_j|^p\left\|b_j\right\|_{L^{\vec{q}}(\mathbb{R}^n)}^{p/2}\right) \left(\sum_{j\geq k}	\left\|b_j\right\|_{L^{\vec{q}}(\mathbb{R}^n)}^{p^{\prime}/2}\right)^{p/p^{\prime}}\\
		&\leq \sum_{k\in \mathbb{Z}}|B_{k,\vec{a}}|^{\alpha p}\left(\sum_{j\geq k}|\lambda_j|^p|B_{k,\vec{a}}|^{-\alpha p/2}\right)\left(\sum_{j\geq k}	|B_{k,\vec{a}}|^{-\alpha p^{\prime}/2}\right)^{p/p^{\prime}}\\
		&\leq C\sum_{j\in \mathbb{Z}}|\lambda_j|^p\sum_{k\leq j}\left(\frac{|B_{k,\vec{a}}|}{|B_{j,\vec{a}}|}\right)^{{\alpha p}/2}
		\leq C\sum_{j\in \mathbb{Z}}|\lambda_j|^p<\infty.
	\end{align*}
	This proof is completed.
\end{proof}

\subsection{Extrapolation to Anisotropic mixed-norm Herz spaces}
One of the deepest applications in the study of weight norm inequalities in harmonic analysis is Rubio de Francia extrapolation, which is important in resolving the boundedness issue for operators. The weight norm inequalities of operators on some indexes and an iteration algorithm, which is generated by a sub-linear operator with boundedness, are the key components of extrapolation theory. To simplify, we select the Hardy-Littlewood maximal operators as the elements of the iteration algorithm. Next, we review some related results.
\par 
The Hardy-Littlewood maximal operator $M(f)$ of $f \in L_{\rm loc}^1\left(\mathbb{R}^n\right)$  is defined by
$$M(f)(x)=\sup_{\mathfrak{B}\ni B\ni x}\frac{1}{|B|}\int_B|f(x)|dx.$$

\begin{definition}
	For $1<p<\infty$, a locally integrable function $\omega: \mathbb{R}^{n} \rightarrow[0, \infty)$ is said to be an $A_{p}$ weight if
$$
[\omega]_{A_{p}}=\sup _{B}\left(\frac{1}{|B|} \int_{B} \omega(x) d x\right)\left({1}/{|B|} \int_{B} \omega(x)^{-{p^{\prime}}/{p}} d x\right)^{{p}/{p^{\prime}}}<\infty
$$
where $p^{\prime}={p}/{(p-1)}$. A locally integrable function $\omega: \mathbb{R}^{n} \rightarrow[0, \infty)$ is said to be an $A_{1}$ weight if
$$
\frac{1}{|B|} \int_{B} \omega(y) d y \leq C \omega(x), \quad \text { a.e. } x \in B
$$
for some constants $C>0$. The infimum of all such $C$ is denoted by $[\omega]_{A_{1}}$.
\end{definition} 

\begin{lemma} \cite{huang2019atomic} \label{boundedness HL}
	Let $\vec{p}\in (1, \infty]^n$, then there exists a positive constant $C$, depending on $\vec{p}$, such that, for any $f\in L^{\vec{p}}\left(\mathbb{R}^n\right)$,
	$$\|M(f)\|_{L^{\vec{p}} \left(\mathbb{R}^n\right)}\leq C\|f\|_{L^{\vec{p}} \left(\mathbb{R}^n\right)}.$$
\end{lemma}

\begin{lemma}\label{boundedness HK}
	Let $0<p<\infty, 1<\vec{q}<\infty, -\frac{1}{v}\sum_{i=1}^n \frac{a_i}{q_i}< \alpha <1-\frac{-1}{v}\sum_{i=1}^n \frac{a_i}{q_i}$, then Hardy-Littlewood maximal operators $M$ are bounded on anisotropic mixed-norm Herz spaces $\dot{K}_{\vec{q}, \vec{a}}^{\alpha, p}\left(\mathbb{R}^n\right)$(or anisotropic mixed-norm Herz spaces ${K}_{\vec{q}, \vec{a}}^{\alpha, p}\left(\mathbb{R}^n\right))$.
\end{lemma}

\begin{remark}\rm{
	In the Lemma \ref{boundedness HK}, we apply the direct estimation method to solve the boundedness problem for the Hardy-Littlewood maximal operator. In truth, we can also resolve the problem using the atomic decomposition approach, however, this approach will cause the issue of $\alpha$ incomplete indexes, i.e. only gain $0< \alpha <1-\frac{-1}{v}\sum_{i=1}^n \frac{a_i}{q_i}$.}
\end{remark}
\begin{proof}
We only prove the homogeneous case, and the non-homogeneous case is similar. 
\begin{align*}
\|M(f)\|_{\dot{K}_{\vec{q}, \vec{a}}^{\alpha, p}}&=\left(\sum_{k\in \mathbb{Z}}|B_{k,\vec{a}}|^{\alpha p}\left\|M(f)\chi_k\right\|_{L^{\vec{q}}\left(\mathbb{R}^n\right)}^p\right)^{1/p}\\
	&\leq \left(\sum_{k\in \mathbb{Z}}|B_{k,\vec{a}}|^{\alpha p}\left\|\sum_{l=-\infty}^{\infty}\left|M(f\chi_l)\right|\chi_k\right\|_{L^{\vec{q}}\left(\mathbb{R}^n\right)}^p\right)^{1/p}\\
	&\leq C \left(\sum_{k\in \mathbb{Z}}|B_{k,\vec{a}}|^{\alpha p}\left\|\sum_{l=-\infty}^{k-2}\left|M(f\chi_l)\right|\chi_k\right\|_{L^{\vec{q}}\left(\mathbb{R}^n\right)}^p\right)^{1/p}\\
	&\quad +C \left(\sum_{k\in \mathbb{Z}}|B_{k,\vec{a}}|^{\alpha p}\left\|\sum_{l=k-1}^{k+1}\left|M(f\chi_l)\right|\chi_k\right\|_{L^{\vec{q}}\left(\mathbb{R}^n\right)}^p\right)^{1/p}\\
&\quad + C \left(\sum_{k\in \mathbb{Z}}|B_{k,\vec{a}}|^{\alpha p}\left\|\sum_{l=k+2}^{\infty}\left|M(f\chi_l)\right|\chi_k\right\|_{L^{\vec{q}}\left(\mathbb{R}^n\right)}^p\right)^{1/p}\\
&\leq C \left(I_1+I_2+I_3\right).
\end{align*}
Applying the Lemma \ref{boundedness HL}, which is the boundedness of Hardy-Littlewood maximal operators on mixed-norm Lebesgue spaces, we get
\begin{align*}
	I_2&\leq \left(\sum_{k\in \mathbb{Z}}|B_{k,\vec{a}}|^{\alpha p}\left\|\sum_{l=k-1}^{k+1}\left|M(f\chi_l)\right|\right\|_{L^{\vec{q}}\left(\mathbb{R}^n\right)}^p\right)^{1/p}\\
	&\leq C \left(\sum_{k\in \mathbb{Z}}|B_{k,\vec{a}}|^{\alpha p}\sum_{l=k-1}^{k+1}\left\|f\chi_l \right\|_{L^{\vec{q}}\left(\mathbb{R}^n\right)}^p\right)^{1/p}\\
	&\leq C \left(\sum_{k\in \mathbb{Z}}|B_{k,\vec{a}}|^{\alpha p}\left\|f\chi_k \right\|_{L^{\vec{q}}\left(\mathbb{R}^n\right)}^p\right)^{1/p}\\
	&\leq C\left\|f\right\|_{K_{\vec{q}, \vec{a}}^{\alpha, p}\left(\mathbb{R}^n\right)}.
\end{align*}
For $I_1$, $\alpha <1-\frac{1}{v}\sum_{i=1}^n \frac{a_i}{q_i}$, when $0<p\leq 1$,
\begin{align*}
{I_1}&\leq C\left({\sum_{k{\in}{\mathbb{Z}}}}|B_{k,\vec{a}}|^{\alpha p}\left\|{\sum_{l={-\infty }}^{k-2}}\left|M(f{\chi_l})\right|{\chi_k}\right\|^p_{L^{\vec{q}}\left(\mathbb{R}^n\right)}\right)^{{1}/{p}}\\
&\leq C \left(\sum_{k\in \mathbb{Z}}|B_{k,\vec{a}}|^{\alpha p} \left\|\sum_{l=-\infty}^{k-2}|B_{k,\vec{a}}|^{-1} \left\| f\chi_l\right\|_{L^1\left(\mathbb{R}^n\right)} \chi_k\right\|_{L^{\vec{q}}\left(\mathbb{R}^n\right)}^p\right)^{1/p}\\
&\leq C \left(\sum_{k\in \mathbb{Z}}|B_{k,\vec{a}}|^{\alpha p} \sum_{l=-\infty}^{k-2}|B_{k,\vec{a}}|^{-p} \left\| f\chi_l\right\|_{L^1\left(\mathbb{R}^n\right)}^p 2^{kp\sum\limits_{i=1}^na_i/q_i}\right)^{1/p}\\
&\leq C \left(\sum_{k\in \mathbb{Z}}|B_{k,\vec{a}}|^{\alpha p} \sum_{l=-\infty}^{k-2}|B_{k,\vec{a}}|^{-p} \left\| f\chi_l\right\|_{L^{\vec{q}}\left(\mathbb{R}^n\right)}^p \|\chi_l\|_{L^{\vec{q}^{\prime}}\left(\mathbb{R}^n\right)}^p 2^{kp\sum\limits_{i=1}^na_i/q_i}\right)^{1/p}\\
&\leq C \left (\sum_{k\in \mathbb{Z}} |B_{k,\vec{a}}|^{\alpha p} \sum_{l=-\infty}^{k-2} \left\| f\chi_l\right\|_{L^{\vec{q}}\left(\mathbb{R}^n\right)}^p 2^{(k-l)p(\sum\limits_{i=1}^n a_i/q_i-v)} \right)\\
&\leq C \left(\sum_{l\in \mathbb{Z}}|B_{l,\vec{a}}|^{\alpha p} \left\|f\chi_l\right\|_{L^{\vec{q}}\left(\mathbb{R}^n\right)}^p\right)^{1/p}=C\left\|f\right\|_{\dot{K}_{\vec{q},\vec{a}}^{\alpha, p}\left(\mathbb{R}^n\right)}.
\end{align*}
When $1<p<\infty$, then
\begin{align*}
{I_1}&\leq C \left(\sum_{k\in \mathbb{Z}}|B_{k,\vec{a}}|^{\alpha p} \left\|\sum_{l=-\infty}^{k-2}|B_{k,\vec{a}}|^{-1} \left\| f\chi_l\right\|_{L^1\left(\mathbb{R}^n\right)} \chi_k\right\|_{L^{\vec{q}}\left(\mathbb{R}^n\right)}^p\right)^{1/p}\\
&\leq C \left( \sum_{k\in \mathbb{Z}}|B_{k,\vec{a}}|^{\alpha p} \left(\sum_{l=-\infty}^{k-2} |B_{k,\vec{a}}|^{-1} \left\|f\chi_l \right\|_{L^1(\mathbb{R}^n)} 2^{k\sum\limits_{i=1}^n a_i/q_i}\right)^p\right)^{1/p}\\
&\leq C \left( \sum_{k\in \mathbb{Z}}|B_{k,\vec{a}}|^{\alpha p} \left(\sum_{l=-\infty}^{k-2} |B_{k,\vec{a}}|^{-1} \left\| f\chi_l\right\|_{L^{\vec{q}}\left(\mathbb{R}^n\right)} \|\chi_l\|_{L^{\vec{q}^{\prime}}\left(\mathbb{R}^n\right)} 2^{k\sum\limits_{i=1}^na_i/q_i}\right)^p\right)^{1/p}\\
&\leq C \left (\sum_{k\in \mathbb{Z}} |B_{k,\vec{a}}|^{\alpha p} \left( \sum_{l=-\infty}^{k-2} \left\| f\chi_l\right\|_{L^{\vec{q}}\left(\mathbb{R}^n\right)} 2^{(k-l)(\sum\limits_{i=1}^n a_i/q_i-v)} \right)^p \right)^{1/p}\\
&\leq C \Bigg \{\sum_{k\in \mathbb{Z}} |B_{k,\vec{a}}|^{\alpha p} \left(\sum_{l=-\infty}^{k-2} \left\| f\chi_l\right\|_{L^{\vec{q}}\left(\mathbb{R}^n\right)}^p 2^{[(k-l)p(\sum\limits_{i=1}^n a_i/q_i-v)]/2} \right)\\
&\quad \times \left(\sum_{l=-\infty}^{k-2} 2^{[(k-l)p^{\prime}(\sum\limits_{i=1}^n a_i/q_i-v)]/2}\right)^{p/p^{\prime}}\Bigg\}^{1/p}\\
& \leq C\left(\sum_{l\in \mathbb{Z}}|B_{l, \vec{a}}|^{\alpha p} \left\| f\chi_l\right\|_{L^{\vec{q}}\left(\mathbb{R}^n\right)}^p \left( \sum_{k=l+2}^{\infty} 2^{[(k-l)p(v\alpha-v+\sum\limits_{i=1}^na_i/q_i)]/ 2}\right)\right)^{1/p}\\
&\leq C\left\|f\right\|_{\dot{K}_{\vec{q},\vec{a}}^{\alpha, p}\left(\mathbb{R}^n\right)}.
\end{align*}
For $I_3$, $-\frac{1}{v}\sum_{i=1}^n \frac{a_i}{q_i}< \alpha$, when $0<p\leq 1$, then 
\begin{align*}
{I_3}&\leq C\left({\sum_{k{\in}{\mathbb{Z}}}}|B_{k,\vec{a}}|^{\alpha p}\left\|{\sum_{l={k+2}}^{\infty}}\left|M(f{\chi_l})\right|{\chi_k}\right\|^p_{L_{\vec{q}}\left(\mathbb{R}^n\right)}\right)^{{1}/{p}}\\	
&\leq C \left( \sum_{k\in \mathbb{Z}}|B_{k,\vec{a}}|^{\alpha p} \left(\sum_{l=k+2}^{\infty} |B_{l,\vec{a}}|^{-1} \left\|f\chi_l \right\|_{L^1(\mathbb{R}^n)} 2^{k\sum\limits_{i=1}^n a_i/q_i}\right)^p\right)^{1/p}\\
&\leq C \left( \sum_{k\in \mathbb{Z}}|B_{k,\vec{a}}|^{\alpha p} \left(\sum_{l=k+2}^{\infty} |B_{l,\vec{a}}|^{-1} \left\| f\chi_l\right\|_{L^{\vec{q}}\left(\mathbb{R}^n\right)} \|\chi_l\|_{L^{\vec{q}^{\prime}}\left(\mathbb{R}^n\right)} 2^{k\sum\limits_{i=1}^na_i/q_i}\right)^p\right)^{1/p}\\
&\leq C \left (\sum_{k\in \mathbb{Z}} |B_{k,\vec{a}}|^{\alpha p} \left( \sum_{l=k+2}^{\infty} \left\| f\chi_l\right\|_{L^{\vec{q}}\left(\mathbb{R}^n\right)} 2^{(k-l)(\sum\limits_{i=1}^n a_i/q_i-v)} \right)^p \right)^{1/p}\\
& \leq C\left(\sum_{l\in \mathbb{Z}}|B_{l, \vec{a}}|^{\alpha p} \left\| f\chi_l\right\|_{L^{\vec{q}}\left(\mathbb{R}^n\right)}^p \left( \sum_{k=-\infty}^{l-2} 2^{[(k-l)p(v\alpha+\sum\limits_{i=1}^na_i/q_i)]/ 2}\right)\right)^{1/p}\\
&\leq C \left(\sum_{l\in \mathbb{Z}}|B_{l,\vec{a}}|^{\alpha p} \left\|f\chi_l\right\|_{L^{\vec{q}}\left(\mathbb{R}^n\right)}^p\right)^{1/p}=C\left\|f\right\|_{\dot{K}_{\vec{q},\vec{a}}^{\alpha, p}\left(\mathbb{R}^n\right)}.
\end{align*}
When $1<p<\infty$, one can get
\begin{align*}
{I_3}&\leq C\left({\sum_{k{\in}{\mathbb{Z}}}}|B_{k,\vec{a}}|^{\alpha p}\left\|{\sum_{l={k+2}}^{\infty}}\left|M(f{\chi_l})\right|{\chi_k}\right\|^p_{L^{\vec{q}}\left(\mathbb{R}^n\right)}\right)^{{1}/{p}}\\	
&\leq C \left (\sum_{k\in \mathbb{Z}} |B_{k,\vec{a}}|^{\alpha p} \left( \sum_{l=k+2}^{\infty} \left\| f\chi_l\right\|_{L^{\vec{q}}\left(\mathbb{R}^n\right)} 2^{(k-l)(\sum\limits_{i=1}^n a_i/q_i)} \right)^p \right)^{1/p}\\	
&\leq C \Bigg \{\sum_{k\in \mathbb{Z}} |B_{k,\vec{a}}|^{\alpha p} \left( \sum_{l=k+2}^{\infty} \left\| f\chi_l\right\|_{L^{\vec{q}}\left(\mathbb{R}^n\right)}^p 2^{[(k-l)p(\sum\limits_{i=1}^n a_i/q_i)]/2} \right) \\
&\quad \times \left(\sum_{l=k+2}^{\infty} \left\| f\chi_l\right\|_{L^{\vec{q}}\left(\mathbb{R}^n\right)}^p 2^{[(k-l)p^{\prime}(\sum\limits_{i=1}^n a_i/q_i)]/2} \right)^{p/p^{\prime}}\Bigg\}^{1/p}\\
& \leq C\left(\sum_{l\in \mathbb{Z}}|B_{l, \vec{a}}|^{\alpha p} \left\| f\chi_l\right\|_{L^{\vec{q}}\left(\mathbb{R}^n\right)}^p \left( \sum_{k=-\infty}^{l-2} 2^{[(k-l)p(v\alpha+\sum\limits_{i=1}^na_i/q_i)]/ 2}\right)\right)^{1/p}\\
&\leq C \left(\sum_{l\in \mathbb{Z}}|B_{l,\vec{a}}|^{\alpha p} \left\|f\chi_l\right\|_{L^{\vec{q}}\left(\mathbb{R}^n\right)}^p\right)^{1/p}=C\left\|f\right\|_{\dot{K}_{\vec{q},\vec{a}}^{\alpha, p}\left(\mathbb{R}^n\right)}.
\end{align*}
This implies that Hardy-Littlewood maximal operators $M$ is bounded on $\dot{K}_{\vec{q},\vec{a}}^{\alpha, p}\left(\mathbb{R}^n\right)$, hence finishes the proof of Lemma \ref{boundedness HK}.
\end{proof}
\par 
In order to show the extrapolation theorem on anisotropic mixed-norm Herz spaces, we first introduce an iteration algorithm.
\par 
If $0<p<\infty$, $1<\vec{q}<\infty$, $-\frac{1}{v} \sum_{i=1}^n \frac{a_{i}}{q_{i}}<\alpha<\frac{1}{p_{0}}-\frac{1}{v} \sum_{i=1}^n \frac{a_{i}}{q_{i}}$, then by Lemma \ref{boundedness HK}, we know $M$ is boundedness on $\dot {K}_{(\vec{q}/p_{0})',\vec{a}}^{-p_{0}\alpha,(p/p_{0})^{\prime}}\left(\mathbb{R}^{n}\right)$. Let $B$ is the operators norm of $M$ on $\dot {K}_{(\vec{q}/p_{0})',\vec{a}}^{-p_{0}\alpha,(p/p_{0})^{\prime}}\left(\mathbb{R}^{n}\right)$,
$$B=\|M\|_{\dot {K}_{(\vec{q}/p_{0})',\vec{a}}^{-p_{0}\alpha,(p/p_{0})^{\prime}}\left(\mathbb{R}^{n}\right)\rightarrow \dot {K}_{(\vec{q}/p_{0})',\vec{a}}^{-p_{0}\alpha,(p/p_{0})^{\prime}}\left(\mathbb{R}^{n}\right)}.$$
For any non-negative locally integral function $h$, the iteration algorithm is defined by
$$\mathfrak{R}h:=\sum_{k=0}^{\infty} \frac{M^{k} h}{2^{k} B^{k}},$$ 
where $M^{k}$ is the $k$th iterations of $M$ and we denote by $M^{\circ} h=h$. \\
The operator $\mathfrak{R}$ has the following properties:
\begin{align*}
    &h(x) \leq \mathfrak{R}h(x),  \tag{R1}\\
	&\|\mathfrak{R}h\|_{\dot {K}_{(\vec{q}/p_{0})',\vec{a}}^{-p_{0}\alpha,(p/p_{0})^{\prime}}\left(\mathbb{R}^{n}\right)}\leq 2 \|h\|_{\dot {K}_{(\vec{q}/p_{0})',\vec{a}}^{-p_{0}\alpha,(p/p_{0})^{\prime}}\left(\mathbb{R}^{n}\right)},  \tag{R2}\\
	&[\mathfrak{R} h]_{A_{1}} \leq 2 B. \tag{R3}
\end{align*}
The inequality (R1) can  straight get by definition of $\mathfrak{R}$, and 
\begin{align*}
\left\|\mathfrak{R}h\right\|_{\dot {K}_{(\vec{q}/p_{0})',\vec{a}}^{-p_{0}\alpha,(p/p_{0})^{\prime}}\left(\mathbb{R}^{n}\right)}
&=\left\|\sum^{\infty}_{k=0}\frac{M^{k}h}{2^{k}B^{k}}\right\|_{\dot {K}_{(\vec{q}/p_{0})',\vec{a}}^{-p_{0}\alpha,(p/p_{0})^{\prime}}\left(\mathbb{R}^{n}\right)}\\
&=\sum^{\infty}_{k=0}\frac{1}{2^{k}B^{k}}\left\|M^{k}h\right\|_{\dot {K}_{(\vec{q}/p_{0})',\vec{a}}^{-p_{0}\alpha,(p/p_{0})^{\prime}}\left(\mathbb{R}^{n}\right)}\\
&\leq \sum^{\infty}_{k=0}\frac{B^{k}}{2^{k}B^{k}}\|h\|_{\dot {K}_{(\vec{q}/p_{0})',\vec{a}}^{-p_{0}\alpha,(p/p_{0})^{\prime}}\left(\mathbb{R}^{n}\right)}
\leq 2\|h\|_{\dot {K}_{(\vec{q}/p_{0})',\vec{a}}^{-p_{0}\alpha,(p/p_{0})^{\prime}}\left(\mathbb{R}^{n}\right)},
\end{align*}
this proves the inequality (R2), and 
$$M(\mathfrak{R} h) \leqslant \sum_{k=0}^{\infty} \frac{M^{k+1} h}{2^{k} B^{k}}=2 B \sum_{k=0}^{\infty} \frac{M^{k+1 h}}{2^{k+1} B^{k+1}}=2 B \mathfrak{R} h,$$
this implies the $\mathfrak{R}h$ is a $A_1$ weight.
\par We establish the diagonal extrapolation theorem on anisotropic mixed-norm Herz spaces using the Rubio de Francia extrapolation as follows.
\begin{theorem} \label{extropalation 1}
	Let $0<p_{0}<\infty$, and $\mathcal{F}$ be a family of pains $(f, g)$ of non-negative measurable functions. Suppose that for every
	$$w\in\left \{\mathfrak{R}h: h\in\ \dot {K}_{(\vec{q}/p_{0})',\vec{a}}^{-p_{0}\alpha,(p/p_{0})^{\prime}}\left(\mathbb{R}^{n}\right): \|h\|_{\dot {K}_{(\vec{q}/p_{0})',\vec{a}}^{-p_{0}\alpha,(p/p_{0})^{\prime}}\left(\mathbb{R}^{n}\right)}\leq 1 \right \},$$
	such that
	$$\int_{\mathbb{R}^{n}} f(x)^{p_{0}} w(x) d x \leqslant \int_{\mathbb{R}^{n}} g(x)^{p_{0}} w(x) dx<\infty ~~~(f, g) \in \mathcal F.$$
	Then $p_{0}<p, \vec{q}<\infty$, and $-\frac{1}{v} \sum_{i=1}^{n} \frac{a_{i}}{q_{i}}<\alpha<\frac{1}{p_{0}}-\frac{1}{v} \sum_{i=1}^{n} \frac{a_{i}}{q_{i}}$, for every $(f, g) \in \mathcal F$ with $g \in \dot K_{\vec{q},\vec{a}}^{\alpha,p}\left(\mathbb{R}^{n}\right)$,
$$\|f\|_{\dot K_{\vec{q},\vec{a}}^{\alpha,p}\left(\mathbb{R}^{n}\right)}\leq\|g\|_{\dot K_{\vec{q},\vec{a}}^{\alpha,p}\left(\mathbb{R}^{n}\right)}.$$
\end{theorem}
\begin{proof}
Using the properties of the iteration algorithm ${\rm (R1) (R2)}$,
\begin{align*}
	\int_{\mathbb{R}^n} f(x)^{p_0} h(x) dx &\lesssim \int_{\mathbb{R}^n} f(x)^{p_0} \mathfrak{R}h(x) dx\\
	&\lesssim \int_{\mathbb{R}^n} g(x)^{p_0} \mathfrak{R}h(x) dx\\
	&\lesssim  \left\|g^{p_0}\right\|_{\dot K_{\vec{q}/p_0,\vec{a}}^{\alpha p_0,p/p_0}\left(\mathbb{R}^{n}\right)} \left\|\mathfrak{R}h\right\|_{\dot {K}_{(\vec{q}/p_{0})',\vec{a}}^{-p_{0}\alpha,(p/p_{0})^{\prime}}\left(\mathbb{R}^{n}\right)}\\
\end{align*}
By Remark \ref{mixedpro}~(iv),we can get
\begin{align*}
\|f\|_{\dot K_{\vec{q},\vec{a}}^{\alpha,p}\left(\mathbb{R}^{n}\right)}^{p_0}&=\left\|f^{p_0}\right\|_{\dot K_{\vec{q}/p_0,\vec{a}}^{\alpha p_0,p/p_0}\left(\mathbb{R}^{n}\right)}\\
& \lesssim \sup \left\{ \int_{\mathbb{R}^n} f(x)^{p_0} h(x) dx :  \|h\|_{\dot {K}_{(\vec{q}/p_{0})',\vec{a}}^{-p_{0}\alpha,(p/p_{0})^{\prime}}\left(\mathbb{R}^{n}\right)}\leq 1 \right\}\\
&\lesssim \|g\|_{\dot K_{\vec{q},\vec{a}}^{\alpha,p}\left(\mathbb{R}^{n}\right)}^{p_0}.
\end{align*}
This implies that $\|f\|_{\dot K_{\vec{q},\vec{a}}^{\alpha,p}\left(\mathbb{R}^{n}\right)}\lesssim\|g\|_{\dot K_{\vec{q},\vec{a}}^{\alpha,p}\left(\mathbb{R}^{n}\right)}$ is right.
\end{proof}

\begin{theorem}\label{extropalation 2}
	Let $0<p_{0}<\infty$, and $\mathcal{F}$ be a family of pains $(f, g)$ of non-negative measurable functions. Suppose that for every
	$$w\in\left \{\mathfrak{R}h: h\in\  {K}_{(\vec{q}/p_{0})',\vec{a}}^{-p_{0}\alpha,(p/p_{0})^{\prime}}\left(\mathbb{R}^{n}\right): \|h\|_{ {K}_{(\vec{q}/p_{0})',\vec{a}}^{-p_{0}\alpha,(p/p_{0})^{\prime}}\left(\mathbb{R}^{n}\right)}\leq 1 \right \},$$
	such that
	$$\int_{\mathbb{R}^{n}} f(x)^{p_{0}} w(x) d x \leqslant \int_{\mathbb{R}^{n}} g(x)^{p_{0}} w(x) dx<\infty ~~~(f, g) \in \mathcal F.$$
	Then $p_{0}<p, \vec{q}<\infty$ and $-\frac{1}{v}\sum_{i=1}^n \frac{a_i}{q_i}<\alpha<\frac{1}{p_{0}}-\frac{1}{v} \sum_{i=1}^n \frac{a_{i}}{q_{i}}$, for every $(f, g) \in \mathcal F$ with $g \in K_{\vec{q},\vec{a}}^{\alpha,p}\left(\mathbb{R}^{n}\right)$, such that
$$\|f\|_{K_{\vec{q},\vec{a}}^{\alpha,p}\left(\mathbb{R}^{n}\right)}\leq\|g\|_{ K_{\vec{q},\vec{a}}^{\alpha,p}\left(\mathbb{R}^{n}\right)}.$$
\end{theorem}
\begin{proof}
	The proof is similar to the Theorem \ref{extropalation 1}, here we omit the details. 
\end{proof}

\begin{corollary} \label{ex app 1}
	Let $0< p_0<\infty$, $p_0< p,\vec{q}<\infty$ and $-\frac{1}{v}\sum_{i=1}^n \frac{a_i}{q_i}<\alpha<\frac{1}{p_{0}}-\frac{1}{v}\sum_{i=1}^n \frac{a_i}{q_i}$, suppose that every\\
	$$w\in\left \{\mathfrak{R}h: h\in\ \dot {K}_{(\vec{q}/p_{0})',\vec{a}}^{-p_{0}\alpha,(p/p_{0})^{\prime}}\left(\mathbb{R}^{n}\right): \|h\|_{\dot {K}_{(\vec{q}/p_{0})',\vec{a}}^{-p_{0}\alpha,(p/p_{0})^{\prime}}\left(\mathbb{R}^{n}\right)}\leq 1 \right \},$$
	the operator $T$:~$L_w^{p_0}(\mathbb{R}^n)\rightarrow L_w^{p_0}(\mathbb{R}^n)$, satisfies\\
	$$\int_{\mathbb{R}^n}\left|Tf(x)\right|^{p_0} w(x)dx \lesssim \int_{\mathbb{R}^n}\left|f(x)\right|^{p_0} w(x)dx.$$
	Then, for every $f\in \dot{K}_{\vec{q}, \vec{a}}^{\alpha, p}\left( \mathbb{R}^n\right)$, such that 
	$$\|Tf\|_{\dot K_{\vec{q},\vec{a}}^{\alpha,p}\left(\mathbb{R}^{n}\right)}\leq\|f\|_{\dot K_{\vec{q},\vec{a}}^{\alpha,p}\left(\mathbb{R}^{n}\right)}.$$
\end{corollary}

\begin{corollary} \label{ex app 2}
	Let $0< p_0<\infty$, $p_0< p,\vec{q}<\infty$ and $-\frac{1}{v}\sum_{i=1}^n \frac{a_i}{q_i}<\alpha<\frac{1}{p_{0}}-\frac{1}{v}\sum_{i=1}^n \frac{a_i}{q_i}$, suppose that every\\
	$$w\in\left \{\mathfrak{R}h: h\in\ \dot {K}_{(\vec{q}/p_{0})',\vec{a}}^{-p_{0}\alpha,(p/p_{0})^{\prime}}\left(\mathbb{R}^{n}\right): \|h\|_{\dot {K}_{(\vec{q}/p_{0})',\vec{a}}^{-p_{0}\alpha,(p/p_{0})^{\prime}}\left(\mathbb{R}^{n}\right)}\leq 1 \right \},$$
	the operator $T$:~$L_w^{p_0}(\mathbb{R}^n)\rightarrow L_w^{p_0}(\mathbb{R}^n)$, satisfies\\
	$$\int_{\mathbb{R}^n}\left|Tf(x)\right|^{p_0} w(x)dx \lesssim \int_{\mathbb{R}^n}\left|f(x)\right|^{p_0} w(x)dx.$$
	Then, for every $f\in \dot{K}_{\vec{q}, \vec{a}}^{\alpha, p}\left( \mathbb{R}^n\right)$, such that 
	$$\|Tf\|_{\dot K_{\vec{q},\vec{a}}^{\alpha,p}\left(\mathbb{R}^{n}\right)}\leq\|f\|_{\dot K_{\vec{q},\vec{a}}^{\alpha,p}\left(\mathbb{R}^{n}\right)}.$$
\end{corollary}

\begin{proof}
	By the conclusion of Theorem \ref{extropalation 1} and Theorem \ref{extropalation 2}, Corollary \ref{ex app 1} and Corollary \ref{ex app 2} can be immediately obtained.
\end{proof}

We also establish the off-diagonal extrapolation theorem on anisotropic mixed Herz spaces. 

\begin{theorem}
Let $0<p_{0}<\infty$, and $\mathcal{F}$ be a family of pairs $(f, g)$ of non-negative measurable functions. Suppose that for every
	$$w\in\left \{\mathfrak{R}h: h\in\ {\dot K_{\bar{\vec{q}}_2^{\prime},\vec{a}}^{-\alpha q_0, \bar{p}_2^{\prime}}\left(\mathbb{R}^{n}\right)}: \|h\|_{\dot K_{\bar{\vec{q}}_2^{\prime},\vec{a}}^{-\alpha q_0, \bar{p}_2^{\prime}}\left(\mathbb{R}^{n}\right)}\leq 1 \right \},$$
	such taht
	$$\left(\int_{\mathbb{R}^{n}} f(x)^{q_{0}} w(x) d x\right)^{\frac{1}{q_{0}}} \lesssim\left(\int_{\mathbb{R}^{n}} g(x)^{p_{0}} w(x)^{p_{0} / q_{0}} d x\right)^{\frac{1}{p_{0}}}<\infty~~~(f, g) \in \mathcal{F} .$$
	Then, $p_{0}<p_1, \vec{q_1}<p_0q_0/(q_0-p_0)$, $p_2, q_2$ satisfy $1/p_1-1/p_2=1/\vec{q}_1-\vec{q}_2=1/p_0-1/q_0$, and $\alpha$ satisfies $-\frac{1}{v}\sum_{i=1}^n \frac{a_i}{q_i}<\alpha<\frac{1}{p_{0}}-\frac{1}{v} \sum_{i=1}^n \frac{a_{i}}{q_{i}}$. Then for every $(f,g) \in \mathcal F$ with $g\in \dot K_{\vec{q}_1,\vec{a}}^{\alpha,p_1}\left(\mathbb{R}^{n}\right)$, such that
	$$\left\| f\right\|_{\dot K_{\vec{q}_2,\vec{a}}^{\alpha,p_2}\left(\mathbb{R}^{n}\right)} \lesssim \left\| g\right\|_{\dot K_{\vec{q}_1,\vec{a}}^{\alpha,p_1}\left(\mathbb{R}^{n}\right)}.$$ 
\end{theorem}

\begin{proof}
To simple calculation, one can know ${p_{0}}/{q_{0}}\cdot \left({\vec{q}_{1}}/{p_{0}}\right)^{\prime}=\left({\vec{q}_{2}}/{q_{0}}\right)^{\prime}>1$ and ${p_{0}}/{q_{0}}\cdot \left({p_{1}}/{p_{0}}\right)^{\prime}=\left({p_{2}}/{q_{0}}\right)^{\prime}>1$. Writing $\bar{p}_2=p_2/p_0$ and $\bar{\vec{q}}_2=\vec{q}_2/q_0$, Lemma \ref{boundedness HK} tell us that Hardy-Littlewood maximal operators $M$ is bounded on $\dot K_{\bar{\vec{q}}_2^{\prime},\vec{a}}^{-\alpha q_0, \bar{p}_2^{\prime}}\left(\mathbb{R}^{n}\right)$. Let $B$ is the operator norm of $M$, i.e.
$$B=\|M\|_{\dot K_{\bar{\vec{q}}_2^{\prime},\vec{a}}^{-\alpha q_0, \bar{p}_2^{\prime}}\left(\mathbb{R}^{n}\right)\rightarrow \dot K_{\bar{\vec{q}}_2^{\prime},\vec{a}}^{-\alpha q_0, \bar{p}_2^{\prime}}\left(\mathbb{R}^{n}\right)}.$$
Using the same way, for any measurable function $h$, we define
$$\mathfrak{R}h:=\sum_{k=0}^{\infty} \frac{M^{k} h}{2^{k} B^{k}},$$ 
where $M^{k}$ is the $k$th iterations of $M$ and we denote by $M^{\circ} h=h$. \\
Similar to (R1), (R2), and (R3), we have
\begin{align*}
    &h(x) \leq \mathfrak{R}h(x),  \tag{R4}\\
	&\|\mathfrak{R}h\|_{\dot K_{\bar{\vec{q}}_2^{\prime},\vec{a}}^{-\alpha q_0, \bar{p}_2^{\prime}}\left(\mathbb{R}^{n}\right)}\leq 2 \|h\|_{\dot K_{\bar{\vec{q}}_2^{\prime},\vec{a}}^{-\alpha q_0, \bar{p}_2^{\prime}}\left(\mathbb{R}^{n}\right)}',  \tag{R5}\\
	&[\mathfrak{R} h]_{A_{1}} \leq 2 B. \tag{R6}
\end{align*}
For any measurable function $h$, and any measurable function pairs $(f,g) \in \mathcal{F}$, by inequality (R4) and (R5), we can get 
\begin{align*}
	\int_{\mathbb{R}^n} f(x)^{q_0} h(x) dx &\lesssim \int_{\mathbb{R}^n} f(x)^{q_0} \mathfrak{R}h(x) dx\\
	&\lesssim \left(\int_{\mathbb{R}^n} g(x)^{p_0} \left(\mathfrak{R}h(x)\right)^{p_0/q_0} dx \right)^{q_0/p_0}\\
	&\lesssim  \left\|g^{p_0}\right\|_{\dot K_{\vec{q}_1/p_0,\vec{a}}^{\alpha p_0,p_1/p_0}\left(\mathbb{R}^{n}\right)}^{q_0/p_0} \left\| \left(\mathfrak{R}h(x)\right)^{p_0/q_0}\right\|_{\dot {K}_{(\vec{q}_1/p_{0})',\vec{a}}^{-p_{0}\alpha,(p_1/p_{0})^{\prime}}\left(\mathbb{R}^{n}\right)}^{q_0/p_0}\\
	&\lesssim  \left\|g^{p_0}\right\|_{\dot K_{\vec{q}_1/p_0,\vec{a}}^{\alpha p_0,p_1/p_0}\left(\mathbb{R}^{n}\right)}^{q_0/p_0} \left\| \mathfrak{R}h(x)\right\|_{\dot K_{\bar{\vec{q}}_2^{\prime},\vec{a}}^{-\alpha q_0, \bar{p}_2^{\prime}}\left(\mathbb{R}^{n}\right)}\\
	&\lesssim \left\| g\right\|_{\dot K_{\vec{q}_1,\vec{a}}^{\alpha,p_1}\left(\mathbb{R}^{n}\right)}^{q_0} \|h\|_{\dot K_{\bar{\vec{q}}_2^{\prime},\vec{a}}^{-\alpha q_0, \bar{p}_2^{\prime}}\left(\mathbb{R}^{n}\right)}\\
	& \lesssim \left\| g\right\|_{\dot K_{\vec{q}_1,\vec{a}}^{\alpha,p_1}\left(\mathbb{R}^{n}\right)}^{q_0}.
\end{align*}
Furthermore,
\begin{align*}
	\left\| f\right\|_{\dot K_{\vec{q}_2,\vec{a}}^{\alpha,p_2}\left(\mathbb{R}^{n}\right)}^{q_0} &=\left\| f^{q_0} \right\|_{\dot K_{\vec{q}_2/q_0,\vec{a}}^{\alpha,p_2/p_0}\left(\mathbb{R}^{n}\right)}\\
	&\lesssim \sup\left\{ \int_{\mathbb{R}^n}|f(x)|^{q_0} h(x) dx : \|h\|_{\dot {K}_{(\vec{q}_2/p_{0})',\vec{a}}^{-p_{0}\alpha,(p_2/p_{0})^{\prime}}\left(\mathbb{R}^{n}\right)}\leq 1 \right\}\\
	&\lesssim\| g \|_{{\dot K_{\vec{q}_{1}, \vec{a}}^{\alpha, p_{1}}}\left(\mathbb{R}^{n}\right)}^{q_0}.
\end{align*}
This finishes the proof.
\end{proof}

\begin{theorem}
	Let $0<p_{0}<\infty$, and $\mathcal{F}$ be a family of pairs $(f, g)$ of non-negative measurable functions. Suppose that for every
	$$w\in\left \{\mathfrak{R}h: h\in\ {\dot K_{\bar{\vec{q}}_2^{\prime},\vec{a}}^{-\alpha q_0, \bar{p}_2^{\prime}}\left(\mathbb{R}^{n}\right)}: \|h\|_{\dot K_{\bar{\vec{q}}_2^{\prime},\vec{a}}^{-\alpha q_0, \bar{p}_2^{\prime}}\left(\mathbb{R}^{n}\right)}\leq 1 \right \},$$
	such that 
	$$\left(\int_{\mathbb{R}^{n}} f(x)^{q_{0}} w(x) d x\right)^{\frac{1}{q_{0}}} \lesssim\left(\int_{\mathbb{R}^{n}} g(x)^{p_{0}} w(x)^{p_{0} / q_{0}} d x\right)^{\frac{1}{p_{0}}}<\infty~~~(f, g) \in \mathcal{F} .$$
	Then, $p_{0}<p_1, \vec{q_1}<p_0q_0/(q_0-p_0)$, $p_2, q_2$ satisfy $1/p_1-1/p_2=1/\vec{q}_1-\vec{q}_2=1/p_0-1/q_0$, and $\alpha$ satisfies $-\frac{1}{v}\sum_{i=1}^n \frac{a_i}{q_i}<\alpha<\frac{1}{p_{0}}-\frac{1}{v} \sum_{i=1}^n \frac{a_{i}}{q_{i}}$. Then for every $(f,g) \in \mathcal F$ with $g\in  K_{\vec{q}_1,\vec{a}}^{\alpha,p_1}\left(\mathbb{R}^{n}\right)$, such that
	$$\left\| f\right\|_{ K_{\vec{q}_2,\vec{a}}^{\alpha,p_2}\left(\mathbb{R}^{n}\right)} \lesssim \left\| g\right\|_{ K_{\vec{q}_1,\vec{a}}^{\alpha,p_1}\left(\mathbb{R}^{n}\right)}.$$ 
\end{theorem}

The off-diagonal extrapolation theorems also have the following two corollaries about operators with weighted boundedness.
\begin{corollary} \label{frac 1}
	Let $0<p_{0}<\infty$, $p_{0}<p_1, \vec{q_1}<p_0q_0/(q_0-p_0)$, $p_2, q_2$ satisfy $1/p_1-1/p_2=1/\vec{q}_1-\vec{q}_2=1/p_0-1/q_0$, and $\alpha$ satisfies $-\frac{1}{v}\sum_{i=1}^n \frac{a_i}{q_i}<\alpha<\frac{1}{p_{0}}-\frac{1}{v} \sum_{i=1}^n \frac{a_{i}}{q_{i}}$. Suppose that
	$$w\in\left \{\mathfrak{R}h: h\in\ {\dot K_{\bar{\vec{q}}_2^{\prime},\vec{a}}^{-\alpha q_0, \bar{p}_2^{\prime}}\left(\mathbb{R}^{n}\right)}: \|h\|_{\dot K_{\bar{\vec{q}}_2^{\prime},\vec{a}}^{-\alpha q_0, \bar{p}_2^{\prime}}\left(\mathbb{R}^{n}\right)}\leq 1 \right \},$$
	the operator $T:~L^{p_0}_w\left(\mathbb{R}^{n}\right)\rightarrow L^{q_0}_w$, satisfies
	$$\left(\int_{\mathbb{R}^n}\left|Tf(x)\right|^{q_0}w(x)dx\right)^{1/q_0}\lesssim \left(\int_{\mathbb{R}^n} \left|f(x) \right|^{p_0} w(x) dx \right)^{1/p_0}. $$
	Then, for every $f\in \dot{K}_{\vec{q}_1, \vec{a}}^{\alpha, p_1}\left( \mathbb{R}^n\right)$, such that 
	$$\left\| Tf\right\|_{ K_{\vec{q}_2,\vec{a}}^{\alpha,p_2}\left(\mathbb{R}^{n}\right)} \lesssim \left\| f\right\|_{ K_{\vec{q}_1,\vec{a}}^{\alpha,p_1}\left(\mathbb{R}^{n}\right)}.$$ 
\end{corollary} 

\begin{corollary} \label{frac 2}
	Let $0<p_{0}<\infty$, $p_{0}<p_1, \vec{q_1}<p_0q_0/(q_0-p_0)$, $p_2, q_2$ satisfy $1/p_1-1/p_2=1/\vec{q}_1-\vec{q}_2=1/p_0-1/q_0$, and $\alpha$ satisfies $-\frac{1}{v}\sum_{i=1}^n \frac{a_i}{q_i}<\alpha<\frac{1}{p_{0}}-\frac{1}{v} \sum_{i=1}^n \frac{a_{i}}{q_{i}}$. Suppose that
	$$w\in\left \{\mathfrak{R}h: h\in\ {\dot K_{\bar{\vec{q}}_2^{\prime},\vec{a}}^{-\alpha q_0, \bar{p}_2^{\prime}}\left(\mathbb{R}^{n}\right)}: \|h\|_{\dot K_{\bar{\vec{q}}_2^{\prime},\vec{a}}^{-\alpha q_0, \bar{p}_2^{\prime}}\left(\mathbb{R}^{n}\right)}\leq 1 \right \},$$
	the operator $T:~L^{p_0}_w\left(\mathbb{R}^{n}\right)\rightarrow L^{q_0}_w$, satisfies
	$$\left(\int_{\mathbb{R}^n}\left|Tf(x)\right|^{q_0}w(x)dx\right)^{1/q_0}\lesssim \left(\int_{\mathbb{R}^n} \left|f(x) \right|^{p_0} w(x) dx \right)^{1/p_0}. $$
	Then, for every $f\in \dot{K}_{\vec{q}_1, \vec{a}}^{\alpha, p_1}\left( \mathbb{R}^n\right)$, such that 
	$$\left\| Tf\right\|_{ K_{\vec{q}_2,\vec{a}}^{\alpha,p_2}\left(\mathbb{R}^{n}\right)} \lesssim \left\| f\right\|_{ K_{\vec{q}_1,\vec{a}}^{\alpha,p_1}\left(\mathbb{R}^{n}\right)}.$$ 
\end{corollary} 

The extrapolation theorem will be used to address various operator boundedness issues, such as singular integral operators and their commutators, fractional maximal operators, fractional integral operators and its commutator, and so on. As long as the requirements of our theorems are met for these operators, our theorems can really be applied to the boundedness of numerous additional integral operators and their commutators. 
\par 
Commutators are well-known to be significant operators in harmonic analysis. Note that the commutator, which is generated by a locally integrable function $b$ and an integral operator $T$, has the definition $[b, T]:=b T-T b$. In 1976, Coifman, Rochberg and Weiss \cite{coifman1976factorization} proved that the commutators of %Calder$\acute{\rm o}$n-Zygmund singular integral operators are bounded on Lebesgue spaces if and only if the symbol $b$ is in BMO. 
The Calder$\acute{\rm o}$n-Zygmund singular integral operators are basic integral operators in harmonic analysis. The classical Calder$\acute{\rm o}$n-Zygmund operator $T$ is a initially $L^2(\mathbb{R}^{n})$ bounded operator with the associated standard kernel $K$, that is to say, functions $K(x, y)$ defined on $\mathbb{R}^{n} \times \mathbb{R}^{n} \backslash\left\{(x, x): x \in \mathbb{R}^{n} \right\}$ satisfying
\begin{enumerate}
	\item [\rm(i)] the size condition:
$$
|K(x, y)| \leq \frac{A}{|x-y|^{n}},
$$
for some constant $A>0$;
    \item [\rm(ii)]  the regularity conditions: for some $\delta>0$, 
$$
\left|K(x, y)-K\left(x^{\prime}, y\right)\right| \leq \frac{A\left|x-x^{\prime}\right|^{ \delta}}{\left(|x-y|+\left|x^{\prime}-y\right|\right)^{n+\delta}}
$$
holds whenever $\left|x-x^{\prime}\right| \leq \frac{1}{2} \max \left(|x-y|,\left|x^{\prime}-y\right|\right)$ and
$$
\left|K(x, y)-K\left(x, y^{\prime}\right)\right| \leq \frac{A\left|y-y^{\prime}\right|^{ \delta}}{\left(|x-y|+\left|x-y^{\prime}\right|\right)^{n+\delta}}
$$
holds whenever $\left|y-y^{\prime}\right| \leq \frac{1}{2} \max \left(|x-y|,\left|x-y^{\prime}\right|\right)$.
\end{enumerate}
Then $T$ can be represented as
$$
T f(x)=\int_{\mathbb{R}^{n}} K(x, y) f(y) d y, \quad x \notin \operatorname{supp} f,
$$
which is obvious a $L^{\vec{p}}\left(\mathbb{R}^{n}\right)$-bounded operator.\\
Now we recall the definition of BMO. A locally integrable function $f$ belongs to BMO if and only if
$$
\|f\|_{\mathrm{BMO}}:=\sup _{B} \frac{\left\|\left(f-f_{B}\right) \chi_{B}\right\|_{L^{1}\left(\mathbb{R}^{n}\right)}}{\left\|\chi_{B}\right\|_{L^{1}\left(\mathbb{R}^{n}\right)}}<\infty,
$$
where $f_{B}=\frac{1}{|B|} \int_{B} f(x) d x$.
The weighted estimates for the operator $T$ and its commutator $[b,T]$ can see \cite{grafakos2008classical}.
\begin{lemma}[\cite{grafakos2008classical}] \label{CZ boundedness}
	Let the Calder\'on-Zygmund singular integral operator $K$ be defined as above. Then for all $1<p<\infty$ and $w \in A_{p}$, we have
$$
\|T f\|_{L_{w}^{p}\left(\mathbb{R}^{n}\right)} \lesssim\|f\|_{L_{w}^{p}\left(\mathbb{R}^{n}\right)},
$$
and for $b \in \mathrm{BMO}$, also have
$$
\left\|[b, K] f \right\|_{L_{w}^{p}\left(\mathbb{R}^{n}\right)} \lesssim\|f\|_{L_{w}^{p}\left(\mathbb{R}^{n}\right)} .
$$
\end{lemma}
By Theorem \ref{ex app 1}, Theorem \ref{ex app 2} and Lemma \ref{CZ boundedness} , we can get the following result.
\begin{theorem}
	Let the Calder\'on-Zygmund singular integral operator $K$ be defined as above. Then for $1< p,\vec{q}<\infty$ and $-\frac{1}{v}\sum_{i=1}^n \frac{a_i}{q_i}<\alpha<1-\frac{1}{v} \sum_{i=1}^n \frac{a_{i}}{q_{i}}$, and $f\in \dot K_{\vec{q}, \vec{a}}^{\alpha, p}$, then
	$$
\|T f\|_{\dot{K}_{\vec{q}, \vec{a}}^{\alpha, p}\left(\mathbb{R}^{n}\right)} \lesssim\|f\|_{\dot{K}_{\vec{q}, \vec{a}}^{\alpha, p}\left(\mathbb{R}^{n}\right)},
$$
and for $b \in \mathrm{BMO}$, also have
$$
\|[b, K] f\|_{\dot{K}_{\vec{q}, \vec{a}}^{\alpha, p}\left(\mathbb{R}^{n}\right)} \lesssim\|f\|_{\dot{K}_{\vec{q}, \vec{a}}^{\alpha, p}\left(\mathbb{R}^{n}\right)}.
$$
\end{theorem}
\begin{proof}
	By Theorem \ref{ex app 1}, Theorem \ref{ex app 2} and Lemma \ref{CZ boundedness} , we can easily obtain the above result. In fact, we only show that $p_0 > 1$ which is sufficiently close to 1 such that $-\frac{1}{v}\sum_{i=1}^n \frac{a_i}{q_i}<\alpha<\frac{1}{p_{0}}-\frac{1}{v} \sum_{i=1}^n \frac{a_{i}}{q_{i}}$ close to $-\frac{1}{v}\sum_{i=1}^n \frac{a_i}{q_i}<\alpha<1-\frac{1}{v} \sum_{i=1}^n \frac{a_{i}}{q_{i}}$.
\end{proof}
\par 
Let $f \in L_{\mathrm{loc}}\left(\mathbb{R}^{n}\right)$, then the fractional maximal operator $M_{\alpha}$ and the fractional integral operator $I_{\alpha}$ are defined by
$$
M_{\alpha} f(x):=\sup _{r>0} \frac{1}{|B(x, r)|^{n-\alpha}} \int_{B}|f(y)| d y,~~ 0 \leq \alpha<n,
$$
and
$$
I_{\alpha} f(x):=\int_{\mathbb{R}^{n}} \frac{f(y)}{|x-y|^{n-\alpha}} d y, ~~0<\alpha<n
$$
respectively. If $\alpha=0$, then $M_{0} \equiv M$ is the Hardy-Littlewood maximal operator.\\
In order to state the results in this section, we need the $A_{p, q}$ weight.
\begin{definition}
	Suppose $0<\alpha<n, 1<p<n / \alpha$ and $1 / p-1 / q=\alpha / n$, we say a locally integrable function $w: \mathbb{R}^{n} \rightarrow(0, \infty)$ belongs to $A_{p, q}$ if and only if
$$
\sup _{B} \frac{1}{|B|} \int_{B} w(z) d z\left(\frac{1}{|B|} \int_{B} w(z)^{-p^{\prime} / q} d z\right)^{q / p^{\prime}}<\infty.
$$
\end{definition}

Note that this is equivalent to $w \in A_{r}$, where $r=1+q / p^{\prime}$, so, in particular if $w \in A_{1}$, then $w \in A_{p, q}$.

The weighted boundedness of the fractional maximal operator $M_{\alpha}$, the fractional integral operator $I_{\alpha}$ and its commutator $\left[b, I_{\alpha}\right]$ was obtained in \cite{muckenhoupt}.

\begin{lemma} \label{frac weighted}
	Let $0<\alpha<n$ and $1<p<q<\infty$ with $1 / p-1 / q=\alpha / n$. If $w \in A_{p, q}$ and $b \in \mathrm{BMO}$, then
$$
\begin{aligned}
&\left(\int_{\mathbb{R}^{n}} M_{\alpha} f(x)^{q} w(x) d x\right)^{1 / q} \lesssim\left(\int_{\mathbb{R}^{n}}|f(x)|^{p} w(x)^{p / q} d x\right)^{1 / p} ,\\
&\left(\int_{\mathbb{R}^{n}}\left|I_{\alpha} f(x)\right|^{q} w(x) d x\right)^{1 / q} \lesssim\left(\int_{\mathbb{R}^{n}}|f(x)|^{p} w(x)^{p / q} d x\right)^{1 / p} ,\\
&\left(\int_{\mathbb{R}^{n}}\left|\left[b, I_{\alpha}\right] f(x)\right|^{q} w(x) d x\right)^{1 / q} \lesssim\left(\int_{\mathbb{R}^{n}}|f(x)|^{p} w(x)^{p / q} d x\right)^{1 / p}.
\end{aligned}
$$
\end{lemma}
We give the fact that the class $A_{p, q}$ with slightly different from the typical definition, with $w$ being replaced by $w^q$. Despite being unusual, our formulation is more appropriate for our needs.\\
By Theorem \ref{frac 1} and Lemma \ref{frac weighted}, we can get the following result.
\begin{theorem}
	Let $0<\alpha<n, 1<p<n / \alpha, 1 / p_{1}-1 / p_{2}=1 / \vec{q}_1-1 / \vec{q}_2=\alpha / n, b \in$ BMO and $\hat{\alpha}$ satisfies
$$
-\frac{1}{v} \sum_{i=1}^{n} \frac{a_{i}}{q_{2i}}<\hat{\alpha}<\frac{1}{p_{0}}-\frac{1}{v} \sum_{i=1}^n \frac{a_{i}}{q_{2i}}-\alpha .
$$
Then $M_{\alpha}, I_{\alpha}$ and $\left[b, I_{\alpha}\right]$ can be extended to be bounded operators from $\dot{K}_{\vec{q}_{1},\vec{a}}^{\hat{\alpha}, p_{1}}\left(\mathbb{R}^{n}\right)$ to $\dot{K}_{\vec{q}_{2},\vec{a}}^{\hat{\alpha}, p_{2}}\left(\mathbb{R}^{n}\right)$.
\end{theorem}

Recall the definition of Littlewood-Paley operators. Suppose $\psi$ is integrable on $\mathbb{R}^{n}$ and
\begin{enumerate}
	\item [{\rm (i)}] $\int_{\mathbb{R}^{n}} \psi(x) d x=0$,
	\item [{\rm (ii)}] $|\psi(x)| \leq C(1+|x|)^{-(n+\alpha)}$, for some $\alpha>0$,
	\item [{\rm (iii)}] $\int_{\mathbb{R}^{n}}|\psi(x+y)-\psi(x)| d x \leq C \left|y\right|^{\gamma}$, for all $y \in \mathbb{R}^{n}$, for some $\gamma>0$.
\end{enumerate}
Let $\psi_{t}(x)=t^{-n} \psi(x / t)$ with $t>0$ and $x \in \mathbb{R}^{n}$. For $f$ in $L^{2}\left(\mathbb{R}^{n}\right)$ with compact support, the Littlewood-Paley $g$-function of $f$ is defined by
$$
g_{\psi}(f)(x)=\left\{\int_{0}^{\infty}\left|f * \psi_{t}(x)\right|^{2} \frac{d t}{t}\right\}^{1 / 2} ;
$$
the Lusin area function of $f$ is defined by
$$
S_{\psi, a}(f)(x)=\left(\frac{1}{a^{n}\left|B_{0}\right|} \int_{\Gamma_{a}(x)}\left|f * \psi_{t}(x)\right|^{2} t^{-n} d y \frac{d t}{t}\right)^{1 / 2},
$$
where $\left|B_{0}\right|$ is the Lebesgue measure of the unit ball $B_{0}$ of $\mathbb{R}^{n}$, and $\Gamma_{a}(x)=\{(y, t) \in$ $\left.\mathbb{R}_{+}^{n+1}:|x-y|<a t\right\}$; and the Littlewood-Paley $g_{\lambda}^{*}$-function of $f$ is defined by
$$
g_{\psi, \lambda}^{*}(f)(x)=\left\{\int_{0}^{\infty} \int_{\mathbb{R}^{n}} \frac{\left|f * \psi_{t}(y)\right|^{2}}{\left(1+\frac{|x-y|}{t}\right)^{2 \lambda}} t^{-n} d y \frac{d t}{t}\right)^{1 / 2}.
$$
It is well known that, in classical situation where $\psi$ is related to the gradient of the Poisson kernel, all those functions have $L^p(\mathbb{R}^n)$ norms equivalent to $L^p(\mathbb{R}^n)$ norm of function $f$, that is, $\|g(f)\|_{L^p\left(\mathbb{R}^n\right)}\lesssim \left\|f\right\|_{L^p\left(\mathbb{R}^n\right)}$ for   $1<p<\infty$\cite{grafakos2008classical}. This equivalence actually holds for any general Littlewood-Paley function defined in terms of such $\psi.$ Some researchers also show that those Littlewood-Paley function can also characterize Herz spaces ${\dot{K}_{q}^{\alpha, p}}(\mathbb{R}^n)$ and ${K_{q}^{\alpha, p}}(\mathbb{R}^n)$ when $f\in {\dot{K}_{q}^{\alpha, p}}(\mathbb{R}^n)$ or $ f\in {K_{q}^{\alpha, p}}(\mathbb{R}^n)$ for $1\leq p\leq \infty$\cite{li1996boundedness}. In 2010, Ward prove the Littlewood-Paley characterization of mixed Lebesgue spaces\cite{ward2010new}. 
We will show that this equivalence actually hold also for the spaces ${\dot{K}_{\vec{q},\vec{a}}^{\alpha, p}}\left(\mathbb{R}^n\right)$ and ${K_{\vec{q},\vec{a}}^{\alpha, p}}\left(\mathbb{R}^n\right)$ with $1\leq p \leq \infty$,  $1<\vec{q}<\infty$ and $-\frac{1}{v}\sum_{i=1}^n \frac{a_i}{q_i}<\alpha<1-\frac{1}{v}\sum_{i=1}^n \frac{a_i}{q_i},$ within the ranges of any of $p$ and $\vec{q}$, the corresponding mixed norm Herz spaces are Banach spaces.

It is known that for any general Littlewood-Paley function associated with $\psi$, the inequality hold for all $a, \lambda>0$\cite{torchinsky123real}
$$S_{\psi, a}(f)(x)\leq Cg_{\psi, \lambda}^{*}\left(f\right)(x),\quad \quad x\in \mathbb{R}^n.$$
The next inequality also true\cite{wilson1988note}, for any compactly supported  $\psi$ satisfying (i), (ii) and (iii), there is another compactly supported radial function $\varrho$ of this type, such that
$$g_{\psi}(f)(x)\leq CS_{\varrho, 2}(f)(x), \quad \quad~~~~~ x\in \mathbb{R}^n.$$

\begin{theorem}
	Let $1\leq p\leq \infty,$ $1<\vec{q}<\infty,$ $-\frac{1}{v}\sum_{i=1}^n \frac{a_i}{q_i}< \alpha <1-\frac{1}{v}\sum_{i=1}^n \frac{a_i}{q_i}$, $\lambda>3n/2$ and let $\psi$ satisfy $(i),(ii),(iii).$ Then following statements are equivalent.
	\begin{align*}
		\left\|f\right\|_{\dot K_{\vec{q},\vec{a}}^{\alpha,p}(\mathbb R^n)}\lesssim \left\|g_{\psi}(f)\right\|_{\dot K_{\vec{q},\vec{a}}^{\alpha,p}(\mathbb R^n)} \lesssim \left\|S_{\psi, a}(f)\right\|_{\dot K_{\vec{q},\vec{a}}^{\alpha,p}(\mathbb R^n)} \lesssim \left\|g_{\psi,\lambda}^{*}(f)\right\|_{\dot K_{\vec{q},\vec{a}}^{\alpha,p}(\mathbb R^n)} \lesssim \left\|f\right\|_{\dot K_{\vec{q},\vec{a}}^{\alpha,p}(\mathbb R^n)}.
	\end{align*}
	There is a similar result for $K_{\vec{q}}^{\alpha, p}\left(\mathbb{R}^n\right)$.
\end{theorem}
\begin{proof}
		In this proof, we mainly prove that
				\begin{equation}
			\left\|g_{\psi,\lambda}^{*}(f)\right\|_{\dot{K}_{\vec{q},\vec{a}}^{\alpha, p}\left(\mathbb{R}^n\right)}\leq C \left\|f\right\|_{\dot{K}_{\vec{q},\vec{a}}^{\alpha, p}\left(\mathbb{R}^n\right)} \label{1},
		\end{equation}
		and
		\begin{equation}
			\left\|g_{\psi,\lambda}^{*}(f)\right\|_{\dot{K}_{\vec{q},\vec{a}}^{\alpha, p}\left(\mathbb{R}^n\right)}\leq C \left\|f\right\|_{\dot{K}_{\vec{q},\vec{a}}^{\alpha, p}\left(\mathbb{R}^n\right)} \label{2}.
		\end{equation}
Once this is done, we obtain
\begin{equation}
			\left\|g_{\psi}(f)\right\|_{\dot{K}_{\vec{q},\vec{a}}^{\alpha, p}\left(\mathbb{R}^n\right)}\leq C\left\|S_{\psi,a}(f)\right\|_{\dot{K}_{\vec{q},\vec{a}}^{\alpha, p}\left(\mathbb{R}^n\right)}
			\leq C\left\|g_{\psi,\lambda}^{*}(f)\right\|_{\dot{K}_{\vec{q},\vec{a}}^{\alpha, p}\left(\mathbb{R}^n\right)}\leq C \left\|f\right\|_{\dot{K}_{\vec{q},\vec{a}}^{\alpha, p}\left(\mathbb{R}^n\right)} \label{3}.
		\end{equation}
	In fact, inequalities (\ref{1}) and (\ref{2}) can easily obtain via similar estimation of Theorem \ref{boundedness HL}. \\
	On the other hand, through standard arguments, using the H\"older's inequality and (\ref{3}),  we have
	\begin{align*}\left|\int_{\mathbb{R}^{n}} f_{1}(x) \overline{f_{2}(x)} d x\right| 
	& \leqslant C \int_{\mathbb{R}^{n}} \int_{0}^{\infty}\left|\psi_{t} * f_{1}(x)\right| \left|\psi_{t} * f_{2}(x)\right| \frac{d t}{t} d x \\ & \leqslant C \int_{\mathbb{R}^{n}} g_{\psi}\left(f_{1}\right)(x)  g_{\psi}\left(f_{2}\right)(x) d x \\ 
	& \leqslant C\left\|g_{\psi}\left(f_{1}\right)\right\|_{H\dot K_{\vec{q}, \vec{a}}^{\alpha,p}\left(\mathbb{R}^n\right)} \left\|g_{\psi}\left(f_{2}\right)\right\|_{H\dot K_{\vec{q}^{\prime}, \vec{a}}^{-\alpha, p^{\prime}}\left(\mathbb{R}^n\right)} \\ 
	& \leqslant C\left\|g_{\psi}\left(f_{1}\right)\right\|_{H\dot K_{\vec{q}, \vec{a}}^{\alpha,p}\left(\mathbb{R}^n\right)}. 
	\end{align*}
	Taking the $\sup_{\left\|f_2 \right\|_{H\dot K_{\vec{q}^{\prime}, \vec{a}}^{-\alpha, p^{\prime}}\left(\mathbb{R}^n\right)} \leq 1}$, then 
	\begin{align*}
		\left\|f\right\|_{H\dot K_{\vec{q}, \vec{a}}^{\alpha,p}\left(\mathbb{R}^n\right)} \leq C \left\|g_{\psi}(f)\right\|_{H\dot K_{\vec{q}, \vec{a}}^{\alpha,p}\left(\mathbb{R}^n\right)}.
	\end{align*}
	Combine with the above two parts, this proof is completed.
\end{proof}

\section{Anisotropic Mixed-Norm Herz-Hardy spaces}
In this section, we mainly study the anisotropic mixed-norm Herz-Hardy spaces, which is the appropriate substitute for anisotropic mixed-norm Herz spaces on some index range. Under this framework, we establish atomic decomposition and molecular decomposition. As an application, we resolve the bounded map issue of classical Calder\'on-Zygmund operators from $H\dot K_{\vec{q}, \vec{a}}^{\alpha, p}(\mathbb R^n)$(or $HK_{\vec{q}, \vec{a}}^{\alpha, p}(\mathbb R^n)$) to $\dot K_{\vec{q}, \vec{a}}^{\alpha, p}(\mathbb R^n)$ (or $K_{\vec{q}, \vec{a}}^{\alpha, p}(\mathbb R^n)$ ).
\par 
Now we give the definition of anisotropic mixed-norm Herz-Hardy spaces $H \dot{K}_{\vec{q},\vec{a}}^{\alpha, p}\left(\mathbb{R}^{n}\right)$ and $H K_{\vec{q}, \vec{a}}^{\alpha, p}\left(\mathbb{R}^{n}\right).$ First, we recall several maximal operators.
\par
Denote by $\mathcal S \left(\mathbb{R}^n\right)$ the  set of all Schwarz function ,and $\mathcal S^{\prime} \left(\mathbb{R}^n\right)$ the dual spaces of $\mathcal S \left(\mathbb{R}^n\right)$. For any $N \in \mathbb Z_+$,
$$\mathcal{S}_N \left(\mathbb{R}^{n}\right):=\left\{\varphi \in \mathcal{S} \left(\mathbb{R}^{n}\right):~\|\varphi\|_{\mathcal{S}_N \left(\mathbb{R}^{n}\right)}:= \sup_{x \in \mathbb{R}^n}\left[\langle x\rangle^N_{\vec{a}} \sup_{|\alpha|\leq N} |\partial^{\alpha} \varphi(x)|\right]\leq 1 \right\}.$$
In what follows, for any $\varphi \in \mathcal{S} \left(\mathbb{R}^{n}\right)$ and $t\in (0,\infty)$, let $\varphi_t(\cdot):= t^{-v}\varphi(t^{-\vec{a}} \cdot).$
\par 
%For any $\vec{p}:=(p_1, p_2,\dots, p_n) \in (0,\infty)^n$, we always let
%\begin{equation}
%	p_-:=\min \{1, p_1, \dots, p_n\}~~\text{and}~~p_+:=\max \{1, p_1, \dots, p_n\}\label{p}
%\end{equation}
\begin{definition}
	Let  $\varphi \in \mathcal S \left(\mathbb{R}^{n}\right)$ and $f \in \mathcal S^{\prime}\left(\mathbb{R}^{n}\right)$. The non-tangential maximal function $M \varphi(f) $, with respect to $\varphi$, is defined by setting, for any $x \in \mathbb{R}^{n}$,
$$M_{\varphi}(f)(x):=\sup _{y \in B_{\vec{a}}(x, t), t \in(0, \infty)}\left|f * \varphi_{t}(y)\right|.$$
The radial maximal function of $f$  with respect to  $\varphi$  is defined as
$$M_{N}^{0} f(x):=\sup _{k \in \mathbb{Z}}\left|f * \varphi_{k}\right|.$$
Moreover. for any given $N \in \mathbb{N}$, the non-tangential grand maximal function  $M_{N}(f)$  of  $f \in \mathcal S^{\prime}\left(\mathbb{R}^{n}\right)$  is defined by setting, for any  $x \in \mathbb{R}^{n}$,
$$M_{N}(f)(x):=\sup _{\varphi \in \mathcal S_{N}\left(\mathbb{R}^{n}\right)} M_{\varphi}(f)(x).$$ 
\end{definition}

\begin{remark} \label{equivalent MN}
	\rm { From [\cite{bownik2003anisotropic}. Proposition 3.10]. We know that radial maximal functions and grand maximal functions are pointwise equivalent i.e. for every  $N \geqslant 0$, there is a constant $C=C(N)$  So that. for all $f \in \mathcal{S}^{\prime}$  and $x \in \mathbb{R}^{n}$.
\begin{align*}
M_{N}^{0} f(x) \leq M_{N} f(x) \leq C M_{N}^{0} f(x). 
\end{align*}}
\end{remark}
\par 
In 2017, Cleanthous et al. introduced anisotropic mixed-norm Hardy spaces as follows.
\begin{definition}
	 Let $\vec{a} \in[1, \infty)^{n}, \vec{p} \in(0, \infty)^{n}$ we denote by $p_-:=\min \left(1, p_{1}, \ldots, p_{n}\right)$ and
\begin{equation}
	N_{\vec{p}}:=\left[v \frac{a_{+}}{a_{-}}\left(1+\frac{2}{p_-}\right)+v+2 a_{+}\right]+1 \label{N number}
\end{equation}
We will say that a distribution $f \in \mathcal{S}^{\prime}$ belongs to the anisotropic mixed-norm Hardy space $H_{\vec{a}}^{\vec{p}}(\mathbb R^n)$ when ${M}_{N_{\vec{p}}} f \in L^{\vec{p}}$. The map $\|f\|_{H_{\vec{a}}^{\vec{p}}}:=\left\|{M}_{N_{\vec{p}}} f\right\|_{\vec{p}}$ is the quasi-norm of $H_{\vec{a}}^{\vec{p}}(\mathbb R^n).$
\end{definition}
\begin{definition}
	Let $\vec{a}\in [1, \infty)^n$, $\alpha \in \mathbb{R}$, $0<p<\infty$, $0<\vec{q} \leqslant \infty $, $N \geqslant N_{\vec{q}} $. The anisotropic mixed-norm homogeneous Herz-Hardy space $H \dot K_{\vec{q}, \vec{a}}^{\alpha, p}(\mathbb{R}^{n})$ is defined by
	$$H \dot K_{\vec{q}, \vec{a}}^{\alpha, p}\left(\mathbb{R}^{n}\right):=\left\{f \in \mathcal {S}^{\prime}\left(\mathbb{R}^{n}\right):~\left\|f\right\|_{H\dot K_{\vec{q}, \vec{a}}^{\alpha, p}\left(\mathbb{R}^{n}\right)}=\left\|M_N f\right\|_{\dot K_{\vec{q}, \vec{a}}^{\alpha, p}\left(\mathbb{R}^{n}\right)}< \infty\right\}.$$
\end{definition}

\begin{definition}
	Let $\vec{a}\in [1, \infty)^n$, $\alpha \in \mathbb{R}$, $0<p<\infty$, $0<\vec{q} \leqslant \infty $, $N \geqslant N_{\vec{q}} $. The anisotropic mixed-norm non-homogeneous Herz-Hardy space $H K_{\vec{q}, \vec{a}}^{\alpha, p}(\mathbb{R}^{n})$ is defined by
	$$H  K_{\vec{q}, \vec{a}}^{\alpha, p}\left(\mathbb{R}^{n}\right):=\left\{f \in \mathcal {S}^{\prime}\left(\mathbb{R}^{n}\right):~\left\|f\right\|_{H K_{\vec{q}, \vec{a}}^{\alpha, p}\left(\mathbb{R}^{n}\right)}=\left\|M_N f\right\|_{K_{\vec{q}, \vec{a}}^{\alpha, p}\left(\mathbb{R}^{n}\right)}< \infty\right\}.$$
\end{definition}

\begin{remark}
	\rm{
	\begin{enumerate}
		\item [(i)] The quasi-norm of anisotropic mixed-norm Herz-Hardy spaces depend on $N$, however, we know that it is independent of the choice of $N$ as long as $N$ is as in (\ref{N number}).
		\item [(ii)] When $\vec{a}=(1,1,\dots, 1)$ and $\vec{p}=(1,1,\dots, 1)$, then anisotropic mixed-norm Herz-Hardy space coincides with the classical isotropic Herz-Hardy space.
		\item [(iii)] Anisotropic mixed-norm Herz-Hardy spaces $H \dot K_{\vec{q}, \vec{a}}^{\alpha, p}\left(\mathbb{R}^{n}\right)$ and $H K_{\vec{q}, \vec{a}}^{\alpha, p}\left(\mathbb{R}^{n}\right)$ are quasi-norm Banach spaces. This result can immediately obtain via Proposition \ref{quasiBanach}.
	\end{enumerate}}
\end{remark}	
\begin{proposition}
Let  $0<p<\infty$, $1<\vec{q}<\infty$ and  $N>N_{\vec{q}}$. $\frac{-1}{v} \sum_{i=1}^n \frac{a_{i}}{q_{i}}<\alpha< 1-\frac{1}{v}\sum_{i=1}^n \frac{a_{i}}{q_{i}}$.
Then     
$$H \dot K_{\vec{q}, \vec{a}}^{\alpha, p}\left(\mathbb{R}^{n}\right)\cap L_{\rm{l o c}}^{\vec{q}}\left(\mathbb{R}^{n} /\{0\}\right)=\dot K_{\vec{q}, \vec{a}}^{\alpha, p}\left( \mathbb{R}^n\right).$$
and
$$H K_{\vec{q}, \vec{a}}^{\alpha, p}\left(\mathbb{R}^{n}\right)\cap L_{\rm{l o c}}^{\vec{q}}\left(\mathbb{R}^{n}\right)=K_{\vec{q}, \vec{a}}^{\alpha, p}\left( \mathbb{R}^n\right).$$
\end{proposition}
\begin{proof}
	Remark \ref{equivalent MN} shows that racial maximal operators are equivalent to grand maximal operators, especially, racial maximal operators have followed the result if function $\varphi \in L^1 \left(\mathbb{R}^n\right)$,
	$$\left|M_N^0(f)(x)\right|\leq C\|\varphi\|_{L^1\left(\mathbb{R}^{n}\right)}Mf(x).$$
	Then, the following inequalities can be gained 
	\begin{equation}
		\left|M_N(f)(x)\right|\lesssim \left|M_N^0(f)(x)\right|\lesssim Mf(x). \label{HL equ}
	\end{equation}
Moreover, combine the Theorem \ref{boundedness HK} with the above inequality, we can obtain 
$$\left\|f\right\|_{H\dot K_{\vec{q}, \vec{a}}^{\alpha, p}\left(\mathbb{R}^{n}\right)}=\left\|M_N (f)\right\|_{\dot K_{\vec{q}, \vec{a}}^{\alpha, p}\left(\mathbb{R}^{n}\right)} \lesssim \left\|M (f)\right\|_{\dot K_{\vec{q}, \vec{a}}^{\alpha, p}\left(\mathbb{R}^{n}\right)} \lesssim \left\| f \right\|_{\dot K_{\vec{q}, \vec{a}}^{\alpha, p}\left(\mathbb{R}^{n}\right)}.  $$
On the other hand,
$$\left\| f \right\|_{\dot K_{\vec{q}, \vec{a}}^{\alpha, p}\left(\mathbb{R}^{n}\right)} \leq \left\| M_N(f) \right\|_{\dot K_{\vec{q}, \vec{a}}^{\alpha, p}\left(\mathbb{R}^{n}\right)}= \left\|f\right\|_{H\dot K_{\vec{q}, \vec{a}}^{\alpha, p}\left(\mathbb{R}^{n}\right)}.$$
The above two inequality can conclude $\dot K_{\vec{q}, \vec{a}}^{\alpha, p}\left(\mathbb{R}^{n}\right) \subset H\dot K_{\vec{q}, \vec{a}}^{\alpha, p}\left(\mathbb{R}^{n}\right)$ and $H\dot K_{\vec{q}, \vec{a}}^{\alpha, p}\left(\mathbb{R}^{n}\right) \subset \dot K_{\vec{q}, \vec{a}}^{\alpha, p}\left(\mathbb{R}^{n}\right)$, respectively, the non-homogeneous case is similar.
\end{proof}

\subsection{Atomic Decomposition on Anisotropic mixed-norm Herz-Hardy spaces}
The atom decomposition of anisotropic mixed-norm Herz-Hardy spaces is one of the most important characterization of $H \dot K_{\vec{q}, \vec{a}}^{\alpha, p}\left(\mathbb{R}^{n}\right)$ and $H K_{\vec{q}, \vec{a}}^{\alpha, p}\left(\mathbb{R}^{n}\right)$. That is to say, a distribution in Herz-Hardy spaces can be represented as a sum of a sequence of atoms, whose properties, such as compact support condition, integrability and vanishing moments, are simple and better.
\begin{definition}
	Let  $1<\vec{q}<\infty$, $1-\frac{1}{v}\sum_{i=1}^n \frac{a_i}{q_i} \leq \alpha<\infty$, and non-negative integer $s \geq \left\lfloor \frac{v}{a_-}\left(\alpha+\frac{1}{v} \sum_{i=1}^n \frac{a_{i}}{q_{i}}-1\right)\right\rfloor.$
	\begin{enumerate}
		\item [\rm{(1)}] A function $a(x)$ on $\mathbb{R}^{n}$ is called a central $(\alpha, \vec{q}, s)$ atom, if it satisfies
		\begin{itemize}
			\item [\rm{(i)}] ${\rm supp}~a \subset B_{k, \vec{a}}, B_{k, \vec{a}} \in \mathfrak{B}, k\in Z,$
			\item [\rm{(ii)}] $ \|a\| _{L^{\vec{q}}} \leq\left|B_{k, \vec{a}}\right|^{-\alpha},$
			\item  [\rm{(iii)}]  $\int_{\mathbb{R}^n} a(x) \cdot x^{\beta} d x=0 ~ for~ all~|\beta| \leq s.$
		\end{itemize}
		\item [\rm{(2)}] A function $a(x)$ on $\mathbb{R}^{n}$ is called a central $(\alpha, \vec{q}, s)$ atom of restricted type, if it satisfies {\rm(ii)} and {\rm(iii)}
		\begin{itemize}
			\item [\rm{(iv)}] ${\rm supp}~a \subset B_{k, \vec{a}}, B_{k, \vec{a}} \in \mathfrak{B}, k\geq 0.$
		\end{itemize}
	\end{enumerate}
\end{definition}

\begin{theorem}\label{atom suffi}
	Let $0<p<\infty$, $1<\vec{q}<\infty$, $1-\frac{1}{v}\sum_{i=1}^n \frac{a_i}{q_i} \leq \alpha<\infty$, and non-negative integer $s \geq \left\lfloor \frac{v}{a_-}\left(\alpha+\frac{1}{v} \sum_{i=1}^n \frac{a_{i}}{q_{i}}-1\right)\right\rfloor.$ If 
	$$f(x)=\sum_{k=-\infty}^{\infty} \lambda_{k} a_{k}(x),~~~~~\text{in the distribution sense,} $$
	where each $a_{k}$ is a central $(\alpha, \vec{q}, s)$-atom with support $B_{k}$, and $\sum_{k=\infty}^{\infty}\left|\lambda_{k}\right|^{p}<\infty $, then $f \in H \dot K_{\vec{q}, \vec{a}}^{\alpha, p}\left(\mathbb{R}^{n}\right)$,
$$\|f\|_{H \dot K_{\vec{q}, \vec{a}}^{\alpha, p}\left(\mathbb{R}^{n}\right)} \leq C\left(\sum_{k=-\infty}^{\infty}\left|\lambda_{k}\right|^{p}\right)^{\frac{1}{p}},$$
where C is a constant independent of $f$.
\end{theorem}

\begin{proof}
	It suffices to show that 
	$$\|f\|_{H \dot K_{\vec{q}, \vec{a}}^{\alpha, p}\left(\mathbb{R}^{n}\right)} \leq C\left(\sum_{k=-\infty}^{\infty}\left|\lambda_{k}\right|^{p}\right)^{\frac{1}{p}}$$
Based on atom decomposition, we have
\begin{align*}
\|f\|_{H \dot{K}_{\vec{q} , \vec{a}}^{{\alpha , p}}\left(\mathbb{R}^{n}\right)}= \left\|\sum_{k = -\infty}^{\infty} \lambda_{k} a_{k}\right\|_{H \dot{K}_{\vec{q} , \vec{a}}^{{\alpha , p}}\left(\mathbb{R}^{n}\right)} = \left\|M_{N}\left(\sum_{k = -\infty}^{\infty} \lambda_{k} a_{k}\right)\right\|_{\dot{K}_{\vec{q} , \vec{a}}^{{\alpha , p}}\left(\mathbb{R}^{n}\right)}
\end{align*}
Moreover, 
\begin{align*}
\left\|M_{N}\left(\sum_{k \in \mathbb Z} \lambda_{k} a_{k}\right)\right\|_{H \dot{K}_{\vec{q} , \vec{a}}^{{\alpha , p}}\left(\mathbb{R}^{n}\right)}^{p}&=\sum_{j \in \mathbb {Z}}\left|B_{j, \vec{a}}\right|^{\alpha p}\left\|M_{N}\left(\sum_{k \in \mathbb{Z}} \lambda_{k} a_{k}\right)  \chi_{j}\right\|_{L^{\vec{q}}\left(\mathbb{R}^{n}\right)}^{p}\\
&\leqslant \sum_{j \in \mathbb{Z}}\left|B_{j, \vec{a}}\right|^{\alpha p}\left(\sum_{k \in \mathbb{Z}}\left|\lambda_{k}\right|\left\|\left(M_{N} a_{k}\right)  \chi_{j}\right\|_{L^{\vec{q}}\left(\mathbb{R}^{n}\right)}\right)^{p}\\
&\leq \sum_{j \in \mathbb Z}\left|B_{j,\vec{a}}\right|^{\alpha p} \sum_{k{\in \mathbb Z}}\left|\lambda_{k}\right|^{p}\left\|(M_{N} a_{k})  \chi_{j}\right\|_{L^{\vec{q}}\left(\mathbb{R}^{n}\right)}^{p}\\
&=\sum_{k \in \mathbb Z}|\lambda_k|^{p}  \sum_{j \in \mathbb Z}\left|B_{j,\vec{a}}\right|^{\alpha p}\left\|(M_{N} a_{k}) \chi_{j}\right\|_{L^{\vec{q}}\left(\mathbb{R}^{n}\right)}^{p} \\
&=\sum_{k \in \mathbb Z}|\lambda_k|^{p} \left\|M_{N} a_{k}\right\|_{\dot{K}_{\vec{q} ,\vec{a}}^{{\alpha , p}}\left(\mathbb{R}^{n}\right)}.
\end{align*}

Above inequalities tell us only need to prove $\left\|M_{N} a_{k}\right\|_{\dot{K}_{\vec{q}, \vec{a}}^{{\alpha, p}}\left(\mathbb{R}^{n}\right)}\leq C$ for all central $(\alpha , \vec{q}, s)-atom.$ Form Remark \ref{equivalent MN}, suppose that supp $a \subset B_{k_{0}, \vec{a}}$.
\begin{align*}
\left\|M_{N} a\right\|_{ \dot{K}_{\vec{q} , \vec{a}}^{{\alpha , p}}\left(\mathbb{R}^{n}\right)}^{p}
&\lesssim \left\|M_{N}^{0} a\right\|_{ \dot{K}_{\vec{q} , \vec{a}}^{{\alpha , p}}\left(\mathbb{R}^{n}\right)}^{p}\\
&=\sum_{k=-\infty}^{\infty}\left|B_{k, \vec{a}} \right|^{\alpha p}\left\|(M_{N}^{0} a)  \chi_{k}\right\|_{L^{\vec{q}}\left(\mathbb{R}^{n}\right)}^{p}\\
&\leq \sum_{k = -\infty}^{k_{0}+1}\left|B_{k, \vec{a}} \right|^{\alpha p}\left\|(M_{N}^{0} a)  \chi_{k}\right\|_{L^{\vec{q}}\left(\mathbb{R}^{n}\right)}^{p} \\
&\quad +\sum_{k = k_{0}+2}^{\infty}\left|B_{k, \vec{a}} \right|^{\alpha p}\left\|(M_{N}^{0} a)  \chi_{k}\right\|_{L^{\vec{q}}\left(\mathbb{R}^{n}\right)}^{p} \\ &= I_{1}+I_{2} .
\end{align*}
When  $x \in B_{k_{0}+1, \vec{a}}$, it is easy to see that $ B_{k_{0}+1, \vec{a}}=2 B_{{k_{0}}, \vec{a}}$, then
\begin{align*}
I_{1}  &= \sum_{k = -\infty}^{k_{0}+1}\left|B_{k, \vec{a}}\right|^{\alpha p}\left\|(M_{N}^{0} a) \chi_{k}\right\|_{L^{\vec{q}}\left(\mathbb{R}^{n}\right)}^{p} \\
&\leq \sum_{k = -\infty}^{k_{0}+1}\left|B_{k, \vec{a}}\right|^{\alpha p}\left\|M_{N}^{0} a\right\|_{L^{\vec{q}}\left(\mathbb{R}^{n}\right)}^{p} \\
&\lesssim \sum_{k = -\infty}^{k_{0}+1}\left|B_{k, \vec{a}}\right|^{\alpha p}\left\|M a\right\|_{L^{\vec{q}}\left(\mathbb{R}^{n}\right)}^{p}\\
&\lesssim \sum_{k=- \infty}^{k_{0}+1} 2^{\left(k-k_{0}\right) v \alpha p} \leq C.
\end{align*}
On the other hand, for any  $(\alpha, \vec{q},s) -atom$  $a$ , by the vanishing moment condition of $a$ , we conclude that for any  $t \in(0, \infty)$  and  $x \in\left(2 B_{k_{0}, \vec{a}}\right)^{c} $.
\begin{align*}
\left|M_{N}^{0} a(x)\right| &=\left|a * \varphi_{t}(x)\right| \\
&\leq \int_{B_{k_{0},\vec{a}}}\left|a(y)  \varphi_{t}(x-y)\right| dy\\
&= t^{-v} \int_{B_{k_{0},\vec{a}}}|a(y)|\left|\varphi\left(\frac{x-y}{t^{\vec{a}}}\right)-\sum_{|\alpha| \leq s} \frac{\partial^{\alpha} \varphi\left(\frac{x-x_{k_{0}}}{t^{\vec{a}}}\right)}{\alpha !}\left(\frac{x_{k_{0}}-y}{t^{\vec{a}}}\right)^{\alpha}\right| dy \\
&\lesssim t^{-v} \int_{B_{k_{0},\vec{a}}}|a(y)| \left|\sum_{|\alpha| = S+1} \partial^{\alpha} \varphi\left(\frac{\xi}{t^{\vec{a}}}\right)\left(\frac{x_{k_{0}}-y}{t^{\vec{a}}}\right)^{\alpha} \right| dy ,
\end{align*}
where $\varphi$ satisfies $\int_{\mathbb{R}^{n}} \varphi(x) d x \neq 0$. $ x_{k_{0}}, r_{k_{0}}$  denote the center and radius of  $B_{k_{0},\vec{a}} $ and  $\xi=x-x_{k_{0}}+\theta\left(x_{k_{0}}-y\right)$  for some  $\theta \in[0,1]$.
\begin{align*}
|a * \varphi_{t}(x)|&\lesssim t^{-v} \int_{B_{k_{0},\vec{a}}}|a(y)| \left|\sum_{|\alpha| = s+1} \varphi\left(\frac{\xi}{t^{\vec{a}}}\right)\left(\frac{x_{k_{0}}-y}{t^{\vec{a}}}\right)^{\alpha}\right| dy .\\
&\lesssim t^{-v} \int_{B_{k_{0},\vec{a}}}|a(y)| \frac{1}{\left(1+\left|\frac{\xi}{t^{-\vec{a}}}\right|\right)^{k}}\left|\frac{x_{k_{0}}-y}{t^{\vec{a}}}\right|^{s+1} d y \text {. }\\
&\leqslant t^{-v} \int_{B_{k_{0},\vec{a}}}|a(y)| \left(1+\left.\right|^{\frac{\xi}{t^{\vec{a}}}}|_{\vec{a}}\right)^{-k a_-}\left|\frac{x_{k_{0}}-y}{t^{\vec{a}}}\right|^{s+1} dy\\
&\lesssim t^{-v} \int_{B_{k_{0},\vec{a}}}|a(y)| \left(t^{-1}|\xi|_{\vec{a}}\right)^{-k a_{-}}\left|\frac{x_{k_{0}}-y}{t^{ \vec{a}}}\right|^{s+1} d y \\
&\lesssim t^{-v} \int_{B_{k_{0},\vec{a}}}|a(y)| \left(\frac{t}{|\xi|_{\vec{a}}}\right)^{k a_-}\left(\frac{r_{k_{0}}}{t}\right)^{(s+1) a_-}dy \\
&\lesssim t^{-v} \int_{B_{k_{0},\vec{a}}}|a(y)|\left(\frac{t}{\left|x-x_{k_{0}}\right|_{ \vec{a}}}\right)^{k a_-}\left(\frac{r_{k_{0}}}{t}\right)^{(s+1) a_-}d y \\
&\lesssim t^{-v} \int_{B_{k_{0},\vec{a}}}|a(y)|\left(\frac{r_{k_{0}}}{\left|x-x_{k_{0}}\right|_{ \vec{a}}}\right)^{(s+1) a_-}d y \\
&\lesssim t^{-v}\left(\frac{r_{k_{0}}}{\left|x-x_{k_0}\right|_{\vec{a}}}\right)^{(s+1) a_-} \|a\|_{L^{\vec{q}}\left(\mathbb{R}^{n}\right)}\left\|B_{k_{0},\vec{a}}\right\|_{L^{\vec{q}^{\prime}}\left(\mathbb{R}^{n}\right)}\\
&\lesssim t^{-v}\left(\frac{r_{k_{0}}}{\left|x-x_{k_0}\right|_{\vec{a}}}\right)^{(s+1) a_-}\left|B_{k_{0}, \vec{a}}\right|^{-\alpha} r_{k_{0}}^{-\sum \frac{a_{i}}{q_{i}'}}\\
&\lesssim\left(\frac{r_{k_{0}}}{\left|x-x_{k_{0}}\right|_{\vec{a}}} \right)^{v+(s+1) a_-}\left|B_{k_{0}, \vec{a}}\right|^{-\alpha} r_{k_{0}}^{-\sum \frac{a_{i}}{q_{i}}}.
\end{align*}
which implies that for any  $i \in \mathbb{N}$. and  $x \in\left(2 B_{k_{0}}\right)^{c}$,
$$
M_{N}^{0} a(x) \lesssim \left(\frac{r_{k_{0}}}{\left|x-x_{k_{0}}\right|_{ \vec{a}}}\right)^{v+(s+1) a_{-}}\left|B_{k_{0}, \vec{a}}\right|^{-\alpha} r_{k_{0}}^{-\Sigma \frac{a_{i}}{q_{i}}},
$$
then, when $1-\frac{1}{v}\sum_{i=1}^n \frac{a_i}{q_i} \leq \alpha<\infty$, and non-negative integer $s \geq \left\lfloor \frac{v}{a_-}\left(\alpha+\frac{1}{v} \sum_{i=1}^n \frac{a_{i}}{q_{i}}-1\right)\right\rfloor$. 
\begin{align*}
I_{2} &=\sum_{k=k_{0}+2}^{\infty}\left|B_{k, \vec{a}} \right|^{\alpha p}\left\|(M_{N}^{0} a) x_{k}\right\|_{L^{\vec{q}}\left(\mathbb{R}^n\right)}^{p} \\
&=\sum_{k=k_{0}+2}^{\infty}\left|B_{k, \vec{a}}\right|^{\alpha p} \left|B_{k_{0}, \vec{a}}\right|^{-\alpha p}  r_{k_{0}}^{-p \sum\limits_{i=1}^n \frac{a_{i}}{q_{i}}}  2^{\left(k_{0}-k\right) p(v+(s+1) a_-} \left\|x_{k}\right\|_{L^{\vec{q}}\left(\mathbb{R}^n\right)}^{p} \\
&=\sum_{k=k_{0}+2}^{\infty} 2^{-\left(k_{0}-k\right) v \alpha p}  2^{\left(k_{0}-k\right) p\left(-\sum\limits_{i=1}^n \frac{a_{i}}{q_{i}}\right)}  2^{\left(k_{0}-k\right) p(v+(s+1) a_-)}\\
&=\sum_{k=k_{0}+2}^{\infty} 2^{\left(k_{0}-k\right) p\left(v+(s+1) a_--v \alpha-\sum\limits_{i=1}^n \frac{a_{i}}{q_{i}}\right)} \leq C.
\end{align*}

If $1<p<\infty$. then 
\begin{align*}
\left\|M_{N}^{0} f\right\|_{\dot{K}_{\vec{q} ,\vec{a}}^{{\alpha , p}}}^{p} &=\sum_{j=-\infty}^{\infty}\left|B_{j, \vec{a}}\right|^{\alpha p}\left\|(M_{N}^{0} f)  \chi_{j}\right\|_{L^{\vec{q}}\left(\mathbb{R}^n\right)}^{p} \\
&=\sum_{j=-\infty}^{\infty}\left|B_{j, \vec{a}}\right|^{\alpha p}\left\|M_{N}^{0}\left(\sum_{k=-\infty}^{\infty} \lambda_{k} a_{k}\right)  \chi_{j}\right\|_{L^{\vec{q}}\left(\mathbb{R}^n\right)}^{p} \\
&\leq \sum_{j=-\infty}^{\infty}\left|B_{j,\vec{a}}\right|^{\alpha p}\left(\sum_{k=-\infty}^{\infty}\left|\lambda_{k}\right|\left\|(M_{N}^{0} a_{k}) \chi_{j}\right\|_{L^{\vec{q}}\left(\mathbb{R}^n\right)}\right)^{p}\\
&\leq C \sum_{j=-\infty}^{\infty}\left|B_{j, \vec{a}}\right|^{\alpha p}\left(\sum_{k=-\infty}^{j-2}\left|\lambda_{k}\right| \| (M_{N}^{0} a_{k}) \chi_{j}\|_{L^{\vec{q}}\left(\mathbb{R}^n\right)}\right)^{p} \\
&\quad +C \sum_{j=-\infty}^{\infty}\left|B_{j, \vec{a}}\right|^{\alpha p}\left(\sum_{k=j-1}^{\infty}\left|\lambda_{k}\right|\left\|(M_{N}^{0} a_{k} )\chi_{j}\right\|_{L^{\vec{q}}\left(\mathbb{R}^n\right)}\right)^{p} \\
&=II_{1}+II_{2}.
\end{align*}
By H\"older's inequality on mixed-norm Lebesgue spaces, inequality~(\ref{HL equ}) and Lemma \ref{boundedness HL}, we have
\begin{align*}
II_{1} & \lesssim \sum_{j=-\infty}^{\infty}\left|B_{j, \vec{a}}\right|^{\alpha p}\left(\sum_{k=j-1}^{\infty}\left|\lambda_{k}\right|\left\|(M_{N}^{0} a_{k}) \chi_{j}\right\|_{L^{\vec{q}}\left(\mathbb{R}^n\right)}\right)^{p} \\
& \leqslant \sum_{j=-\infty}^{\infty}\left|B_{j, \vec{a}}\right|^{\alpha p}\left(\sum_{k=j-1}^{\infty}\left|\lambda_{k}\right| \left\|M_{N}^{0} a_{k}\right\|_{L^{\vec{q}}\left(\mathbb{R}^n\right)}\right)^{p} \\
& \leqslant \sum_{j=-\infty}^{\infty}\left|B_{j, \vec{a}}\right|^{\alpha p}\left(\sum_{k=j-1}^{\infty}\left|\lambda_{k}\right| \left\|a_{k}\right\|{ }_{L^{\vec{q}}\left(\mathbb{R}^n\right)}\right)^{p} \\
& \leqslant \sum_{j=-\infty}^{\infty}\left|B_{j, \vec{a}}\right|^{\alpha p}\left(\sum_{k=j-1}^{\infty}\left|\lambda_{k}\right| \left|B_{k,\vec{a}}\right|^{-\alpha}\right)^{p}\\
&\lesssim \sum_{j=-\infty}^{\infty}\left|B_{j}, \vec{a}\right|^{\alpha p}\left(\sum_{k=j-1}^{\infty}\left|\lambda_{k}\right|^{p}\left|B_{k, \vec{a}}\right|^{{-\alpha p}/{2}}\right) \left(\sum_{k=-\infty}^{k_0+1}\left|B_{k, \vec{a}}\right|^{{-\alpha p}/{2}}\right)^{{p}/{p^{\prime}}} \\
&\lesssim \sum_{j=-\infty}^{\infty} \sum_{k=j-1}^{\infty}\left|B_{j, \vec{a}}\right|^{{\alpha p}/{2}} \left|B_{k, \vec{a}}\right|^{-{\alpha p}/{2}} \left|\lambda_{k}\right|^{p} \\
&=\sum_{k=-\infty}^{\infty}\left|\lambda_{k}\right|^{p}  \sum_{j=-\infty}^{k+1} 2^{(j-k) v \alpha p} \quad \lesssim \sum_{k=-\infty}^{\infty}\left|\lambda_{k}\right|^{p } < \infty.
\end{align*}
Using the same estimate of $I_2$ for  $M_{N}^{0} a$. When $x \in 2 B_{k} $, $k\geq k_0+2$, then
\begin{align*}
\left\|M_{N}^{0} a \cdot x_{k}\right\|_{L^{\vec{q}}\left(\mathbb{R}^n\right)} \leq 2^{\left(k_{0}-k\right)(v+(s+1) a-)} \left|B_{k_{0}, \vec{a}}\right|^{-\alpha}2^{\left(k-k_{0}\right) \sum\limits_{i=1}^n \frac{a_{i}}{q_{i}}} .
\end{align*}
Combine above inequality and H\"older's inequality, we get
\begin{align*}
\text {II}_{2} & \lesssim \sum_{j=-\infty}^{\infty}\left|B_{j, \vec{a}}\right|^{ \alpha p}\left(\sum_{k=-\infty}^{j-2}\left|\lambda_{k}\right|\left\|(M_{N}^{0} a_{k}) \chi_{j}\right\|_{L^{\vec{q}}\left(\mathbb{R}^n\right)}\right)^{p} \\
&=\sum_{j=-\infty}^{\infty}\left|B_{j, \vec{a}}\right|^{\alpha p}\left(\sum_{k=-\infty}^{j-2}|\lambda_k|  2^{(k-j)\left(v+(s+1) a_{-}-\sum\limits_{i=1}^n \frac{a_{i}}{q_{i}}\right)}  2^{-k v\alpha}\right)^{p}\\
&=\sum_{j=-\infty}^{\infty}\left(\sum_{k=-\infty}^{j-2}\left|\lambda_{k}\right| \cdot 2^{(k-j)\left(v+(s+1) a_{-}-\sum\limits_{i=1}^n \frac{a_{i}}{q_{i}}-v \alpha\right)}\right)^{p }\\
&=\sum_{j=-\infty}^{\infty}\Bigg\{\left(\sum_{k=-\infty}^{j-2}\left|\lambda_{k}\right|  2^{(k-j)p(v+(s+1) a_{-}-\sum\limits_{i=1}^n \frac{a_{i}}{q_{i}}-v \alpha)/{2}}\right)\\
&\quad \times \left(\sum_{k=j-1}^{\infty} 2^{{(k-j) p^{\prime}}(v+(s+1) a_--\sum\limits_{i=1}^n \frac{a_{i}}{q_{i}}-v \alpha)/{2}}\right)^{{p}/{p^{\prime}}}\Bigg\}\\
&\lesssim \sum_{k=-\infty}^{\infty}\left|\lambda_{k}\right|^{p} \sum_{j=k+2}^{\infty} 2^{{(k-j) p}(v+(s+1) a_{-}-\sum\limits_{i=1}^n \frac{q_{i}}{q_{i}}-v_{\alpha})/{2}} \\
&\leq \sum_{k=-\infty}^{\infty}|\lambda|_{k}^{p}<\infty.
\end{align*}
This finishes the proof of Theorem \ref{atom suffi}.
\end{proof}
On the basis of the above Theorem \ref{atom suffi}, which shows that distribution function with certain decomposition is an element of anisotropic mixed-norm Herz-Hardy spaces. We have the following space characterization of anisotropic mixed-norm Herz-Hardy spaces.
\begin{theorem} \label{atom cha}
Let~$0<p<\infty, 1<\vec{q}<\infty$, and $1-\frac{1}{v} \sum_{i=1}^n \frac{a_{i}}{q_{i}} \leq \alpha<\frac{\left(a_{+}-\sum_{i+1}^n \frac{a_{i}}{q_{i}}\right)}{v}+1$, then
$f \in {H \dot{K}_{\vec{q} \cdot \vec{a}}^{{\alpha \cdot p}}}\left(\mathbb{R}^{n}\right)$ if and only if
$$f(x)=\sum_{k=-\infty}^{\infty} \lambda_{k} a_{k}(x)~~~\text{in the distribution sense},$$ 
where  $a_{k}$  is a central  $(\alpha, \vec{q}, 0)$-atom with support  $B_{k,\vec{a}}$, and $\sum_{k=-\infty}^{\infty}\left|\lambda_{k}\right|^{p}<\infty.$  
Moreover, $$\|f\|_{H \dot{K}_{\vec{q} , \vec{a}}^{{\alpha , p}}}\sim \inf \left(\sum_{k=-\infty}^{\infty}\left|\lambda_{k}\right|^{p}\right)^{{1}/{p}},$$
where the infimum is taken over all the above decompositions of $f$.
\end{theorem}

\begin{proof}
	The proof of sufficiency is included in the proof of the above Theorem, so we need to prove only the necessity. Taking $\varphi\in \mathcal S(\mathbb{R}^n)$ such that $\int \varphi(x)\mathrm{d}x=1$. Set $\varphi_{(i)}(x)=2^{iv}\varphi(2^{i\vec{a}}x),~\varphi_{\tilde{i}}(x)=\varphi(2^{-i\vec{a}}x),~f^{i}(x):=f\ast\varphi_{(-i)}(x)$ for $f\in H\dot{K}^{\alpha,p}_{\vec{q},\vec{a}}(\mathbb{R}^n)$ and  $i\in\mathbb{N}$ from [\cite{huang2019atomic} Lemma 4.12] asserts that $f^{(i)}\rightarrow f$ in $\mathcal S^{'}$ as $i\rightarrow\infty$. In addition, take $\psi\in C_{0}^{\infty}(\mathbb{R}^n)$ such that ${\rm supp}~\psi\subset \tilde{A_{0}}:=A_{-1}\cup A_{0}\cup A_{1}$, $0\leq\psi\leq1$ and $\psi(x)=1$,~if $x\in A_{0}$.~${\rm supp}~\psi_{\tilde{k}}\subset\tilde{A_{k}}=A_{k-1}\cup A_{k}\cup A_{k+1}$.\\
	Set
$$
\Psi_{k}(x)=
\left\{
  \begin{array}{ll}
    \frac{\psi_{\tilde{k}}(x)}{\sum\limits_{j\in \mathbb{Z}}\psi_{\tilde{j}}(x)}, \quad &if~x\neq0, \\
    0, \quad & if~x=0.
  \end{array}
\right.
$$
Then, we have $\Psi_{k}\in C_{0}^{\infty}(\mathbb{R}^n),~{\rm supp}~\Psi_{k}\subset \tilde{A_{k}},~0\leq\Psi_{k}\leq1$, and $\sum_{k}\Psi_{k}(x)=1$,~if $x\neq 0$.~Let $v_{k}(x)=\mid A_{\tilde{k}}\mid^{-1}\chi_{\tilde{A_{k}}}(x)$.
\begin{align*}
f^{(i)}(x)
&=\sum\limits_{k=-\infty}^{\infty}f^{(i)}(x)\Psi_{k}(x)\\
&=\sum\limits_{k=-\infty}^{\infty}\left\{f^{(i)}(x)\Psi_{k}(x)-\left(\int_{\mathbb{R}^{n}}f^{(i)}(y)\Psi_{k}(y)\mathrm{d}y\right)v_{k}(x)\right\}\\
&~~~~+ \sum\limits_{k=-\infty}^{\infty}\left(\int_{\mathbb{R}^{n}}f^{(i)}(y)\Psi_{k}(y)\mathrm{d}y\right)v_{k}(x)\\
&=\sum_{1}^{(i)}+\sum_{2}^{(i)}.
\end{align*}

For $\sum_{1}^{(i)}$, denote $g_{k}^{(i)}(x)={f^{(i)}(x)\Psi_{k}(x)-(\int_{\mathbb{R}^{n}}f^{(i)}(y)\Psi_{k}(y)\mathrm{d}y)v_{k}(x)}$,~by H\"{o}lder's inequality,
$$
\parallel g^{(i)}_{k}\parallel_{L^{\vec{q}}\left(\mathbb{R}^n\right)}\leq C\parallel f^{(i)}\Psi_{k}\parallel_{L^{\vec{q}}\left(\mathbb{R}^n\right)}\leq C^{'}\sum\limits_{j=k-1}^{k+1}\parallel (M_{N}f) \chi_{j}\parallel_{L^{\vec{q}}\left(\mathbb{R}^n\right)}.
$$
Let $a_{1,k}^{(i)}(x)=\frac{g_{k}^{(i)}(x)}{C^{'}|B_{k+1,\vec{a}}|^{\alpha}\sum\limits_{j=k-1}^{k+1}\parallel (M_{N}f)  \chi_{j}\parallel_{L^{\vec{q}}\left(\mathbb{R}^n\right)}}$, by the definition of $a_{1,k}^{(i)}$, then ${\rm supp}~a_{1,k}^{(i)}\subset \tilde{A_{k}}\subset B_{k+1,\vec{a}}$.
Moreover,
\begin{align*}
g_{k}^{(i)}
&=C^{'}|B_{k+1,\vec{a}}|^{\alpha}\sum\limits_{j=k-1}^{k+1}\parallel (M_{N}f)  \chi_{j}\parallel_{L^{\vec{q}}\left(\mathbb{R}^n\right)}\frac{g_{k}^{(i)}(x)}{C^{'}|B_{k+1,\vec{a}}|^{\alpha}\sum\limits_{j=k-1}^{k+1}\parallel (M_{N}f) \chi_{j}\parallel_{L^{\vec{q}}\left(\mathbb{R}^n\right)}}\\
&=\lambda_{1,k}a_{1,k}^{(i)}.
\end{align*}
Then,
$$
\|a_{1,k}^{(i)}\|_{L^{\vec{q}}\left(\mathbb{R}^n\right)}=\frac{\|g_{k}^{(i)}\|_{L^{\vec{q}}\left(\mathbb{R}^n\right)}}{C^{'}|B_{k+1,\vec{a}}|^{\alpha}\sum\limits_{j=k-1}^{k+1}\parallel (M_{N}f)  \chi_{j}\parallel_{L^{\vec{q}}\left(\mathbb{R}^n\right)}}\leq|B_{k+1,\vec{a}}|^{-\alpha}.
$$
And,
\begin{align*}
\int_{\mathbb{R}^{n}}a_{1,k}^{(i)}(x)\mathrm{d}x
&=\int_{\mathbb{R}^{n}}\frac{g_{k}^{(i)}(x)}{C^{'}|B_{k+1,\vec{a}}|^{\alpha}\sum\limits_{j=k-1}^{k+1}\parallel M_{N}f\parallel_{L^{\vec{q}}\left(\mathbb{R}^n\right)}}\mathrm{d}x\\
&=C^{'}|B_{k+1,\vec{a}}|^{-\alpha}\left(\sum\limits_{j=k-1}^{k+1}\parallel M_{N}f\parallel_{L^{\vec{q}}\left(\mathbb{R}^n\right)}\right)^{-1}\\
& \quad \times \left(\int_{\mathbb{R}^{n}}{f^{(i)}(x)\Psi_{k}(x)-\left(\int_{\mathbb{R}^{n}}f^{(i)}(y)\Psi_{k}(y)\mathrm{d}y\right)v_{k}(x)}\mathrm{d}x\right)\\
&=0.
\end{align*}
It is say that, $a_{1,k}^{(i)}$ is a central $(\alpha,q,0)-$atom supported on $B_{k+1,\vec{a}}$.
\begin{align*}
\sum\limits_{k=-\infty}^{\infty}|\lambda_{1,k}|^{p}
&=\sum\limits_{k=-\infty}^{\infty}\left||B_{k+1,\vec{a}}|^{\alpha}\sum\limits_{j=k-1}^{k+1}\parallel (M_{N}f)  \chi_{j}\parallel_{L^{\vec{q}}\left(\mathbb{R}^n\right)}\right|^{p}\\
&\leq C\sum\limits_{k=-\infty}^{\infty}|B_{k+1,\vec{a}}|^{\alpha p}\parallel (M_{N}f)\chi_{k}\parallel^{p}_{L^{\vec{q}}\left(\mathbb{R}^n\right)}\\
&\leq C\sum\limits_{k=-\infty}^{\infty}|B_{k,\vec{a}}|^{\alpha p}\parallel (M_{N}f)\chi_{k}\parallel^{p}_{L^{\vec{q}}\left(\mathbb{R}^n\right)}\\
&=C\parallel f\parallel_{H\dot{K}^{\alpha,p}_{\vec{q},\vec{a}}\left(\mathbb{R}^n\right)}.
\end{align*}
Let us now turn to $\sum_{2}^{(i)}$. Summing by parts, one can get
\begin{align*}
\sum_{2}^{(i)}
&=\sum\limits_{k=-\infty}^{\infty}(\int_{\mathbb{R}^{n}}f^{(i)}(y)\Psi_{k}(y)\mathrm{d}y)v_{k}(x)\\
&=\sum\limits_{k=-\infty}^{\infty}\left(\sum\limits_{j=-\infty}^{\infty}\int_{\mathbb{R}^{n}}f^{(i)}(y)\Psi_{i}(y)\mathrm{d}y \left(v_{k}(x)-v_{k+1}(x)\right)\right)\\
&:=\sum\limits_{k=-\infty}^{\infty}h_{k}^{(i)}(x).
\end{align*}

Let $\Phi(x)=\sum_{j=-\infty}^{-2}\Psi_{j}(x)$. Since ${\rm supp}~ \Psi_{j}\subset \tilde{A_{j}}$, and $\sum_{j=-\infty}^{-2}\chi_{\tilde{A_{j}}}\leq C$, it is easy to see that $\Phi\in C^{\infty}_{0}(B_{-1})$ and hence $\Phi\in \mathcal S(\mathbb{R}^{n})$. Note that
$$
\sum\limits_{j=-\infty}^{k}\Psi_{j}(x)=\Phi\left(2^{(-k-2)\cdot \vec{a}}x\right)=2^{(k+2)v}\Phi_{k+2}(x).
$$
Hence, for any $x\in B_{k+2,\vec{a}}$, using [\cite{bownik2003anisotropic}, Lemma 6-6], we have
\begin{align*}
&~~~~~~\left|\int_{\mathbb{R}^{n}}f^{(i)}(y)\sum\limits_{j=-\infty}^{k}\Psi_{j}(y)\mathrm{d}y\right|\\
&=2^{(k+2)v}\left|\int_{B^{k+2}}f^{(i)}(y)\Phi_{k+2}(y)\mathrm{d}y\right|\\
&=2^{(k+2)v}\left|f^{(i)}\ast\Phi_{k+2}(0)\right|\\
&\leq C2^{(k+2)v}\|\tilde{\Phi}\|_{\mathcal S_{N+2}}M_{N+2}f(x)\\
&\leq C2^{(k+2)v}M_{N}f(x).
\end{align*}
Where $\tilde{\Phi}(y)=\Phi(-y)$ and $C$ is a constant dependent of $N$.\\
On the other hand,
$$
\left|v_{k}(x)-v_{k+1}(x)\right|\leq C2^{(-k-2)v}\sum\limits_{j=k-1}^{k+2}\chi_{j}(x).
$$
Therefore.
$$
\|h_{k}^{(i)}\|_{L^{\vec{q}}\left(\mathbb{R}^n\right)}\leq C^{''}\sum\limits_{j=k-1}^{k+2}\|(M_{N}f) \chi_{j}\|_{L^{\vec{q}}\left(\mathbb{R}^n\right)}.
$$
Let~$a_{2,k}^{(i)}(x)=\frac{h_{k}^{(i)}(x)}{C^{''}|B_{k+2,\vec{a}}|^{\alpha}\sum\limits_{j=k-1}^{k+1}\|(M_{N}f) \chi_{j}\|_{L^{\vec{q}}\left(\mathbb{R}^n\right)}}$, then, $h_{k}^{(i)}(x)=\lambda_{2,k}a_{2,k}^{(i)}(x)$.\\
Since ${\rm supp}~h_{k}^{(i)}\subset B_{k+2,\vec{a}}$, then $a_{2,k}^{(i)}\subset B_{k+2,\vec{a}}$.\\
Furthermore,
$$
\|a_{2,k}^{(i)}\|_{L^{\vec{q}}\left(\mathbb{R}^n\right)}=\frac{\|h_{k}^{(i)}\|_{L^{\vec{q}}\left(\mathbb{R}^n\right)}}{C^{''}|B_{k+2,\vec{a}}|^{\alpha}\sum\limits_{j=k-1}^{k+1}\|(M_{N}f ) \chi_{j}\|_{L^{\vec{q}}\left(\mathbb{R}^n\right)}}
\leq|B_{k+2,\vec{a}}|^{-\alpha}.
$$
And,
\begin{align*}
\sum\limits_{k=-\infty}^{\infty}\left|\lambda_{2,k}\right|^{p}
&=\sum\limits_{k=-\infty}^{\infty}\left|C^{''}|B_{k+2,\vec{a}}|^{\alpha}\sum\limits_{j=k-1}^{k+1}\left\|(M_{N}f ) \chi_{j}\right\|_{L^{\vec{q}}\left(\mathbb{R}^n\right)}\right|^{p}\\
&\leq C\sum\limits_{k=-\infty}^{\infty}|B_{k+2,\vec{a}}|^{\alpha p}\left\|(M_{N}f )\chi_{j}\right\|_{L^{\vec{q}}\left(\mathbb{R}^n\right)}^{p}
\leq C\parallel f\parallel_{H\dot{K}^{\alpha,p}_{\vec{q},\vec{a}}\left(\mathbb{R}^n\right)}.
\end{align*}
Using the Banach-Alaoglu theorem and usual diagonal method, there exists a subsequence $\{i_{v}\}^{\infty}_{v=1}\subset\mathbb{N}$ such that for each $l\in \mathbb{Z}, \{a_{l}^{(i_{v})}\}$ weak-* converges to a central $(\alpha,q,0)$-atom $a_{l}$ supported in $B_{k+2,\vec{a}}$. Hence, in the distribution sense, $a_{l}^{(i_{v})}\rightarrow a_{l}$ as
$v\rightarrow\infty$.
\par 
It remains to show that $f=\sum_{l=-\infty}^{\infty}\lambda_{l}a_{l}$ in the distribution sense.~Taking $\phi \in \mathcal S$.\\
 Noting that ${\rm supp}~ (a_{l}^{(i_{v})})\subset(\tilde{A}_{l}\cup
\tilde{A}_{l+1})\subset\bigcup_{j=l-1}^{l+2}A_{j}$, then
\begin{align*}
<f,\phi>
&=\lim\limits_{v\rightarrow\infty}<f^{i_{v}},\phi>\\
&=\lim\limits_{v\rightarrow\infty}\int_{\mathbb{R}^{n}}\sum\limits_{l=-\infty}^{\infty}\lambda_{l}a_{l}^{(i_{v})}(x)\phi(x)\mathrm{d}x\\
&=\lim\limits_{v\rightarrow\infty}\sum\limits_{k=-\infty}^{\infty}\int_{A_{k}}\sum\limits_{l=-\infty}^{\infty}\lambda_{l}a_{l}^{(i_{v})}(x)\phi(x)\mathrm{d}x\\
&=\lim\limits_{v\rightarrow\infty}\sum\limits_{k=-\infty}^{\infty}\int_{A_{k}}\sum\limits_{l=k-2}^{k+1}\lambda_{l}a_{l}^{(i_{v})}(x)\phi(x)\mathrm{d}x\\
&=\lim\limits_{v\rightarrow\infty}\sum\limits_{k=-\infty}^{\infty}\sum\limits_{l=k-2}^{k+1}\int_{A_{k}}\lambda_{l}a_{l}^{(i_{v})}(x)\phi(x)\mathrm{d}x.
\end{align*}
On the other hand.
\begin{align*}
&~~~~~~\lim\limits_{k_{1},k_{2}\rightarrow\infty}\int_{\mathbb{R}^{n}}\sum\limits_{l=-k_{2}}^{k_{1}}\lambda_{l}a_{l}^{(i_{v})}(x)\phi(x)\mathrm{d}x\\
&=\lim\limits_{k_{1},k_{2}\rightarrow\infty}\sum\limits_{l=-k_{2}}^{k_{1}}\sum\limits_{k=l-1}^{l+2}\int_{A_{k}}\lambda_{l}a_{l}^{(i_{v})}(x)\phi(x)\mathrm{d}x\\
&=\sum\limits_{l=-\infty}^{\infty}\sum\limits_{k=l-1}^{l+2}\int_{A_{k}}\lambda_{l}a_{l}^{(i_{v})}(x)\phi(x)\mathrm{d}x.
\end{align*}
Thus $<f,\phi>=\lim\limits_{v\rightarrow\infty}\lambda_{l}\int_{\mathbb{R}^{n}}a_{l}^{(i_{v})}\phi(x)\mathrm{d}x$.\\
If $l+2\leq0$, when $\alpha< \frac{(a_{+}-\sum_{i=1}^n \frac{a_{i}}{q_{i}})}{v}$, then
\begin{align*}
\left|\int_{\mathbb{R}^{n}}a_{l}^{(i_{v})}(x)\phi(x)\mathrm{d}x\right|
&=\left|\int_{B_{l+2}}a_{l}^{(i_{v})}(x)(\phi(x)-\phi(0))\mathrm{d}x\right|\\
&\leq C\int_{B_{l+2}}\left|a_{l}^{(i_{v})}(x)\right||x|\sup\limits_{z\in B_{l+2}}\sup\limits_{|\alpha|=1}|\partial^{\alpha}\phi(z)|\mathrm{d}x\\
&\leq C\int_{B_{l+2}}\left|a_{l}^{(i_{v})}(x)\right| 2(1+|x|_{\vec{a}})^{a_{+}}\mathrm{d}x\\
&\leq C2^{(l+2)a_{+}}\left\|a_{l}^{(i_{v})}\right\|_{L^{\vec{q}}\left(\mathbb{R}^n\right)}\|\chi_{B_{l+2}}\|_{L^{\vec{q}^{\prime}}\left(\mathbb{R}^n\right)}\\
&\leq C2^{(l+2)a_{+}}|B_{l+2}|^{-\alpha}(2^{(l+2)})^{\sum\limits_{i=1}^{n}\frac{a_{i}}{q_{i}}^{\prime}}\\
%&=C2^{(l+2)a_{+}}(2^{l+2})^{-\alpha v}(2^{(l+2)})^{v-\sum\frac{a_{i}}{q_{i}}}\\
%&=C2^{(l+2)(a_{+}-\alpha v+v-\sum\frac{a_{i}}{q_{i}})}\\
&=C2^{\left(l+2)(a_{+}-(\alpha-1)v-\sum\limits_{i=1}^{n}\frac{a_{i}}{q_{i}}\right)}\\
&\leq C2^{l\left(a_{+}-(\alpha-1)v-\sum\limits_{i=1}^{n}\frac{a_{i}}{q_{i}}\right)}.
\end{align*}
If~$l+2>0$, let $k_{0}\in\mathbb{Z}_{+}$, such that  $1<\alpha+\frac{(k_{0}a_{-})}{v}+\frac{1}{v} \sum_{i=1}^{n}\frac{a_{i}}{q_{i}}$, then
\begin{align*}
\left|\int_{\mathbb{R}^{n}}a_{l}^{(i_{v})}(x)\phi(x)\mathrm{d}x\right|
&\leq\int_{\tilde{A_{l}}\cup{\tilde A_{l+1}}}\left|a_{l}^{(i_{v})}(x)\right|(1+|x|)^{-k_{0}}\mathrm{d}x\\
&\lesssim \left\|a_{l}^{(i_{v})}\right\|_{L^{\vec{q}}}(1+|x|_{\vec{a}})^{-k_{0}a_{-}}\|\chi_{B_{l+1}}\|_{L^{\vec{q}^{\prime}}\left(\mathbb{R}^n\right)}\\
%&\lesssim 2^{-lv\alpha}\cdot2^{-lk_{0}a_{-}}2^{(l+1)\sum\frac{a_{i}}{q_{i}^{'}}}\\
&\lesssim 2^{-lv\alpha}2^{-lk_{0}a_{-}}2^{l\left(v-\sum\limits_{i=1}^{n}\frac{a_{i}}{q_{i}}\right)}\\
&=2^{lv\left(-\alpha-\frac{k_0a_-}{v}+1-\frac{1}{v}\sum\limits_{i=1}^{n}\frac{a_i}{q_i}\right)}.
\end{align*}
Noting that $$1-\frac{1}{v}\sum_{i=1}^n\frac{a_i}{q_i}\le\alpha<\frac{(a_+-\sum_{i=1}^n\frac{a_i}{q_i})}{v}+1,$$

if we denote$$\mu_l=\begin{cases}
\left|\lambda_l\right|2^{l\left(a_+-(\alpha-1)v-\sum\limits_{i=1}^n\frac{a_i}{q_i}\right)},~~l+2\le0,\\
\left|\lambda_l\right|2^{l\left(-\alpha-\frac{k_0a_-}{v}+1-\frac{1}{v}\sum\limits_{i=1}^{n}\frac{a_i}{q_i}\right)},~~l+2>0,\\
\end{cases}$$

then$$\sum_{l=-\infty}^\infty \mu_l\le C\left(\sum_{l=-\infty}^\infty\left|\lambda_l\right|^p\right)^{{1}/{p}}
\le C\lVert M_Nf\rVert_{\dot{K}_{\vec{q},\vec{a}}^{\alpha, p}\left(\mathbb{R}^n\right)}<\infty,$$

and$$\left|\lambda_l\right|\cdot\left|\int_{\mathbb{R}^n}a_l^{(i_u)}(x)\phi(x)dx\right|\le C\mu_l.$$

By the dominated convergence theorem of series,$$\langle f,\varphi\rangle=\sum_{l=-\infty}^{\infty}\lambda_l\int_{\mathbb{R}^n}a_l(x)\varphi(x)dx.$$
We get the desired result.
\end{proof}

\subsection{The Molecular Decomposition on Anisotropic Mixed Herz-Hardy Spaces}
In this subsection, we create the molecular decomposition of anisotropic mixed-norm Herz-Hardy. Molecular decomposition, which improves the compact support condition of atoms, can be seen the generation of atom decomposition. To begin with, we first give the notation of molecular.
\begin{definition} \label{molecular}
Let $0<p<\infty$, $1<\vec q<\infty$, $1-\frac{1}{v}\sum_{i=1}^n \frac{a_i}{q_i}\le\alpha<\infty$, and non-negative integer $s\ge \lfloor \frac{v}{a_-}(\alpha+\frac{1}{v}\sum_{i=1}^n \frac{a_i}{q_i}-1)  \rfloor$.
Set $\epsilon>max\{s,\frac{v}{a_-}(\alpha+\sum_{i=1}^n \frac{a_i}{q_i}-1)\}$, $a=(1-\frac{1}{v}\sum_{i=1}^n \frac{a_i}{q_i})-\alpha+\epsilon$,
and $d=(1-\frac{1}{v}\sum_{i=1}^n \frac{a_i}{q_i})+\epsilon$.
A function $M_l\in L^{\vec q}(\mathbb{R}^n)$ with $l\in \mathbb{Z}$ (or $l\in\mathbb{N}$) is called a dyadic central $(\alpha,\vec q;s,\epsilon)_l$-molecule (or dyadic central $(\alpha,\vec q;s,\epsilon)_l$-molecule of restricted type), if it satisfies
\begin{enumerate}
	\item [{\rm(i)}] $\lVert M_l\rVert_{L^{\vec q}\left(\mathbb{R}^n\right)}\le \left|B_{l,\vec a}\right|^{-\alpha};$
	\item [{\rm(ii)}] $R_{\vec q}(M_l):=\lVert M_l\rVert_{L^{\vec q}\left(\mathbb{R}^n\right)}^{a/d}\lVert \left|\cdot\right|_{\vec a}^{vd}M_l(\cdot)\rVert_{L^{\vec q}\left(\mathbb{R}^n\right)}^{1-a/d};$
	\item  [{\rm(iii)}] $\int_{\mathbb{R}^n}M_l(x)x^\beta dx=0$, for any $\beta$ with $\left|\beta\right|\le s$.
\end{enumerate}
\end{definition}

\begin{definition}
Let $\alpha$, $p$, $\vec q$, $s$ and $\epsilon$ be as in Definition \ref{molecular}.
\begin{enumerate}
	\item [{\rm (a)}] A function $M\in L^{\vec q}(\mathbb{R}^n)$ is called a central $(\alpha,\vec q;s,\epsilon)$-molecule, if it satisfies
     \begin{itemize}
     	\item [{\rm (iv)}] $R_{\vec q}(M):=\lVert M\rVert_{L^{\vec q}\left(\mathbb{R}^n\right)}^{a/d}\lVert \left|\cdot\right|_{\vec a}^{vd}M(\cdot)\rVert_{L^{\vec q}\left(\mathbb{R}^n\right)}^{1-a/d}<\infty;$
     	\item [{\rm (v)}] $\int_{\mathbb{R}^n}M(x)\cdot x^\beta dx=0$, for any $\beta$ with $\left|\beta\right|\le s$.
	     \end{itemize}
	 \item [{\rm (b)}] A function $M\in L^{\vec q}(\mathbb{R}^n)$ is called a central $(\alpha,\vec q;s,\epsilon)$-molecule of restricted type, if it satisfies the condition (iv), (v) above and
	 \begin{itemize} 
	 	\item [{\rm (vi)}] $\lVert M\rVert_{L^{\vec q}\left(\mathbb{R}^n\right)}\le1.$
	 \end{itemize}
\end{enumerate}
\end{definition}
The following lemma reveals the relationship between atoms and molecules.
\begin{lemma} \label{molecular rela}
Let $\alpha$, $p$, $\vec q$, $s$ and $\epsilon$ be as in Definition \ref{molecular}. If $M$ is a central $(\alpha,\vec q,s)$-atom (or $(\alpha,\vec q,s)$-atom of restricted type), then $M$ is also a central $(\alpha,\vec q;s,\epsilon)$-molecule (or $(\alpha,\vec q;s,\epsilon)$-molecule of restricted type), such that $R_{\vec q}(M)\le C$ with $C$ independent of $M$.
\end{lemma}
\begin{proof}
We only need to show the case that $M$ is a $(\alpha,\vec q,s)$-atom with support on ball $B_{k, \vec{a}}$.
\begin{equation*}
\begin{split}
&\left\lVert M\right\rVert_{L^{\vec q}\left(\mathbb{R}^n\right)}^{a/d} \left\lVert \left|x\right|_{\vec a}^{vd}M(\cdot)\right\rVert_{L^{\vec q}\left(\mathbb{R}^n\right)}^{1-a/d}\\
&\le C\left|B_{k,\vec a}\right|^{-\alpha\cdot a/d}(2^{kvd})^{1-a/d} \left|B_{k,\vec a}\right|^{-\alpha(1-a/d)}\\
&\le C2^{kv\alpha (-a/d)} (2^{kvd})^{1-a/d} 2^{kv\alpha (a/d)} 2^{-kv\alpha}\\
&\le C2^{kv(d-a)} 2^{-kv\alpha}\le C.
\end{split}
\end{equation*}
This proof is completed.
\end{proof}

\begin{theorem}\label{molecular cha 1}
Let $0<p<\infty$, $1<\vec q<\infty$, $1-\frac{1}{v}\sum_{i=1}^n \frac{a_i}{q_i}\le\alpha<\infty$, and non-negative integer $s\ge \lfloor \frac{v}{a_-}(\alpha+\frac{1}{v}\sum_{i=1}^n \frac{a_i}{q_i}-1)  \rfloor$.
Set $\epsilon>max\{s,\frac{v}{a_-}(\alpha+\sum_{i=1}^n \frac{a_i}{q_i}-1)\}$, $a=(1-\frac{1}{v}\sum_{i=1}^n \frac{a_i}{q_i})-\alpha+\epsilon$,
and $d=(1-\frac{1}{v}\sum_{i=1}^n \frac{a_i}{q_i})+\epsilon$.Then,
\begin{enumerate}
	\item [{\rm(1)}] $f\in H\dot{K}_{\vec{q},\vec{a}}^{\alpha, p}(\mathbb{R}^n)$ if and only if $f$ can be represented as
	 $$f(x)=\sum_{k=-\infty}^{\infty}\lambda_kM_k(x),~~ \text{in the distributional sense,}$$
	  where each $M_k$ is a central $(\alpha,\vec q;s,\epsilon)_k$-molecule, and $\sum_{k=-\infty}^{\infty}\left|\lambda_k\right|^p<+\infty$. Moreover,$$\lVert f\rVert_{H\dot{K}_{\vec{q},\vec{a}}^{\alpha, p}\left(\mathbb{R}^n\right)}\sim\inf\left(\sum_{k=-\infty}^{\infty}\left|\lambda_k\right|^p\right)^{{1}/{p}},$$where the infimum is taken over all above decompositions of $f$.
	  \item [{\rm(2)}] $f\in HK_{\vec{q},\vec{a}}^{\alpha, p}(\mathbb{R}^n)$ if and only if $f$ can be represented as $$f(x)=\sum_{k=-\infty}^{\infty}\lambda_kM_k(x),~~ \text{in the distributional sense,}$$
    where each $M_k$ is a central $(\alpha,\vec q;s,\epsilon)_k$-molecule, and $\sum_{k=-\infty}^{\infty}\left|\lambda_k\right|^p<+\infty$. Moreover,$$\lVert f\rVert_{HK_{\vec{q},\vec{a}}^{\alpha, p}\left(\mathbb{R}^n\right)}\sim\inf\left(\sum_{k=-\infty}^{\infty}\left|\lambda_k\right|^p\right)^{{1}/{p}},$$where the infimum is taken over all above decompositions of $f$.
\end{enumerate}
\end{theorem}

\begin{theorem} \label{molecular cha 2}
Let $0<p<\infty$, $1<\vec q<\infty$, $1-\frac{1}{v}\sum_{i=1}^n \frac{a_i}{q_i}\le\alpha<\infty$, and non-negative integer $s\ge \lfloor \frac{v}{a_-}(\alpha+\frac{1}{v}\sum_{i=1}^n \frac{a_i}{q_i}-1) \rfloor$.
Set $\epsilon>max\{s,\frac{v}{a_-}(\alpha+\sum_{i=1}^n \frac{a_i}{q_i}-1)\}$, $a=(1-\frac{1}{v}\sum_{i=1}^n \frac{a_i}{q_i})-\alpha+\epsilon$,
and $d=(1-\frac{1}{v}\sum_{i=1}^n \frac{a_i}{q_i})+\epsilon$. Then,
\begin{enumerate}
	\item [{\rm(1)}] $f\in H\dot{K}_{\vec{q},\vec{a}}^{\alpha, p}(\mathbb{R}^n)$ if and only if $f$ can be represented as
	 $$f(x)=\sum_{k=-\infty}^{\infty}\lambda_kM_k(x),~~ \text{in the distributional sense,}$$
	  where each $M_k$ is a central $(\alpha,\vec q;s,\epsilon)$-molecule, and $\sum_{k=-\infty}^{\infty}\left|\lambda_k\right|^p<+\infty$. Moreover,$$\lVert f\rVert_{H\dot{K}_{\vec{q},\vec{a}}^{\alpha, p}\left(\mathbb{R}^n\right)}\sim\inf\left(\sum_{k=-\infty}^{\infty}\left|\lambda_k\right|^p\right)^{{1}/{p}},$$where the infimum is taken over all above decompositions of $f$.
	  \item [{\rm(2)}] $f\in HK_{\vec{q},\vec{a}}^{\alpha, p}(\mathbb{R}^n)$ if and only if $f$ can be represented as
    $$f(x)=\sum_{k=-\infty}^{\infty}\lambda_kM_k(x),~~ \text{in the distributional sense,}$$ where each $M_k$ is a central $(\alpha,\vec q;s,\epsilon)$-molecule, and $\sum_{k=-\infty}^{\infty}\left|\lambda_k\right|^p<+\infty$. Moreover,$$\lVert f\rVert_{HK_{\vec{q},\vec{a}}^{\alpha, p}\left(\mathbb{R}^n\right)}\sim\inf\left(\sum_{k=-\infty}^{\infty}\left|\lambda_k\right|^p\right)^{{1}/{p}},$$where the infimum is taken over all above decompositions of $f$.
\end{enumerate}
\end{theorem}
By Theorem \ref{atom cha} and Lemma \ref{molecular rela}, we know that Theorem \ref{molecular cha 1} and Theorem \ref{molecular cha 2} can be deduce from the following lemma.
\begin{lemma}
Let $0<p<\infty$, $1<\vec q<\infty$, $1-\frac{1}{v}\sum_{i=1}^n \frac{a_i}{q_i}\le\alpha<\infty$, and non-negative integer $s\ge \lfloor \frac{v}{a_-}(\alpha+\frac{1}{v}\sum_{i=1}^n \frac{a_i}{q_i}-1) \rfloor$.
Set $\epsilon>max\{s,\frac{v}{a_-}(\alpha+\sum_{i=1}^n \frac{a_i}{q_i}-1)\}$, $a=(1-\frac{1}{v}\sum_{i=1}^n \frac{a_i}{q_i})-\alpha+\epsilon$,
and $d=(1-\frac{1}{v}\sum_{i=1}^n \frac{a_i}{q_i})+\epsilon$. Then,
\begin{enumerate}
	\item [{\rm(1)}] If $0<p\le1$, then there exists a constant $C$ such that for any central $(\alpha,\vec q;s,\epsilon)$-molecule (or $(\alpha,\vec q;s,\epsilon)$-molecule of restricted type) $M$.
	$$\lVert M\rVert_{H\dot{K}_{\vec{q},\vec{a}}^{\alpha, p}\left(\mathbb{R}^n\right)}\le C(or\ \lVert M\rVert_{HK_{\vec{q},\vec{a}}^{\alpha, p}\left(\mathbb{R}^n\right)}\le C).$$
	\item [{\rm(2)}] There exists a constant $C$ such that for any $l\in\mathbb{Z}$($l\in\mathbb{N}$) and dyadic central $(\alpha,\vec q;s,\epsilon)_l$-molecule (or dyadic central $(\alpha,\vec q;s,\epsilon)_l$-molecule of restricted type) $M_l)$.
	$$\lVert M\rVert_{H\dot{K}_{\vec{q},\vec{a}}^{\alpha, p}\left(\mathbb{R}^n\right)}\le C(or\ \lVert M\rVert_{HK_{\vec{q},\vec{a}}^{\alpha, p}\left(\mathbb{R}^n\right)}\le C).$$
\end{enumerate}
\end{lemma}
\begin{proof}
we only prove (i) for the homogenous case, the proof for the non-homogeneous case and the proof of (2) are similar.\\
Suppose that $M$ is a central $(\alpha, \vec{q}; s, \varepsilon)$-molecule. Taking $r=\|M\|^{-{1}/{\alpha }} _{L^{\vec{q}}\left(\mathbb{R}^n\right)}$ and denote by $\sigma_r$ the unique integer satisfying $2^{(\sigma_r-1)v}<r\leq  2^{(\sigma_r)v}$. Denote $E_{0,\vec{a}}=B_{\sigma_r,\vec{a}}$ and $E_{k,\vec{a}}=B_{\sigma_r+k,\vec{a}}\setminus B_{\sigma_r+k-1,\vec{a}}$
for $k\in \mathbb{N}$. Set\\
$$M(x)\chi_{E_{k,\vec{a}}(x)}-\frac{\chi_{E_{k,\vec{a}}}(x)}{|E_{k,\vec{a}}|}\int_{\mathbb R^n}M(y)\chi_{E_{k,\vec{a}}}(x)dy=:M_k(x)-F_k(x).$$
Obviously
 $$M(x)=\sum_{k=0}^{\infty} \left(M_k(x)-F_k(x)\right)+\sum_{k=0}^{\infty}\frac{\chi_{E_{k,\vec{a}}}(x)}{|E_{k,\vec{a}}|}\int_{\mathbb R^n}M(y)\chi_{E_{k,\vec{a}}}(x)dy,$$\\
 then, ${\rm supp}~(M_k-F_k)\subset B_{\sigma_r+k,\vec{a}}$, and
 \begin{align*}
 	 \int_{\mathbb R^n}(M_k(x)-F_k(x))dx
 &=\int_{\mathbb R^n} M(x)\chi_{E_{k,\vec{a}}(x)}-\frac{\chi_{E_{k,\vec{a}}}(x)}{|E_{k,\vec{a}}|}\int_{\mathbb R^n}M(y)\chi_{E_{k,\vec{a}}}(x)dydx\\
 &=\int_{\mathbb R^n} M(x)\chi_{E_{k,\vec{a}}(x)}dx-\int_{\mathbb R^n}M(y)\chi_{E_{k,\vec{a}}}(x)dy=0.\\
\end{align*}
We will claim that\\
(a) there is a positive constant $C$ and a sequences of numbers $\lambda_k$ such  that $\sum_{k=0}^{\infty}|\lambda _k|^p<\infty$,
$$M_k(x)-F_k(x)={\lambda }_k a_k(x),~~~~\text{where each $a_k$ is  a $(\alpha,\vec{q},s)$-atorm.}$$
(b)$\sum_{k=0}^{\infty} F_k$ has a $(\alpha,\vec{q},s)$-atom decomposition.
\par 
We first show (a), without loss of generality, suppose that $R_{\vec{q}}(M)=1$ which implies that,\\ 
$$\left\||\cdot|_{\vec{a}}^{vd}M(\cdot)\right\|_{L^{\vec{q}}\left(\mathbb R^n\right)}=\|M\|^{\frac{-a}{d-a}}_{L^{\vec{q}}\left(\mathbb R^n\right)}=r^a.$$
	For $k=0$. The following inequalities can be directly obtained
\begin{align*}
	\left\|M_0-F_0\right\|_{L^{\vec{q}}\left(\mathbb R^n\right)} &\leq \|M_0\|_{L^{\vec{q}}\left(\mathbb R^n\right)}+\|F_0\|_{L^{\vec{q}}\left(\mathbb R^n\right)}\\
	&\leq \|M\|_{L^{\vec{q}}\left(\mathbb R^n\right)}+\left\|\frac{\chi_{E_{0,\vec{a}}}(x)}{|E_{0,\vec{a}}|}\int_{\mathbb R^n}M(y)\chi_{E_{0,\vec{a}}}(x)dy\right\|_{L^{\vec{q}}\left(\mathbb R^n\right)}\\
	&\leq\|M\|_{L^{\vec{q}}\left(\mathbb R^n\right)}+\int_{\mathbb{R}^n}|M(y){\chi_{B_{\sigma_r, \vec{a}}}(y)|dy}\frac{\|\chi_{B_{\sigma_r,{\vec{a}}}}\|_{L^{\vec{q}}\left(\mathbb R^n\right)}}{|B_{\sigma_r,{\vec{a}}}|}\\
	&\leq\|M\|_{L^{\vec{q}}\left(\mathbb R^n\right)}+\||M\|_{L^{\vec{q}}\left(\mathbb R^n\right)}{\|{\chi_{B_{\sigma_r,{\vec{a}}}}}}\|_{L^{\vec{q}^{\prime}}\left(\mathbb R^n\right)}{\|{\chi_{B_{\sigma_r,{\vec{a}}}}}}\|_{L^{\vec{q}}\left(\mathbb R^n\right)}|B_{\sigma_r,{\vec{a}}}|\\
	&\lesssim 2\|M\|_{L^{\vec{q}}\left(\mathbb R^n\right)}\leq 2|B_{\sigma_r,\vec{a}}|^{-\alpha},	
\end{align*}
and for $k \in \mathbb{N}$
\begin{align*}
\|M_k-F_k\|_{L^{{\vec{q}}}\left(\mathbb R^n\right)}&\leq\|M_k\|_{L^{{\vec{q}}}\left(\mathbb R^n\right)}+\|F_k\|_{L^{{\vec{q}}}\left(\mathbb R^n\right)}\\
&\leq C\|M_k\|_{L^{\vec{q}}(\mathbb R^n)}+\frac{C}{\left|E_{k,\vec{a}}\right|} \|M_k\|_{L^{\vec{q}}(\mathbb R^n)}\left\|\chi_{E_{k,\vec{a}}}\right\|_{L^{\vec{q}^{\prime}}\left(\mathbb R^n\right)}\left\|\chi_{E_{k,\vec{a}}}\right\|_{L^{\vec{q}}\left(\mathbb R^n\right)}\\
& \leq C\|M_k\|_{L^{\vec{q}}(\mathbb R^n)}\\
&\leq C\left\||\cdot|^{vd}_{\vec{a}}M(\cdot)\right\|_{L^{\vec{q}}\left(\mathbb R^n\right)} 2^{-{(\sigma_r+k-1)}vd}\\
&\leq C r^a 2^{-{(\sigma_r+k)}vd}\\
&\leq C 2^{\sigma_r va} 2^{-(\sigma_r+k)v(d-a)} 2^{-av(\sigma_r+k)}\\
&\leq C 2^{-vak}|B_{\sigma_r+k,\vec{a}}|^{-\alpha }.
\end{align*}
Thus, for any $k\in \{\mathbb{N}\cup (0)\}$, there is a constant $C$ independent of $k$ such that
$$\|M_k-F_k\|_{L^{\vec{q}}\left(\mathbb R^n\right)}\leq C 2^{-kva} |B_{(\sigma _r+k),\vec{a}}|^{-\alpha}.$$
If we denote\\
$$\sum_{k=0}^{\infty}M_k-F_k=C\sum_{k=0}^{\infty} 2^{-kva}\frac{M_k-F_k}{C 2^{-kva}}=\lambda _{1,k}a_{1,k}.$$
Therefore\\
$$\|a_{1,k}\|_{L^{\vec{q}}\left(\mathbb R^n\right)}=\frac{1}{C 2^{-kva}}\|M_k-F_k\|_{L^{\vec{q}}\left(\mathbb R^n\right)}
\leq|B_{\sigma_{r+k},\vec{a}}|^{-\alpha },$$
and 
$$\sum_{k=0}^{\infty}|\lambda _{1,k}|^p=\sum_{k=0}^{\infty}C 2^{-kavp}
\leq C,$$
where $C$ is independent of $M$.\\
Let
$$m_k=\sum_{i=k}^{\infty}\int_{\mathbb{R}^n}M(x)\chi_{E_i}(x)dx,~~~~\psi_k(x)=\frac{\chi_{E_k}(x)}{|E_k|}.$$
Noting that $m_0=0$. We have
\begin{align*}
	\sum_{k=0}^{\infty}F_k(x)=\sum_{k=0}^{\infty}(m_k-m_{k+1})\psi _k(x)=\sum_{k=0}^{\infty}m_{k+1}\left(\psi _{k+1}(x)-\psi_{k}(x)\right).
\end{align*}
Obviously,
$$\int_{\mathbb R^n}m_{k+1}\left(\psi_{k+1}(x)-\psi_k(x)\right)dx=0,~~~~{\rm supp}~\{m_{k+1}\left(\psi_{k+1}(x)-\psi_k(x)\right)\} \subset B_{\sigma_r+k+1,\vec{a}}.$$
Hence, via a simple calculation and splitting of the out part of the ball $B_{\sigma_r+k-1, \vec{a}}$, the following results can get
\begin{align*}
\left\|m_{k+1}\left(\psi_{k+1}-\psi_k \right)\right\|_{L^{\vec{q}}\left(\mathbb R^n \right)}&=\left\|\sum_{i=k+1}^{\infty}\int_{\mathbb R^n} M(x)\chi_{E_i}(x)dx  \left(\psi_{k+1}-\psi_k \right)\right\|_{L^{\vec{q}}\left(\mathbb R^n \right)}\\
	&=\left\|\psi_{k+1}-\psi_k \right\|_{L^{\vec{q}}\left(\mathbb R^n\right)}\sum_{i=k}^{\infty}\int_{\mathbb{R}^n}M(x)\chi_{E_i}(x)dx\\
	&\leq \left(\|\psi_{k+1}\|_{L^{\vec{q}}\left(\mathbb R^n\right)}+\|\psi_k\|_{L^{\vec{q}}\left(\mathbb R^n\right)}\right)\sum_{i=k}^{\infty}\int_{\mathbb{R}^n}M(x)\chi_{E_i}(x)dx\\
	&=\left(\frac{\left\|\chi_{E_{k+1}, \vec{a}}\right\|_{L^{\vec{q}}(\mathbb R^n)}}{\left|E_{{k+1}, \vec{a}}\right|}+\frac{\left\|\chi_{E_{k}, \vec{a}}\right\|_{L^{\vec{q}}\left(\mathbb R^n\right)}}{\left|E_{{k}, \vec{a}}\right|}\right)  \sum_{i=k}^{\infty} \int_{\mathbb{R}^{n}} M(x) X_{E_{i}}(x) d x\\
	&\lesssim \left(\frac{\|\chi_{E_{k,\vec{a}}}\|_{L^{\vec{q}}\left(\mathbb R^n\right)}}{|E_{k+1,\vec{a}}|}  \frac{\|\chi_{B_{\sigma_r+k+1,\vec{a}}}\|_{L^{\vec{q}}\left(\mathbb R^n\right)}}{\|\chi_{B_{\sigma_r+k,\vec{a}}}\|_{L^{\vec{q}}\left(\mathbb R^n\right)}}+ 2^v \frac{\|\chi_{E_{k,\vec{a}}}\|_{L^{\vec{q}}\left(\mathbb R^n\right)}}{|E_{k+1,\vec{a}}|}\right) \times \int_{B_{\sigma_r+k-1,\vec{a}}^{c}}|M(x)|dx\\
	&\lesssim \left(\frac{\|\chi_{E_{k,\vec{a}}}\|_{L^{\vec{q}}\left(\mathbb R^n\right)}}{|E_{k+1,\vec{a}}|}  \frac{\|\chi_{B_{\sigma_r+k+1,\vec{a}}}\|_{L^{\vec{q}}\left(\mathbb R^n\right)}}{\|\chi_{B_{\sigma_r+k,\vec{a}}}\|_{L^{\vec{q}}\left(\mathbb R^n\right)}}+ 2^v \frac{\|\chi_{E_{k,\vec{a}}}\|_{L^{\vec{q}}\left(\mathbb R^n\right)}}{|E_{k+1,\vec{a}}|}\right)\\
	&\quad \times \left \||\cdot|_{\vec{a}}^{v d} M \right\|_{L^{\vec{q}}(\mathbb R^n)} \left\|\sum_{\sigma_r+k}^{\infty} |\cdot|_{\vec{a}}^{-vd} \chi_{B_{{i+1,\vec{a}}}}\right\|_{L^{\vec{q}^{\prime}}(\mathbb R^n)}\\
	%&\leq C \frac{2^{r}}{\left(2^{v}-1\right) / 2^{r}}\left\|X_{B a r+k} \cdot \vec{a}\right\|_{L^{\vec{q}}} \| r^{a} \cdot 2^{-(a r+k-1) v d} \cdot\left|E_{A r+k+1,} \vec{a}\right|^{-1}\\
	&\leq C 2^{-k v a}|B_{\sigma_r+k+1, \vec{a}}|^{-\alpha}.
\end{align*}
Let 
\begin{align*}
&\quad\sum_{k=0}^{\infty} \frac{\chi_{E_{k,\vec{a}}}(x)} {|E_{k,\vec{a}}|} \int_{\mathbb{R}^{n}} M(y)\chi_{E_{k,\vec{a}}}(y)dy\\
&=\sum_{k=0}^{\infty}C2^{-kva}\frac{m_{k+1}\left(\psi_{k+1}-\psi_k \right)}{C 2{-kva}}\\
&=\sum_{k=0}^{\infty}\lambda_{2,k} a_{2,k}.
\end{align*}
We can easily show that	
$$
\left\|\sum\limits_{k=0}^{\infty}\frac{\chi_{E_{k,\vec{a}}}}{|E_{k,\vec{a}}|}\int_{\mathbb{R}^n} M(y)\chi_{E_{k}}(y)\mathrm{d}y \right \|_{L^{\vec{q}}(\mathbb R^n)}\leq|B_{\sigma_r+k+1, \vec{a}}|^{-\alpha}.
$$
and $$\sum\limits_{k=0}^{\infty}|\lambda_{2,k}|^{p}\leq C2^{-vkap}<C,$$
where $C$ is independent of $M$.
\end{proof}

\subsection{Some Application}
The following condition is necessary for our discussion on the linear $T$ on anisotropic mixed-norm Herz-Hardy space.
\begin{equation}
	Tf=\sum_{i\in \mathbb N}\lambda_iTa_i~~\text{in the distribution sense, if }~~f=\sum_{i\in \mathbb N} \lambda_ia_i~~\text{in the distribution sense.} \label{Tf}
\end{equation}
\begin{theorem}
	Let $0<p<\infty,~1\leq \vec{q} <\infty,~1-\frac{1}{v}\sum_{i=1}^n \frac{a_{i}}{q_{i}}\leq\alpha<\frac{a_{+}-\sum_{i=1}^n \frac{a_{i}}{q_{i}}}{v}+1$.
If a linear operator T satisfies (\ref{Tf}) for every central atomic decomposition is bounded on $L^{\vec{q}}\left(\mathbb R^n\right)$, and for any 
$f\in L^{\vec{q}}\left(\mathbb R^n\right)$ with compact $B_{k,\vec{a}}$, and $\int_{\mathbb{R}^{n}}f(x)dx=0$. $T$ satisfies the following size condition:
$$\left|Tf(x)\right|\leq\frac{C2^{kv}\|f\|_{L^{1}\left(\mathbb R^n\right)}}{|x|^{2}_{\vec{a}}},~if ~~\inf\limits_{y\in {\rm supp} f}|x-y|_{\vec{a}}\geqslant\frac{|x|_{\vec{a}}}{2}.$$
Then, operator $T$ maps $H\dot{K}^{\alpha,p}_{\vec{q},\vec{a}}(\mathbb{R}^n)$ to $\dot{K}^{\alpha,p}_{\vec{q},\vec{a}}(\mathbb{R}^n)$
(or~operator $T$ maps $HK^{\alpha,p}_{\vec{q},\vec{a}}(\mathbb{R}^n)$ to $K^{\alpha,p}_{\vec{q},\vec{a}}(\mathbb{R}^n)).$
\end{theorem}
\begin{proof}
	We only prove the homogeneous case. Suppose $f\in H\dot{K}^{\alpha,p}_{\vec{q},\vec{a}}(\mathbb{R}^n)$, then $f=\sum\limits_{j=-\infty}^{\infty}\lambda_{j}a_{j}$ in the distribution sense.
where  $a_{j}$  is a central  $(\alpha, \vec{q}, s) $-atom with support  contained in  $B_{j,\vec{a}}$ , and
$$\|f \| _{ H \dot{K}^{\alpha,p}_{\vec{q},  \vec{a}}\left(\mathbb R^n\right)} \sim \inf \left(\sum_{j=-\infty}^{\infty}\left|\lambda_{j}\right|^{p}\right)^{\frac{1}{p}}$$ 
where the infimum is taken over all above decomposition
of $f$. \\
By condition $(\ref{Tf})$. We write
\begin{align*}
	\|T f\|_{\dot{K}_{\vec{q}, \vec{a}}^{\alpha,p}(\mathbb R^n)}^{p} &=\sum_{k=-\infty}^{\infty}\left|B_{k , \vec{a}}\right|^{\alpha}\left\|(Tf) \chi_{k}\right\|_{L^{\vec{q}}\left(\mathbb R^n\right)}^{p} \\	
	&=\sum_{k=-\infty}^{\infty}\left|B_{k , \vec{a}}\right|^{\alpha}\left\|T\left(\sum_{j=-\infty}^{\infty} \lambda_{j} a_{j}\right) \chi_{k}\right\|_{L^{\vec{q}}\left(\mathbb R^n\right)}^p\\
	&\leq \sum_{k=-\infty}^{\infty}\left|B_{k , \vec{a}}\right|^{\alpha} \left(\sum_{j=-\infty}^{k-2} \mid \lambda_{j}|\|\left(Ta_{j}\right) \chi_{k} \| _{L^{\vec{q}}\left(\mathbb R^n\right)}\right)^{p}\\
	&+\sum_{k=-\infty}^{\infty}\left|B_{k , \vec{a}}\right|^{\alpha} \left(\sum_{j=k-1}^{\infty}\left|\lambda_{j}\right|\left\|\left(Ta_{j}\right) \chi_{k}\right\|_{L^{\vec{q}}\left(\mathbb R^n\right)}\right)^{p}\\
	&= C\left(I_{1}+I_{2}\right).
\end{align*}
Let us first estimate $I_1$. If $ j \leqslant k-2$, $x \in A_{k} $, and $ y \in B_{j,\vec{a}} $, then
$$
	|x-y|_{\vec{a} } \geqslant|x|_{\vec{a} }-|y|_{\vec{a} } \geqslant \frac{|x| _{\vec{a} }}{2} ,
$$
Furthermore, by size condition of $a_j$, one can get 
\begin{align*}
	\left|T a_{j}(x)\right| & \leq \frac{C 2^{j v}\left\|a_{j}\right\|_{L_1\left(\mathbb R^n\right)}}{\mid x|_{\vec{a}}^{2 v}} \leq C 2^{j v}  2^{-2(k-1) v}\left\|a_{j}\right\|_{L_ {\vec{q}}\left(\mathbb R^n\right)} \| \chi_{B_{j}} \|_{L_{\vec{q}'}\left(\mathbb R^n\right)}  \\
	& \leq C 2^{j v}  2^{-2 v(k-1)} \left|B_{j, \vec{a}}\right|^{-\alpha}  2^{j\sum\limits_{i=1}^n \left(1-\frac{1}{q_{i}}\right) a_{i}} \\
	&=C 2^{j v} 2^{-2 v(k-1)} 2^{-j v \alpha} 2^{j v} 2^{-j \sum\limits_{i=1}^n \frac{a_{i}}{q_{i}}}.
\end{align*}
Then
\begin{align*}
	\left\|\left(T a_{j}\right) \chi_{k} \right\|_{L^{\vec{q}}\left(\mathbb R^n\right)}
	&\leq C 2^{j v} 2^{-2 v(k-1)}  2^{-j v \alpha}  2^{j v} 2^{-j \sum\limits_{i=1}^n \frac{a_{i}}{q_i}} 2^{k \sum\limits_{i=1}^n \frac{a_{i}}{q_{i}}}\\
	&\leq 2^{v(2+j-k)} 2^{(j-k) v\left(1-\frac{1}{v} \sum\limits_{i=1}^n \frac{a_{i}}{q_{i}}\right)-j v \alpha }.
\end{align*}
Similarly to the estimation of $ I_{2} $. we consider two cases. 

When $0<p \leqslant 1 $.
\begin{align*}
	I_{1} &=\sum_{k=-\infty}^{\infty}\left|B_{k, \vec{a}}\right|^{\alpha} \left(\sum_{j=-\infty}^{k-2}\left|\lambda_{j}\right|\left\|\left(T a_{j}\right) \chi_{k}\right\|_{L^{\vec{q}}\left(\mathbb R^n\right)}\right)^{p} \\
	I_{1} &=\sum_{k=-\infty}^{\infty} 2^{k v \alpha}\left(\sum_{j=-\infty}^{k-2} |\lambda_{j} | \cdot 2^{v(2+j-v)}  2^{(j-k) v\left(1-\frac{1}{v} \sum\limits_{i=1}^n \frac{a_{i}}{q_{i}}\right)-j u \alpha}\right)^{p} \\
	& \leq \sum_{k=-\infty}^{\infty} 2^{k v \alpha} \sum_{j=-\infty}^{k-2}\left|\lambda_{j}\right|^{p} 2^{v p(2+j-v)}  2^{(j-k) vp\left(1-\frac{1}{v} \sum\limits_{i=1}^n \frac{a_{i}}{q_{i}}\right)-j u \alpha p} \\
	& \leq C \sum_{j=-\infty}^{\infty}\left|\lambda_{j}\right|^{p}\left(\sum_{k=j+2}^{\infty} 2^{(j-k) v p\left(1-\frac{1}{v} \sum\limits_{i=1}^n \frac{a_{i}}{q_{i}}-\alpha+1\right) p}\right)\\
	& \leq C \sum_{j=-\infty}^{\infty}\left|\lambda_{j}\right|^{p }.
\end{align*}

When $1<p<\infty$, by H\"older's inequality, it follows that
\begin{align*}
I_{1} &\leq C \sum_{k=-\infty}^{\infty} \left|B_{k, \vec{a}}\right|\left(\sum_{j=-\infty}^{k-2} |\lambda_{j} | 2^{v(2+j-v)} 2^{(j-k) v\left(1-\frac{1}{v} \sum\limits_{i=1}^n \frac{a_{i}}{q_{i}}\right)-j u \alpha}\right)^{p}  \\
&\leq C\sum_{k=-\infty}^{\infty} \left(\sum_{j=-\infty}^{k-2} |\lambda_{j} |  2^{v(j-k) \left(1-\frac{1}{v} \sum\limits_{i=1}^n \frac{a_{i}}{q_{i}} - \alpha +1\right)}\right)^{p}\\
&\leq C \sum_{j=-\infty}^{\infty}\left(\sum_{k=j+2}^{\infty}\left|\lambda_{j}\right|^{p} 2^{\left[(j-k)\left(1-\frac{1}{v} \sum\limits_{i=1}^n \frac{a_{i}}{q_{i}}-\alpha+1\right) p\right] / 2}\right)\left( \sum_{k=j+2}^{\infty} 2^{\left[(j-k)\left(1-\frac{1}{v} \sum\limits_{i=1}^n \frac{a_{i}}{q_{i}}-\alpha+1\right) p^{\prime}\right] / 2} \right)^{p / p^{\prime}}\\
&\leq C \sum_{j=-\infty}^{\infty}\left(\sum_{k=j+2}^{\infty}\left|\lambda_{j}\right|^{p} 2^{\left[(j-k)\left(1-\frac{1}{v} \sum\limits_{i=1}^n \frac{a_{i}}{q_{i}}-\alpha+1\right) p\right] / 2}\right) \leq C \sum_{j=-\infty}^{\infty}\left|\lambda_{j}\right|^{p}.
\end{align*}
Let us now estimate $I_{2}$, when $0<p \leq 1 $
\begin{align*}
	I_{2} &=\sum_{k=-\infty}^{\infty}\left|B_{k , \vec{a}}\right|^{\alpha} \left(\sum_{j=k-1}^{\infty}\left|\lambda_{j}\right|\left\|\left(Ta_{j}\right) \chi_{k}\right\|_{L^{\vec{q}}\left(\mathbb R^n\right)}\right)^{p} \\
	& \leq \sum_{k=-\infty}^{\infty}\left|B_{k , \vec{a}}\right|^{\alpha}\left(\sum_{j=k-1}^{\infty}\left|\lambda_{j}\right|^{p} \cdot\left\|\left(T a_{j}\right) \chi_{k}\right\|_{L^{\vec{q}}\left(\mathbb R^n\right)}^p\right) \\
	& \leq C \sum_{k=\infty}^{\infty}\left|B_{k , \vec{a}}\right|^{\alpha}\left(\sum_{j=k-1}^{\infty}\left|\lambda_{j}\right|^{p}\left\|a_{j}\right\|_{L^{\vec{q}}\left(\mathbb R^n\right)}^{p}\right) \\
	&\leq C \sum_{j=-\infty}^{\infty}\left|\lambda_{j}\right|^{p} \left( \sum_{k=-\infty}^{j=k+1} 2^{v(k-j) \alpha p} \right) \leq C \sum_{j=-\infty}^{\infty}\left|\lambda_{j}\right|^{p}.
\end{align*}

If  $1<p<\infty$, by H\"older's inequality and the boundedness of  $T$, 
\begin{align*}
I_{2} &\leq C \sum_{k=-\infty}^{\infty}\left|B_{k,  \vec{a}}\right|^{\alpha}\left(\sum_{j=k-1}^{\infty}\left|\lambda_{j}\right|^{p}\left\|a_{j}\right\|_{L^{\vec{q}}\left(\mathbb R^n\right)} ^{{p}/{2}}\right) \left(\sum_{j=k-1}^{\infty}\left\|a_{j}\right\|_{L^{\vec{q}}\left(\mathbb R^n\right)}^{p^{\prime} / 2}\right)^{{p}/{p^{\prime}}} \\
&\leq C \sum_{k=-\infty}^{\infty}\left|B_{k, \vec{a}}\right|^{\alpha}\left(\sum_{j=k-w}^{\infty}\left|\lambda_{j}\right|^{p} 2^{-j v \alpha p / 2}\right)\left(\sum_{j=k-1}^{\infty} 2^{-j v \alpha p^{\prime} / 2}\right)^{\frac{p}{p^{\prime}}} \\
&\leq C \sum_{k=-\infty}^{\infty} 2^{k v \alpha p / 2}\left(\sum_{j=k-1}^{\infty}\left|\lambda_{j}\right|^{p} 2^{-j v \alpha p / 2}\right) \\
&\leq C \sum_{j=-\infty}^{\infty}\left|\lambda_{j}\right|^{p} \left( \sum_{k=-\infty}^{j+1} 2^{v(k-j) \alpha p/ 2} \right) \leq C \sum_{j=-\infty}^{\infty}\left|\lambda_{j}\right|^{p}
\end{align*}
Thus, the result of this Theorem can be concluded
$$
\|T f\|  _{\dot{K}^{\alpha,p}_{\vec{q}, \vec{a}}\left(\mathbb R^n\right)} \leq C\|f\| _{ H \dot{K}^{\alpha,p}_{\vec{q},  \vec{a}}\left(\mathbb R^n\right)} .
$$
This finishes the proof.
\end{proof}

%\hspace*{-0.6cm}\textbf{\bf Competing interests}\\
%The authors declare that they have no competing interests.\\

\hspace*{-0.6cm}\textbf{\bf Funding}\\
The research was supported by Natural Science Foundation of China (Grant No. 12061069).\\

%\hspace*{-0.6cm}\textbf{\bf Authors contributions}\\
%All authors contributed equality and significantly in writing this paper. All authors read and approved the final manuscript.\\

\hspace*{-0.6cm}\textbf{\bf Acknowledgments}\\
The authors would like to express their thanks to the referees for valuable advice regarding the previous version of this paper.\\

\hspace*{-0.6cm}\textbf{\bf Authors detaials}\\
Yichun Zhao and Jiang Zhou*, zhaoyichun@stu.xju.edu.cn and zhoujiang@xju.edu.cn, College of Mathematics and System Science, Xinjiang University, Urumqi, 830046, P.R China.\\


\bigskip \medskip


\begin{thebibliography}{10}

\bibitem{beurling1964construction}
A.~Beurling.
\newblock Construction and analysis of some convolution algebras.
\newblock {\em Annales del'institut Fourier}, 14:1--32, 1964.

\bibitem{herz1968lipschitz}
C.S.~Herz.
\newblock Lipschitz spaces and {B}ernstein's theorem on absolutely convergent
  {F}ourier transforms.
\newblock {\em Journal of Mathematics and Mechanics}, 18(4):283--323, 1968.

\bibitem{baernstein1985embedding}
A.~Baernstein and E.T.~Sawyer.
\newblock {\em Embedding and Multiplier Theorems for $H^p\left( \mathbb R^n\right)$}, volume 318.
\newblock American Mathematical Society, 1985.

\bibitem{feichtinger1984elementary}
H.G.~Feichtinger.
\newblock An elementary approach to Wiener's third Tauberian Theorem for the euclidean n-space.
\newblock {\em Universit{\"a}t Wien. Mathematisches Institut}, 1984.

\bibitem{lu2008herz}
S.~Lu, D.~Yang and G.~Hu.
\newblock {\em Herz type spaces and their applications},
\newblock Science Press Beijing, 2008.

\bibitem{scapellato2019regularity}
A.~Scapellato.
\newblock Regularity of solutions to elliptic equations on Herz spaces with variable exponents.
\newblock {\em Boundary Value Problems}, (1):1--9, 2019.

\bibitem{chen2018global}
J.~Chen and C.~Song.
\newblock Global stability for the fractional Navier-Stokes equations in the Fourier-Herz space.
\newblock {\em Mathematical Methods in the Applied Sciences}, 41(10):3696--3717, 2018.

\bibitem{chen1989some}
Y.Z.~Chen and K.S.~Lau.
\newblock Some new classes of Hardy spaces.
\newblock {\em Journal of functional analysis}, 84(2):255--278, 1989.

\bibitem{garcia1994theory}
J.~Garc{\'\i}a-Cuerva and M.J.L.~Herrero.
\newblock A theory of Hardy spaces associated to the Herz spaces.
\newblock {\em Proceedings of the London Mathematical Society}, 3(3):605--628, 1994.

\bibitem{lu1995local}
S.~Lu and D.~Yang.
\newblock The local versions of $H^p\left( \mathbb R^n\right)$ spaces at the origin.
\newblock {\em Studia mathematica}, 2(116):103--131, 1995.

\bibitem{zhao2018anisotropic}
H.~Zhao and J.~Zhou.
\newblock Anisotropic {H}erz-type {H}ardy spaces with variable exponent and
  their applications.
\newblock {\em Acta Mathematica Hungarica}, 156(2):309--335, 2018.

\bibitem{drihem2013embeddings}
D.~Drihem.
\newblock Embeddings properties on Herz-type Besov and Triebel-Lizorkin spaces.
\newblock {\em Mathematical Inequality and Application}, 16(2):439--460, 2013.

\bibitem{2008New}
Y.~Ding, S.~Lan and S.~Lu.
\newblock New Hardy Spaces Associated with Some Anisotropic Herz Spaces and Their Applications.
\newblock {\em Acta Mathematica Sinica, English Series}, 24(9):1449-1470, 2008.

\bibitem{zhao2022characterizations}
Y.~Zhao, M.~Wei and J.~Zhou.
\newblock Characterizations of Mixed Herz-Hardy Spaces and their Applications.
\newblock {\em arXiv preprint arXiv:2205.10372}, 2022.

\bibitem{benedek1961space}
A.~Benedek and R.~Panzone.
\newblock The space ${L}_p$, with mixed norm.
\newblock {\em Duke Mathematical Journal}, 28(3):301--324, 1961.

\bibitem{cleanthous2017anisotropic}
G.~Cleanthous, A.G. Georgiadis, and M.~Nielsen.
\newblock Anisotropic mixed-norm {H}ardy spaces.
\newblock {\em The Journal of Geometric Analysis}, 27(4):2758--2787, 2017.

\bibitem{huang2019atomic}
L.~Huang, J.~Liu, D.C.~Yang, and W.~Yuan.
\newblock Atomic and {L}ittlewood--{P}aley characterizations of anisotropic
  mixed-norm Hardy spaces and their applications.
\newblock {\em The Journal of Geometric Analysis}, 29(3):1991--2067, 2019.

\bibitem{wei2021characterization}
M.~Wei.
\newblock A characterization of via the commutator of {H}ardy-type operators on
  mixed {H}erz spaces.
\newblock {\em Applicable Analysis}, 1--16, 2021.

\bibitem{wei2020Herz}
M.~Wei.
\newblock Extrapolation to mixed {H}erz spaces and its applications.
\newblock {\em submitted}.

\bibitem{calderon1977atomic}
A.P.~Calder{\'o}n.
\newblock An atomic decomposition of distributions in parabolic $H^p\left( \mathbb R^n\right)$ spaces.
\newblock {\em Advances in Mathematics}, 25(3):216--225, 1977.

\bibitem{bownik2003anisotropic}
M.~Bownik.
\newblock {\em Anisotropic Hardy spaces and wavelets}.
\newblock American Mathematical Society, 2003.

\bibitem{bownik2005atomic}
M.~Bownik.
\newblock Atomic and molecular decompositions of anisotropic Besov spaces.
\newblock {\em Mathematische Zeitschrift}, 250(3):539--571, 2005.

\bibitem{hochmuth2002wavelet}
R.~Hochmuth.
\newblock Wavelet characterizations for anisotropic Besov spaces.
\newblock {\em Applied and Computational Harmonic Analysis}, 12(2):179--208, 2002.

\bibitem{fabes1966singular}
E.~Fabes and N.~Riviere.
\newblock Singular integrals with mixed homogeneity.
\newblock {\em Studia Mathematica}, 1(27):19--38, 1966.

\bibitem{stein1978problems}
E.M.~Stein and S.~Wainger.
\newblock Problems in harmonic analysis related to curvature.
\newblock {\em Bulletin of the American Mathematical Society}, 84(6):1239--1295, 1978.

\bibitem{coifman1976factorization}
R.R.~Coifman, R.~Rochberg and G.~Weiss.
\newblock Factorization theorems for Hardy spaces in several variables.
\newblock {\em Annals of Mathematics}, 103(3):611--635, 1976.

\bibitem{grafakos2008classical}
L.~Grafakos.
\newblock {\em Classical Fourier Analysis}, volume 2.
\newblock Springer, 2008.

\bibitem{muckenhoupt}
B.~Muckenhoupt and R.L.~Wheeden.
\newblock Norm inequalities for the Littlewood-Paley function.
\newblock {\em Transactions of the American Mathematical Society}, 191:95--111, 1974.

\bibitem{li1996boundedness}
X.~Li and D.~Yang.
\newblock Boundedness of some sublinear operators on Herz spaces.
\newblock {\em Illinois Journal of Mathematics}, 40(3):484--501, 1996.

\bibitem{ward2010new}
E.L.~Ward.
\newblock New estimates in {H}armonic {A}nalysis for mixed {L}ebesgue spaces.
\newblock {\em University of Kansas}, 2010.

\bibitem{torchinsky123real}
A.~Torchinsky.
\newblock The real variable methods in Harmonic Analysis.
\newblock {\em Pure Application Math}, 123, 1986.

\bibitem{wilson1988note}
J.M.~Wilson.
\newblock A note on the $g$-function.
\newblock {\em Proceedings of the American Mathematical Society}, 102(2):381-382, 1988.

\end{thebibliography}
\end{document}